\documentclass[11pt,a4paper]{amsart}
\usepackage[a4paper,margin=25mm]{geometry}
\usepackage[utf8]{inputenc}
\usepackage[expansion=false]{microtype}
\usepackage{amsmath,amssymb,amsthm,mathtools}
\usepackage{enumitem}
\usepackage{xcolor}
\usepackage[colorlinks=true,linkcolor=blue,citecolor=blue,urlcolor=blue]{hyperref}
\hypersetup{pdftitle={Iterates of Ritt operators close to the identity},pdfauthor={Catalin Badea}}

\allowdisplaybreaks

\theoremstyle{plain}
\newtheorem{theorem}{Theorem}[section]
\newtheorem{proposition}[theorem]{Proposition}
\newtheorem{lemma}[theorem]{Lemma}
\newtheorem{corollary}[theorem]{Corollary}
\newtheorem*{theoremA}{Theorem A}
\newtheorem*{theoremB}{Theorem B}
\newtheorem*{theoremC}{Theorem C}
\newtheorem*{theoremD}{Theorem D}
\newtheorem*{theoremE}{Theorem E}
\theoremstyle{definition}
\newtheorem{definition}[theorem]{Definition}
\newtheorem{example}[theorem]{Example}

\theoremstyle{remark}
\newtheorem{remark}[theorem]{Remark}
\newtheorem*{remarks}{Remarks}

\numberwithin{equation}{section}

\newcommand{\D}{\mathbb D}
\newcommand{\T}{\mathbb T}
\newcommand{\C}{\mathbb C}
\newcommand{\N}{\mathbb N}
\newcommand{\Z}{\mathbb Z}

\newcommand{\B}{\mathcal B}
\newcommand{\CR}{C_{\mathrm R}}
\newcommand{\dist}{\operatorname{dist}}
\newcommand{\Id}{I}
\newcommand{\ol}{\overline}
\newcommand{\bn}{\mathbf n}
\newcommand{\Ritt}[1]{\mathrm{Ritt}_{#1}}

\newcommand{\Res}{R}
\newcommand{\od}{\Delta}            
\newcommand{\arc}{\Gamma}           
\newcommand{\zR}{Z}                 

\title[Iterates of Ritt operators close to the identity]{Iterates of Ritt operators\\ close to the identity}
\author[C. Badea]{Catalin Badea}
\address{Department of Mathematics and Statistics, University of Reading,
Reading, United Kingdom and 
UMR 8524 -- Laboratoire Paul Painlev\'e, Universit\'e de Lille,
Lille, France}
\email{c.badea@reading.ac.uk ; cbadea@univ-lille.fr}
\date{22 September 2026}

\begin{document}

\begin{abstract}
We study bounded operators whose prescribed powers remain uniformly close to the
identity, and resolvent conditions sampled along sequences of powers.
The displacement bound $\sup_{n\ge1}\|I-T^n\|\le q<2$ forces a finite peripheral
spectrum consisting of odd-order roots of unity and makes an explicitly
determined odd power of $T$ a Ritt operator. The sharp unconditional
threshold for $T$ itself is $\sqrt3$. If $\Phi_*(q)$ denotes the optimal
universal Ritt resolvent bound below this threshold, then
$\Phi_*(q)=2/(\sqrt3-q)+O(1)$ as $q\uparrow\sqrt3$.
At every higher peripheral threshold the optimal finite-peripheral
resolvent bound has reciprocal order of growth. We also obtain constructive
bounds for general weighted Wiener symbols.
For each $1<q<2$, an angular escape invariant characterises exactly the
sampling sequences for which a displacement bound by $q$, together with
the peripheral spectral condition, forces the Ritt property. We compute
this invariant for asymptotically geometric sequences. At the endpoint
$q=1$, suitable phase conditions imply the full displacement bound without
assuming power boundedness; normal contractions admit an exact scalar
criterion. Finally, we establish quantitative Hilbert-space and $L^p$
resolvent estimates, a characterisation of operators with odd Ritt powers on uniformly
convex spaces, and strict norm gaps for operator-valued disc-algebra
functions, with explicit matrix-polynomial bounds from Fej\'er--Riesz
factorisation.
\end{abstract}

\maketitle
\setcounter{tocdepth}{1}
\tableofcontents

\section{Introduction}\label{sec:intro}

Let $X$ be a complex Banach space and $T\in\mathcal B(X)$. The bounded linear operator $T$ is called a \emph{Ritt operator} if its spectrum $\sigma(T)$ is contained in the closed unit disc $\overline{\mathbb D}$ and its Ritt constant is finite:
\begin{equation}\label{eq:rittdef}
\CR(T):=\sup_{|\lambda|>1}|\lambda-1|\,\|(\lambda-T)^{-1}\|<\infty .
\end{equation}
Ritt operators are the discrete-time counterparts of 
sectorially bounded holomorphic $C_0$-semigroups.
It is known that $T$ is a Ritt operator if and only if $T$ is power bounded and
$$
\sup_{n\ge1} n\|T^n-T^{n-1}\|<\infty;
$$
see \cite{Lyubich,NagyZemanek,Nevanlinna}. Equivalently, $T$ is Ritt if and only if
$$
\sigma(T)\subset\mathbb D\cup\{1\}
$$
and $\Id-T$ is sectorial of angle strictly less than $\pi/2$.

Ritt operators also appear under the name \emph{analytic operators} in \cite{CSC}, apparently independently of the operator-theoretic terminology. We refer to the monograph \cite{LeMerdy} for a comprehensive account of the theory.

The classical Beurling--Kato theorem \cite{Beurling,Kato} says that a $C_0$-semigroup $(e^{-tA})_{t\ge0}$ is holomorphic as soon as $\limsup_{t\downarrow0}\|\Id-e^{-tA}\|<2$. Borichev, Gomilko and Tomilov \cite{BGT} recently developed a discrete Beurling--Kato theory for the powers $(T^n)_{n\ge0}$ of a Ritt operator. Two of their results are the starting points of the present paper. The first is a \emph{discrete Kato criterion}: if $T$ is power bounded, $\sigma(T)\cap\T\subset\{1\}$, and
\begin{equation}\label{eq:onepoint}
  \sup_{n\ge1}\|(\zeta-T^n)^{-1}\|<\infty\qquad\text{for one point }\zeta\in\T\setminus\{1\},
\end{equation}
then $T$ is a Ritt operator; conversely, if $T$ is Ritt, then \eqref{eq:onepoint} holds for every $\zeta\in\T\setminus\{1\}$ \cite[Corollary~7.8]{BGT}. Applied with $\zeta=-1$ and using Neumann's series for the inverse, this gives the discrete zero--two law: if $\sigma(T)\subset\D\cup\{1\}$ and $\sup_{n\ge1}\|\Id-T^n\|<2$, then $T$ is Ritt \cite[Corollary~6.1]{BGT}.
This zero--two law has a converse for contractions acting on uniformly convex spaces: if $X$ is uniformly convex and $T$ is a contractive Ritt operator, then $\sup_{n\ge1}\|\Id-T^n\|<2$ \cite[Theorem~7.9]{BGT}. This is a discrete counterpart of a result due to Pazy for
strongly continuous semigroups on Banach spaces, see \cite[Ch. 2.5, Cor. 5.8]{Pazy}. Borichev, Gomilko and Tomilov also show that \eqref{eq:onepoint} cannot in general be tested along a subsequence $(n_k)$ with $\sup_kn_k/n_{k-1}=\infty$; see \cite[Proposition~6.8]{BGT}. This raises the question of which sampling
sequences preserve the Ritt criterion\footnote{This question has been explicitely asked at page $51$ in the v1 version of \cite{BGT} available on arXiv and during personal communications with Yuri Tomilov.}.

The aim of the present paper is to investigate these sampling questions and determine what the \emph{displacement condition}
\begin{equation}\label{eq:displacement}
  \sup_{n\ge1}\|\Id-T^n\|\le q<2
\end{equation}
says about $T$ when no spectral assumption is made. The case $q<1$ is special: \(\sup_{n\ge1}\|I-T^n\|\le q<1\) implies $T=I$. We call $\|\Id-T^n\|$ the \emph{displacement} of $T^n$ from the identity. Also, for each fixed \(1<q<2\), we characterise the sequences \(\mathcal N\) of positive integers for which the conditions \(\sup_{n\in\mathcal N}\|I-T^n\|\le q\) and \(\sigma(T)\cap\mathbb T\subset\{1\}\) imply that \(T\) is Ritt, for every power bounded operator \(T\) on every complex Banach space. Results of this type bear some analogy with those concerning Jamison sequences; see \cite{BDG,BadeaGrivaux07,BadeaGrivaux17,BGSurvey,EisnerGrivaux,RansfordRoginskaya}.  

The following five theorems describe the main results; the notation is explained after their statements and in Section~\ref{sec:prelim}.

\subsection*{All powers}
Condition \eqref{eq:displacement} alone does not imply the Ritt property: the scalar operator $T=e^{2\pi i/3}\Id$ satisfies $\sup_n\|\Id-T^n\|=\sqrt3$. It turns out that this is the only kind of obstruction. For $\lambda\in\T$ let $\od(\lambda):=\sup_{n\ge1}|\lambda^n-1|$ be the \emph{diameter of the orbit} of $\lambda$ seen from $1$. For $q\in[0,2)$ let
\[
   E_q:=\{1\}\cup\{\lambda\in\T:\od(\lambda)\le q\}.
\]
The set $E_q$ is the finite set consisting of $1$ and the roots of unity of odd order $m$ with $2\cos\frac{\pi}{2m}\le q$ (Lemma~\ref{lem:orbit}), and we let $N(q)$ be the least common multiple of the orders of its elements, an odd integer. For a finite nonempty $E\subset\T$, $T$ is said to be a $\Ritt{E}$ operator, in the sense of Bouabdillah and Le Merdy \cite{BLM}, if $\sigma(T)\subset\ol\D$ and $K_E(T):=\sup_{|\lambda|>1}\dist(\lambda,E)\|(\lambda-T)^{-1}\|<\infty$.

We write $M(T):=\sup_{n\ge0}\|T^n\|$ for the power bound of $T$.

\begin{theoremA}
Let $T\in\B(X)$ satisfy \eqref{eq:displacement} with $1\le q<2$. Then:
\begin{enumerate}[label=\textup{(\roman*)}]
\item $\sigma(T)\cap\T\subset E_q$, and $T$ is a Ritt operator if and only if $\sigma(T)\cap\T\subset\{1\}$;
\item $T$ is a $\Ritt{E_q}$ operator with $K_{E_q}(T)\le\Phi(q)$, where $\Phi(q)$ depends only on $q$; consequently $T^{N(q)}$ is a Ritt operator with $\CR(T^{N(q)})\le N(q)\Phi(q)$;
\item for $L\ge1$, the implication ``\eqref{eq:displacement} $\Rightarrow$ $T^L$ is Ritt'' holds on every Banach space if and only if $N(q)$ divides $L$;
\item if $q<\sqrt3$, then $E_q=\{1\}$, $T$ is a Ritt operator, and
\[
\CR(T)\le\frac{2\sqrt3 M(T)}{\sqrt3-q}
          \le\frac{6+2\sqrt3}{\sqrt3-q}.
\]
Define the optimal universal bound by
\begin{equation}\label{eq:optimalPhi}
 \Phi_*(q):=\sup\left\{\CR(S):
 \begin{array}{l}
 Y\text{ a complex Banach space},\quad S\in\B(Y),\\
 \displaystyle\sup_{n\ge1}\|\Id-S^n\|\le q
 \end{array}\right\},
 \qquad 1\le q<\sqrt3.
\end{equation}
Then $\Phi_*(q)<\infty$, the threshold $\sqrt3$ is sharp, and
\[
 \Phi_*(q)=\frac{2}{\sqrt3-q}+O(1)
 \qquad(q\uparrow\sqrt3).
\]
\end{enumerate}
\end{theoremA}

Thus $\sqrt3$ is the exact threshold of the unconditional discrete
zero--two law. For $q\ge\sqrt3$, the peripheral hypothesis in
\cite[Corollary~6.1]{BGT} excludes precisely the nontrivial roots in $E_q$.
In the range $\sqrt3\le q<2\cos(\pi/10)$, the least universal exponent
is $N(q)=3$. The sharp asymptotic constant in part (iv) is obtained by
balancing the first two displacement bounds through a quadratic polynomial.
The scalar family
$a_r=r\exp(i\arccos(-r/2))$, the elements of which are viewed as operators on $\C$, satisfies
\[
 \sup_n|1-a_r^n|=\sqrt{1+2r^2}=q_r,
 \qquad \CR(a_r)=\frac{4q_r}{3-q_r^2},
\]
and is asymptotically extremal. At every higher threshold $Q$, the
optimal $K_{E_q}$ bound has order $(Q-q)^{-1}$ from below
(Theorem~\ref{thm:higher-thresholds}). General norm-gap symbols and
constructive resolvent estimates are treated in
Section~\ref{sec:generalsymbol}.

\subsection*{Converses on uniformly convex spaces} The following results are proved under the assumption that the Banach space is uniformly convex.
 
\begin{theoremB}
Let $X$ be uniformly convex and let $T\in\B(X)$ be a contraction. The following assertions are equivalent:
\begin{enumerate}[label=\textup{(\roman*)}]
\item $\sup_{n\ge1}\|\Id-T^n\|<2$;
\item $T^N$ is a Ritt operator for some odd positive integer $N$;
\item $T$ is a $\Ritt{E}$ operator for some finite nonempty set $E\subset\T$ of roots of unity of odd order.
\end{enumerate}
The quantitative version of \textup{(ii)}$\Rightarrow$\textup{(i)} is: $\sup_n\|\Id-T^n\|\le2(1-\delta_X(1/(NK_+)))$, where $\delta_X$ is the modulus of convexity of $X$ and $K_+=\sup_n\|(\Id+T^{nN})^{-1}\|$.
\end{theoremB}

Neither contractivity nor uniform convexity can be dropped (see the remarks after Theorem~\ref{thm:UCchar}), and the parity of $N$ is essential: $T=-\Id$ has $T^2=\Id$.

\subsection*{Quantitative estimates on Hilbert space} More is true in the case of Hilbert spaces. 
\begin{theoremC}
Let $T$ be a Ritt operator on a Hilbert space $H$, $C=\CR(T)$, $M=M(T)=\sup_{n\ge0}\|T^n\|$.
\begin{enumerate}[label=\textup{(\roman*)}]
\item For every $\zeta\in\T\setminus\{1\}$,
$\displaystyle\sup_{n\ge1}\|(\zeta-T^n)^{-1}\|\le1+\frac{CM}{|\zeta-1|}$.
\item $\displaystyle\sup_{m\ge1}\CR(T^m)\le CM+2$; in particular $\sup_m\CR(T^m)\le\CR(T)+2$ if $T$ is a contraction.
\item If $T$ is a contraction and
$W\in A(\D;\B(E,F))$ for Hilbert spaces $E,F$, with
$\|W(1)\|<\|W\|_\infty$, then
$\sup_{n\ge1}\|W(T^n)\|<\|W\|_\infty$.
For matrix polynomials an explicit bound is given by
Theorem~\ref{thm:polydefect}. For the displacement itself,
\[
 \sup_{n\ge1}\|\Id-T^n\|
 \le
 \begin{cases}
  2(C+1)/(C+2),&1\le C\le\sqrt2,\\
  2\sqrt{1-C^{-2}},&C\ge\sqrt2.
 \end{cases}
\]
\end{enumerate}
\end{theoremC}

Here $A(\D;\B(E,F))$ is the disc algebra of functions with values in $\B(E,F)$. Part (i) is proved using the Fourier vectors $\sum_k\mu_j^{-k}T^kx$ built on the $n$-th roots $\mu_j$ of $\zeta$; part (ii) follows from (i) and the observation that the Ritt constant of an operator is the supremum of $|\xi-1|\|(\xi-T)^{-1}\|$ over the unit circle (Proposition~\ref{prop:boundary}). The matrix-polynomial estimate in (iii) combines unitary dilation with matrix Fej\'er--Riesz factorisation, while the disc-algebra extension uses Runge approximation and concentration of dilation spectral measures. An analogue of (i) holds on $L^p$, $1<p<\infty$ (Theorem~\ref{thm:Lp}); the polynomial norm-gap assertion itself fails for general uniformly convex spaces, even for a nilpotent contraction on $\ell^4_3$ (Example~\ref{ex:UCpolynomial}).

\subsection*{Sparse sequences of powers}
Let $\bn=(n_k)_{k\ge0}$ be a strictly increasing sequence of positive integers with $n_0=1$, and put
\[
   c(\bn):=\sup_{k}\frac{n_{k+1}}{n_k}\in[1,\infty],\qquad c_\infty(\bn):=\limsup_{k\to\infty}\frac{n_{k+1}}{n_k}\le c(\bn).
\]
For $c\ge1$ let
\[
  \arc_c:=\Big\{e^{it}:\ \frac{2\pi}{c+1}\le t\le\frac{2\pi c}{c+1}\Big\},
\]
the closed arc symmetric about $-1$ whose endpoints have argument ratio $c$. Therefore, $\arc_1=\{-1\}$ and $\arc_c$ increases with $c$. We say that $\bn$ is a \emph{Ritt sampling sequence} if $\sup_kn_k\|T^{n_k}-T^{n_k-1}\|<\infty$ implies that $T$ is Ritt, for every power bounded $T$ on every Banach space.

\begin{theoremD}
\begin{enumerate}[label=\textup{(\roman*)}]
\item Let $c_\infty(\bn)<\infty$ and $c'\ge c_\infty(\bn)$. If $T$ is power bounded, $\sigma(T)\cap\T\subset\{1\}$, and $\sup_{k}\sup_{\eta\in\arc_{c'}}\|(\eta-T^{n_k})^{-1}\|<\infty$, then $T$ is a Ritt operator. In particular, if $n_{k+1}/n_k\to1$, the one-point condition $\sup_k\|(\zeta-T^{n_k})^{-1}\|<\infty$, for any $\zeta\in\T\setminus\{1\}$, suffices, exactly as for the full sequence; and if $c(\bn)\le c<\infty$, the norm gap $\sup_k\|\Id-T^{n_k}\|<2\sin\frac{\pi}{c+1}$ suffices for every power bounded operator, without a peripheral spectral assumption.
\item For $n_k=b^k$, where $b\ge2$ is an integer, the arc $\arc_b$ cannot be replaced by any closed arc contained in its interior, even for normal contractions with $\sigma(T)\subset\D\cup\{1\}$; under this spectral assumption the sharp norm threshold is $\max\{1,2\sin\frac{\pi}{b+1}\}$, and without it the sharp threshold is $2\sin\frac\pi{b+1}$.
\item If $c(\bn)=\infty$, then for every closed arc $\arc\subset\T\setminus\{1\}$ and every $q>1$ there are normal contractions with $\sigma(T)\subset\D\cup\{1\}$, not Ritt, with $\sup_k\sup_{\eta\in\arc}\|(\eta-T^{n_k})^{-1}\|<\infty$, respectively with $\sup_k\|\Id-T^{n_k}\|\le q$.
\item $\bn$ is a Ritt sampling sequence if and only if $c(\bn)<\infty$; in that case $\sup_nn\|T^n-T^{n-1}\|\le M(T)c(\bn)\sup_kn_k\|T^{n_k}-T^{n_k-1}\|$.
\end{enumerate}
\end{theoremD}

The angular part of these statements is governed by the Jamison constant
\[
   J(\bn)=\inf_{\lambda\in\T\setminus\{1\}}\ \sup_{k\ge0}|\lambda^{n_k}-1|
\]
of the sequence \cite{BadeaGrivaux07,EisnerGrivaux,RansfordRoginskaya}: under a sampled norm bound $q<J(\bn)$, spectral mapping forces $\sigma(T)\cap\T\subset\{1\}$. The radial part---how fast the spectrum may approach $1$ inside a Stolz domain---is governed by the growth of the quotients $n_{k+1}/n_k$, and this is the exact condition both for Ritt sampling and for the existence of an arc criterion. 

We also prove the following result; cf. Theorem~\ref{samp:exact}. Write $\mathcal N=\{n_k:k\ge0\}$ for a sequence of positive integers. We do not assume $n_0=1$. 

\begin{theoremE}
Fix $1<q<2$. The following assertions are equivalent.
\begin{enumerate}
\item For every complex Banach space $X$ and every power bounded
operator $T\in\B(X)$,
$$
 \sup_{n\in\mathcal N}\|\Id-T^n\|\le q,
 \qquad \sigma(T)\cap\T\subset\{1\}
$$
imply that $T$ is a Ritt operator.
\item The same implication holds for every normal contraction on $\ell^2$.
\item $q < \alpha(\mathcal N):= \lim_{B\to\infty}\liminf_{t\downarrow0}h_{\mathcal N}(B,t)$, where for $B,t>0$,
$$
 h_{\mathcal N}(B,t)
 :=\max\bigl\{|1-e^{int}|:n\in\mathcal N,\ nt\le B\bigr\},
$$ and the maximum of the empty set is $0$. 
\end{enumerate}
\end{theoremE}

We also prove that \[
 \alpha(\mathcal N)>0
 \quad\Longleftrightarrow\quad c_\infty(\bn)<\infty
\]
and that
\[
 \frac{n_{k+1}}{n_k}\longrightarrow b\in\{2,3,\ldots\}
 \quad\Longrightarrow\quad
 \alpha(\mathcal N)=2\sin\frac{\pi}{b+1}.
\]
The endpoint $q=1$ has a different character. There are sampling sets
with unbounded quotients for which the endpoint bound forces every bounded
operator to be Ritt, without any prior power bound. Factorial sampling, on
the other hand, admits a normal non-Ritt contraction at the same endpoint.
Section~\ref{sec:endpoint} proves these assertions and gives an exact
endpoint criterion for normal contractions.

\subsection*{Methods}
The sampling sufficiency results rest on an elementary mechanism inspired by \cite{BGT}. This is isolated in the \emph{sampling lemma} of Section~\ref{sec:sampling}: if $\lambda$ approaches $1$ tangentially from outside the disc, then some angular part $e^{in\arg\lambda}$, with $n$ belonging to our sequence $\mathcal N$ and $n|\lambda-1|$ bounded, lies on a prescribed arc of $\T\setminus\{1\}$ where the resolvents of $T^n$ are controlled; the power $\lambda^n$ is close enough to this angular part for a Neumann perturbation, and the factorisation $\lambda^n-T^n=(\lambda-T)\sum_j\lambda^{n-1-j}T^j$ transfers this control back to $(\lambda-T)^{-1}$. The converses use Fourier vectors on the roots of $\zeta$, the maximum principle for the operator function $(z-1)(z-T)^{-1}$, the modulus of convexity, and unitary dilations.

\subsection*{Organisation}
Section~\ref{sec:prelim} records resolvent preliminaries and the relation
between finite peripheral spectrum and Ritt powers.
Section~\ref{sec:sampling} proves the sampling lemma and the discrete Kato
criterion. Section~\ref{sec:allpowers} treats full displacement bounds,
sharp threshold asymptotics, and general symbols.
Section~\ref{sec:sparse} concerns sampled resolvents and differences;
Section~\ref{sec:exact-sampling} establishes the exact displacement
criterion for $1<q<2$. The endpoint is treated separately in
Section~\ref{sec:endpoint}. Section~\ref{sec:UC} proves the uniformly
convex characterisation. Sections~\ref{sec:hilbert} and~\ref{sec:matrix}
contain the quantitative resolvent estimates and the operator-valued
norm-gap results.

\subsection*{Notation}
$\D$ is the open unit disc, $\T$ the unit circle, $\mu_N=\{\xi\in\T:\xi^N=1\}$. For $T\in\B(X)$ we write $\Res(\lambda,T)=(\lambda-T)^{-1}$ for $\lambda\in\rho(T)$, $M(T)=\sup_{n\ge0}\|T^n\|$, and $\CR(T)$ as in \eqref{eq:rittdef}. Arcs of $\T$ are denoted by $\arc$ and are closed arcs of positive length unless stated otherwise. We use the principal argument $\arg z\in(-\pi,\pi]$; arc endpoints are also described by angles in $[0,2\pi]$. Throughout, $A\lesssim B$ means $A\le cB$ for a constant $c$ whose dependence is indicated. We use the elementary inequalities $\frac2\pi|t|\le|e^{it}-1|\le|t|$ for $|t|\le\pi$ without further comment.

\section{Preliminaries}\label{sec:prelim}

\subsection{Ritt operators}
We shall use freely the following standard facts; see \cite{LeMerdy,Lyubich,NagyZemanek,Nevanlinna}. If $T$ is a Ritt operator, then $T$ is power bounded, $\sigma(T)\subset\D\cup\{1\}$, and $D(T):=\sup_{n\ge1}n\|T^n-T^{n-1}\|<\infty$; conversely, a power bounded operator with $D(T)<\infty$ is Ritt. Each of the constants involved can be bounded in terms of the others: there are universal nondecreasing functions $f_0,f_1,f_2$ with
\begin{equation}\label{eq:quantequiv}
   M(T)\le f_0(\CR(T)),\qquad D(T)\le f_1(\CR(T)),\qquad \CR(T)\le f_2\big(M(T),D(T)\big),
\end{equation}
and an explicit bound $M(T)\lesssim\CR(T)\log(1+\CR(T))$ is due to Bakaev and Schwenninger \cite{Schwenninger}. We shall also use the following result of Borichev, Gomilko and Tomilov.

\begin{theorem}[{\cite[Theorem~7.6]{BGT}}]\label{thm:BGTpowers}
If $T$ is a Ritt operator, then $\sup_{m\ge1}\CR(T^m)<\infty$.
\end{theorem}

Their proof uses the direct sum $S_T=\bigoplus_{m\ge1}T^m$ on $\ell_\infty(X)$ and the shifted difference quantity $\sup_{n\ge1}n\|T^n(\Id-T)\|$. With our convention for $D(T)$, the following proof gives a universal bound directly.

\begin{proposition}\label{prop:universalPhi}
There is a nondecreasing function $F:[1,\infty)\to[1,\infty)$, independent of the Banach space, such that every Ritt operator satisfies
\[
 \sup_{m\ge1}\CR(T^m)\le F(\CR(T)).
\]
\end{proposition}

\begin{proof}
Clearly $M(T^m)\le M(T)$ and $\|\Id-T^m\|\le1+M(T)$. For $m\ge1$ and $n\ge2$, writing $\Id-T^m=\sum_{j=0}^{m-1}T^j(\Id-T)$ gives
\[
 n\|T^{m(n-1)}(\Id-T^m)\|
 \le nD(T)\sum_{j=0}^{m-1}\frac1{m(n-1)+j+1}
 \le\frac n{n-1}D(T)\le2D(T).
\]
Thus $D(T^m)\le\max\{1+M(T),2D(T)\}$. With $C=\CR(T)$, \eqref{eq:quantequiv} gives the assertion with
\[
 F(C)=\max\Big\{1,\ f_2\big(f_0(C),\max\{1+f_0(C),2f_1(C)\}\big)\Big\}.
\]
The proof is complete.
\end{proof}

On Hilbert space we shall obtain the explicit bound $\sup_m\CR(T^m)\le\CR(T)M(T)+2$ (Proposition~\ref{prop:explicitpowers}).

The next statement says that the Ritt constant depends only on the behaviour of the resolvent on the unit circle. The inequality $\ge$ is the trivial half; the inequality $\le$ is a maximum principle and will be used to transfer boundary resolvent estimates to the whole exterior of the disc. 

\begin{proposition}[boundary characterisation of the Ritt constant]\label{prop:boundary}
Let $S\in\B(X)$ be power bounded with $\sigma(S)\cap\T\subset\{1\}$. Then $\T\setminus\{1\}\subset\rho(S)$ and
\begin{equation}\label{eq:boundary}
   \CR(S)=\sup_{\xi\in\T\setminus\{1\}}|\xi-1|\,\|\Res(\xi,S)\|\ \in[1,\infty].
\end{equation}
In particular $S$ is a Ritt operator if and only if the right-hand side of \eqref{eq:boundary} is finite, and then $\|\Res(\xi,S)\|\le\CR(S)/|\xi-1|$ for $\xi\in\T\setminus\{1\}$.
\end{proposition}

\begin{proof}
Denote the right-hand side of \eqref{eq:boundary} (the sup) by $B$. For $\xi\in\T\setminus\{1\}$ we have $\xi\in\rho(S)$, $\Res(r\xi,S)\to\Res(\xi,S)$ in norm as $r\downarrow1$, and $|r\xi-1|\to|\xi-1|$; hence $B\le\CR(S)$.

Conversely, assume $B<\infty$. We first show that $\CR(S)<\infty$. Let $M=M(S)$, $z=r\xi$ with $r>1$, $\xi\in\T$. The Neumann series gives $(r-1)\|\Res(z,S)\|\le M$. If $\xi\ne1$, the resolvent identity $\Res(z,S)=\Res(\xi,S)+(\xi-z)\Res(z,S)\Res(\xi,S)$ gives
\[
 \|\Res(z,S)\|\le\|\Res(\xi,S)\|\bigl(1+(r-1)\|\Res(z,S)\|\bigr)\le(1+M)\|\Res(\xi,S)\|,
\]
and since $|z-1|\le(r-1)+|\xi-1|$ we obtain $|z-1|\|\Res(z,S)\|\le M+(1+M)B$; when $\xi=1$ the Neumann estimate alone gives $|z-1|\|\Res(z,S)\|\le M$. Thus $\CR(S)\le M+(1+M)B<\infty$.

Consider now the operator-valued function
\[
 G(w)=(1-w)(\Id-wS)^{-1},\qquad w\in\D,
\]
with $G(0)=\Id$; for $0<|w|<1$ and $z=1/w$ one has $\Id-wS=w(z-S)$ and $(1-w)/w=z-1$, so that $G(w)=(z-1)\Res(z,S)$, and $w\mapsto z$ maps $\D\setminus\{0\}$ onto $\{|z|>1\}$ and $\T\setminus\{1\}$ onto itself. It is holomorphic on $\D$, bounded by $\CR(S)$, and extends continuously to $\ol\D\setminus\{1\}$ with $\|G(e^{it})\|\le B$ for $e^{it}\ne1$. Fix $x\in X$ and $x^*\in X^*$ of norm one and put $g(w)=\langle G(w)x,x^*\rangle$, a bounded holomorphic function on $\D$. Its radial limits $g^*(e^{it})$ exist almost everywhere and coincide with $\langle G(e^{it})x,x^*\rangle$ for $e^{it}\ne1$; the exceptional point has measure zero, so $|g^*|\le B$ almost everywhere. The Poisson representation $g(w)=\int_\T P_w\,g^*\,dm$ of bounded holomorphic functions then gives $|g(w)|\le B$ for all $w\in\D$. Taking the supremum over $x$ and $x^*$ yields $\|G(w)\|\le B$, that is, $|z-1|\|\Res(z,S)\|\le B$ for all $|z|>1$, and \eqref{eq:boundary} is proved. Finally $\CR(S)\ge\lim_{|z|\to\infty}|z-1|\,\|\Res(z,S)\|=1$, so both sides of \eqref{eq:boundary} are at least $1$.
\end{proof}

\begin{remark}
For a power bounded operator $T$ with $\sigma(T)\cap\T\subset\mu_N$, the same radial estimate and maximum-principle argument, now applied to $H(w)=(1-w^N)(\Id-wT)^{-1}$, show that $K_{\mu_N}(T)$ of Section~\ref{sec:rootsofunity} is controlled, up to a factor depending only on $N$, by the boundary supremum $\sup_{\xi\in\T\setminus\mu_N}\dist(\xi,\mu_N)\|\Res(\xi,T)\|$.
\end{remark}

\begin{lemma}[factorisation]\label{lem:factor}
Let $T\in\B(X)$, $\lambda\in\C$ and $n\ge1$. Then
\[
  \lambda^n-T^n=(\lambda-T)\,Q_n(\lambda,T),\qquad Q_n(\lambda,T):=\sum_{j=0}^{n-1}\lambda^{n-1-j}T^j,
\]
and the two factors commute. Consequently, if $\lambda^n\in\rho(T^n)$, then $\lambda\in\rho(T)$ and
\[
  \Res(\lambda,T)=Q_n(\lambda,T)\,\Res(\lambda^n,T^n),\qquad
  \|\Res(\lambda,T)\|\le M(T)\,n\max\{1,|\lambda|\}^{n-1}\,\|\Res(\lambda^n,T^n)\| .
\]
\end{lemma}

\begin{proof}
The identity is the operator form of $\lambda^n-z^n=(\lambda-z)\sum_j\lambda^{n-1-j}z^j$. If $a,b$ commute and $ab$ is invertible, then $a$ is invertible with $a^{-1}=b(ab)^{-1}$.
\end{proof}

\begin{lemma}[localisation]\label{lem:reduction}
Let $T$ be power bounded with $\sigma(T)\cap\T\subset\{1\}$, and suppose that for some $\delta>0$, $C>0$,
\[
   |\lambda-1|\,\|\Res(\lambda,T)\|\le C\qquad(|\lambda|>1,\ |\lambda-1|<\delta).
\]
Then $T$ is a Ritt operator, and $\CR(T)\le\max\{3M(T),\,C,\,3B_\delta(T)\}$, where
\[
 B_\delta(T)=\sup\{\|\Res(\lambda,T)\|:\ 1\le|\lambda|\le2,\ |\lambda-1|\ge\delta\}<\infty.
\]
\end{lemma}

\begin{proof}
For $|\lambda|\ge2$ the Neumann series gives $\|\Res(\lambda,T)\|\le M(T)/(|\lambda|-1)$ and $|\lambda-1|\le3(|\lambda|-1)$. The compact set $\{1\le|\lambda|\le2,\ |\lambda-1|\ge\delta\}$ is contained in $\rho(T)$, on which the resolvent is continuous, so $B_\delta(T)<\infty$, and $|\lambda-1|\le3$ there.
\end{proof}

\subsection{Ritt operators with peripheral spectrum in the roots of unity}\label{sec:rootsofunity}
For a finite nonempty set $E\subset\T$ and $T\in\B(X)$ put
\[
  K_E(T):=\sup_{|\lambda|>1}\dist(\lambda,E)\,\|\Res(\lambda,T)\| .
\]
Following Bouabdillah and Le Merdy \cite{BLM}, we say that $T$ is a $\Ritt{E}$ \emph{operator} if $\sigma(T)\subset\ol\D$ and $K_E(T)<\infty$; then $\sigma(T)\cap\T\subset E$, and by a theorem of El-Fallah and Ransford \cite{ElFallahRansford}, $M(T)\le\frac e2K_E(T)^2\,\#E$. Bouabdillah and Le Merdy show that $T$ is $\Ritt{E}$ if and only if $T$ is power bounded and $\sup_n n\|T^{n-1}\prod_{\xi\in E}(\xi-T)\|<\infty$ \cite[Theorem~2.10]{BLM}. 
In this paper we shall consider the case $E=\mu_N$, which is simpler and symmetric.
We write $K_N(T):=K_{\mu_N}(T)$.

\begin{lemma}\label{lem:scalarroots}
Let $N\ge1$ and $c_N:=\sin^{N-1}(\pi/N)$, with the convention $c_1=1$. For $1\le|\lambda|\le2$,
\[
   c_N\,\dist(\lambda,\mu_N)\le|\lambda^N-1|\le 3^{N-1}\dist(\lambda,\mu_N).
\]
\end{lemma}

\begin{proof}
Write $\lambda^N-1=\prod_{\xi\in\mu_N}(\lambda-\xi)$ and let $\xi_0\in\mu_N$ be a nearest point to $\lambda$. For $\xi\ne\xi_0$, $|\lambda-\xi|\le|\lambda|+1\le3$, while $2|\lambda-\xi|\ge|\lambda-\xi|+|\lambda-\xi_0|\ge|\xi-\xi_0|\ge2\sin(\pi/N)$.
\end{proof}

\begin{proposition}[power reduction]\label{prop:TN}
Let $N\ge1$, $s_N=\sin(\pi/(2N))$, and let $T\in\B(X)$ with $\sigma(T)\subset\ol\D$.
Then $T$ is $\Ritt{\mu_N}$ if and only if $T^N$ is Ritt. More precisely,
\begin{align*}
 T^N\text{ Ritt}&\quad\Longrightarrow\quad
 M(T)\le\max_{0\le r<N}\|T^r\|M(T^N),\qquad
 K_N(T)\le \frac{M(T)\CR(T^N)}{s_N},\\
 T\text{ is }\Ritt{\mu_N}&\quad\Longrightarrow\quad
 \CR(T^N)\le N K_N(T).
\end{align*}
For $N=1$ the two constants coincide.
\end{proposition}
\begin{proof}
Suppose first that $T^N$ is Ritt. Writing $n=Nk+r$ with $0\le r<N$ proves the power bound.
For $|\lambda|>1$, choose a nearest $\xi\in\mu_N$ and write
$\lambda=\xi r e^{it}$ with $r>1$ and $|t|\le\pi/N$. Then $\lambda^N-1=\lambda^N-\xi^N=(\lambda-\xi)\xi^{N-1}\sum_{j=0}^{N-1}r^je^{ijt}$, and rotation of the sum through the angle $-(N-1)t/2$ gives
\[
 \left|\sum_{j=0}^{N-1}r^je^{ijt}\right|
 \ge \sum_{j=0}^{N-1}r^j\cos\bigl((j-\tfrac{N-1}{2})t\bigr)
 \ge s_N\sum_{j=0}^{N-1}r^j,
\]
because $|(j-\frac{N-1}2)t|\le\frac{(N-1)\pi}{2N}=\frac\pi2-\frac{\pi}{2N}$. Consequently $|\lambda^N-1|\ge s_N\dist(\lambda,\mu_N)\sum_{j=0}^{N-1}r^j$.
Lemma~\ref{lem:factor}, with
$\|Q_N(\lambda,T)\|\le M(T)\sum_{j=0}^{N-1}r^j$ and $\|\Res(\lambda^N,T^N)\|\le\CR(T^N)/|\lambda^N-1|$, proves the first resolvent estimate.

Conversely, let $T$ be $\Ritt{\mu_N}$, let $|\mu|>1$, let $\lambda_0,\ldots,\lambda_{N-1}$ be the $N$-th roots of $\mu$, and put $r=|\mu|^{1/N}$. The partial fraction decompositions of $1/(\mu-z^N)$ give
\[
 \Res(\mu,T^N)=\sum_{j=0}^{N-1}\frac{\Res(\lambda_j,T)}{N\lambda_j^{N-1}}.
\]
If $\xi$ is nearest to $\lambda_j$ in $\mu_N$, then
$|\mu-1|=|\lambda_j^N-\xi^N|\le N r^{N-1}|\lambda_j-\xi|=N r^{N-1}\dist(\lambda_j,\mu_N)$.
Multiplying the partial fraction decomposition formula by $|\mu-1|$ and using $\dist(\lambda_j,\mu_N)\|\Res(\lambda_j,T)\|\le K_N(T)$ gives
$|\mu-1|\|\Res(\mu,T^N)\|\le NK_N(T)$.
\end{proof}

\subsection{Orbit diameters of unimodular numbers}\label{sec:orbits}
For $\lambda\in\T$ put $\od(\lambda):=\sup_{n\ge1}|\lambda^n-1|$, the \emph{orbit diameter} of $\lambda$ seen from $1$.

\begin{lemma}\label{lem:orbit}
Let $\lambda\in\T\setminus\{1\}$. If $\lambda$ is not a root of unity, then $\od(\lambda)=2$. If $\lambda$ is a root of unity of order $m\ge2$, then $\od(\lambda)=2$ if $m$ is even and $\od(\lambda)=2\cos\frac{\pi}{2m}$ if $m$ is odd. In particular $\od(\lambda)\ge\sqrt3$ for every $\lambda\in\T\setminus\{1\}$, with equality exactly for $\lambda=e^{\pm2\pi i/3}$.
\end{lemma}

\begin{proof}
If $\lambda$ is not a root of unity, then $\{\lambda^n\}$ is dense in $\T$ and $\od(\lambda)=\sup_{\eta\in\T}|\eta-1|=2$. If $\lambda$ has order $m$, then $\{\lambda^n:n\ge1\}=\mu_m$ and $\od(\lambda)=\max_{0\le k<m}2\sin\frac{\pi k}{m}$, which is $2$ if $m$ is even ($k=m/2$) and $2\sin\frac{\pi(m-1)}{2m}=2\cos\frac{\pi}{2m}$ if $m$ is odd. Finally $2\cos\frac\pi{2m}$ increases with $m$ and equals $\sqrt3$ for $m=3$.
\end{proof}

Recall the following definition of $E_q$ for $q\in[0,2)$. We define
\begin{equation}\label{eq:Eq}
   E_q:=\{1\}\cup\{\lambda\in\T:\ \od(\lambda)\le q\}
   =\{1\}\cup\Big\{\lambda\in\T:\ \lambda\text{ has odd order }m\ \text{and}\ 2\cos\tfrac{\pi}{2m}\le q\Big\},
\end{equation}
Let $N(q)$ be the least common multiple of the orders of the elements of $E_q$. By Lemma~\ref{lem:orbit}, $E_q$ is a finite set of roots of unity of odd order, $N(q)$ is odd, $E_q\subset\mu_{N(q)}$, and
\[
  N(q)=1\ \ (q<\sqrt3),\qquad N(q)=3\ \ (\sqrt3\le q<2\cos\tfrac{\pi}{10}),\qquad N(q)=15\ \ (2\cos\tfrac{\pi}{10}\le q<2\cos\tfrac{\pi}{14}),
\]
and so on. Note that $E_q$ is invariant under conjugation and that $E_q=\mu_3$ for $\sqrt3\le q<2\cos\frac\pi{10}\approx1.902$.

\section{The sampling lemma and the discrete Kato criterion}\label{sec:sampling}

The sampling criteria rest on an elementary mechanism inspired by \cite{BGT}: if $\lambda$ approaches $1$ tangentially from outside the disc, then an angular part $e^{in\arg\lambda}$, with $n|\lambda-1|$ bounded, lies on a prescribed arc of $\T\setminus\{1\}$ where the resolvents of $T^n$ are controlled. The power $\lambda^n$ is sufficiently close to this angular part for a Neumann perturbation, and the factorisation of Lemma~\ref{lem:factor} then transfers this control back to $\Res(\lambda,T)$. We isolate the mechanism in a form that applies to arbitrary sets of sampling times.

\begin{definition}\label{def:sampling}
Let $\mathcal N\subset\N$ and let $\arc\subset\T\setminus\{1\}$ be a closed arc. We say that $(\mathcal N,\arc)$ has the \emph{sampling property with constants} $(\theta_*,B)$, $\theta_*>0$, $B>0$, if for every $\theta$ with $0<|\theta|<\theta_*$ there is $n\in\mathcal N$ such that
\[
   e^{in\theta}\in\arc\qquad\text{and}\qquad n|\theta|\le B .
\]
\end{definition}

\begin{lemma}[sampling lemma]\label{lem:master}
Let $T\in\B(X)$ be power bounded, $M=M(T)$. Let $(\mathcal N,\arc)$ have the sampling property with constants $(\theta_*,B)$, and assume that $\arc\subset\rho(T^n)$ for $n\in\mathcal N$ and
\[
   K:=\sup_{n\in\mathcal N}\ \sup_{\eta\in\arc}\|(\eta-T^n)^{-1}\|<\infty .
\]
Put $a:=\max\{1,\,B/\log(1+\tfrac1{2K})\}$, $\delta:=\min\{\theta_*/\pi,\,\tfrac12\}$ and $C:=\max\{(1+2a)M,\ 3(2K+1)MB\}$. Then
\begin{equation}\label{eq:master}
   |\lambda-1|\,\|\Res(\lambda,T)\|\le C\qquad(|\lambda|>1,\ |\lambda-1|<\delta).
\end{equation}
\end{lemma}

\begin{proof}
Let $\lambda=re^{i\theta}$ with $r>1$, $|\theta|\le\pi$ and $|\lambda-1|<\delta$. Then $r<\frac32$, and since $|\theta|\le\frac\pi2|e^{i\theta}-1|\le\frac\pi2\big((r-1)+|\lambda-1|\big)\le\pi|\lambda-1|$, we have $|\theta|<\theta_*$.

\emph{Non-tangential region} $|\theta|\le a(r-1)$. The Neumann series gives $\|\Res(\lambda,T)\|\le M/(r-1)$, while $|\lambda-1|\le(r-1)+r|e^{i\theta}-1|\le(r-1)+\frac32|\theta|\le(1+2a)(r-1)$. Hence $|\lambda-1|\|\Res(\lambda,T)\|\le(1+2a)M$.

\emph{Tangential region} $|\theta|>a(r-1)$. In particular $\theta\ne0$, so there is $n\in\mathcal N$ with $\eta:=e^{in\theta}\in\arc$ and $n|\theta|\le B$. Then $n(r-1)<n|\theta|/a\le B/a$, so $r^n\le e^{n(r-1)}\le e^{B/a}\le1+\frac1{2K}$ and
\[
   |\lambda^n-\eta|=|r^ne^{in\theta}-e^{in\theta}|=r^n-1\le\frac1{2K}.
\]
Since $\lambda^n-T^n=(\eta-T^n)\big(\Id+(\lambda^n-\eta)(\eta-T^n)^{-1}\big)$ and $\|(\lambda^n-\eta)(\eta-T^n)^{-1}\|\le\frac12$, the operator $\lambda^n-T^n$ is invertible with $\|(\lambda^n-T^n)^{-1}\|\le2K$. Lemma~\ref{lem:factor} gives
\[
   \|\Res(\lambda,T)\|\le Mnr^{n-1}\cdot2K\le(2K+1)Mn .
\]
Finally $|\lambda-1|\le(r-1)+\frac32|\theta|\le(\frac1a+\frac32)|\theta|\le3|\theta|$, whence $|\lambda-1|\|\Res(\lambda,T)\|\le3(2K+1)M\,n|\theta|\le3(2K+1)MB$.
\end{proof}

\begin{lemma}[sampling by all powers]\label{lem:allpowers}
Let $\arc=\{e^{it}:\alpha\le t\le\gamma\}$ with $0<\alpha<\gamma<2\pi$. Then $(\N,\arc)$ has the sampling property with constants $(\gamma-\alpha,\,2\pi)$.
\end{lemma}

\begin{proof}
Let $0<\theta<\gamma-\alpha$ and let $n$ be the least positive integer with $n\theta\ge\alpha$. Then $n\theta<\alpha+\theta<\gamma$, so $e^{in\theta}\in\arc$ and $n\theta<2\pi$. Let $-(\gamma-\alpha)<\theta<0$ and let $n$ be the least positive integer with $n|\theta|\ge2\pi-\gamma$. Then $n|\theta|<2\pi-\gamma+|\theta|<2\pi-\alpha$, so $2\pi-n|\theta|\in(\alpha,\gamma]$ and $e^{in\theta}=e^{i(2\pi-n|\theta|)}\in\arc$, while $n|\theta|<2\pi$.
\end{proof}

\begin{lemma}[sampling by a subsequence]\label{lem:subseq}
Let $\bn=(n_k)_{k\ge0}$ be strictly increasing, let $c'>\limsup_kn_{k+1}/n_k$, and let $k_0$ be such that $n_{k+1}\le c'n_k$ for all $k\ge k_0$. Let $\arc=\{e^{it}:\alpha\le t\le\gamma\}$ with $0<\alpha<\gamma<2\pi$ be a closed arc such that
\begin{equation}\label{eq:arccond}
   \frac{\gamma}{\alpha}\ge c'\qquad\text{and}\qquad\frac{2\pi-\alpha}{2\pi-\gamma}\ge c' .
\end{equation}
Then $(\{n_k\},\arc)$ has the sampling property with constants $\big(\min\{\alpha,2\pi-\gamma\}/n_{k_0},\ \max\{\gamma,2\pi-\alpha\}\big)$. In particular this applies to $\arc=\arc_{c'}$, for which both ratios in \eqref{eq:arccond} equal $c'$.
\end{lemma}

\begin{proof}
Let $0<\theta<\alpha/n_{k_0}$ and let $k$ be the largest index with $n_k\theta<\alpha$; it exists since $n_{k_0}\theta<\alpha$ and $n_k\to\infty$, and $k\ge k_0$. Put $n=n_{k+1}$. Then $\alpha\le n\theta\le c'n_k\theta<c'\alpha\le\gamma$, so $e^{in\theta}\in\arc$ and $n\theta\le\gamma$. Let $-(2\pi-\gamma)/n_{k_0}<\theta<0$ and let $k$ be the largest index with $n_k|\theta|<2\pi-\gamma$; again $k\ge k_0$, and $n=n_{k+1}$ satisfies $2\pi-\gamma\le n|\theta|<c'(2\pi-\gamma)\le2\pi-\alpha$. Hence $2\pi-n|\theta|\in(\alpha,\gamma]$, so $e^{in\theta}=e^{i(2\pi-n|\theta|)}\in\arc$, and $n|\theta|\le2\pi-\alpha$.
\end{proof}

Combining the sampling lemma with Lemma~\ref{lem:reduction} we obtain the following general criterion.

\begin{proposition}\label{thm:general}
Let $T\in\B(X)$ be power bounded with $\sigma(T)\cap\T\subset\{1\}$. Let $(\mathcal N,\arc)$ have the sampling property and assume that
\[
   \sup_{n\in\mathcal N}\ \sup_{\eta\in\arc}\|(\eta-T^n)^{-1}\|<\infty .
\]
Then $T$ is a Ritt operator.
\end{proposition}

For $\mathcal N=\N$ and $\arc$ a small arc around a point $\zeta$, Proposition~\ref{thm:general} is the discrete Kato criterion of Borichev, Gomilko and Tomilov. We record the deduction, which shows in particular that the constant in \eqref{eq:master} depends only on $M(T)$, $\zeta$ and $\sup_n\|(\zeta-T^n)^{-1}\|$.

\begin{corollary}[{discrete Kato criterion, \cite[Theorem~5.6]{BGT}}]\label{cor:kato}
Let $T$ be power bounded with $\sigma(T)\cap\T\subset\{1\}$, and let $\zeta\in\T\setminus\{1\}$ with $K_\zeta:=\sup_{n\ge1}\|(\zeta-T^n)^{-1}\|<\infty$. Then $T$ is a Ritt operator.
\end{corollary}

\begin{proof}
Let $\varepsilon:=\frac12\min\{1/(2K_\zeta),|\zeta-1|\}$ and $\arc:=\{\eta\in\T:|\eta-\zeta|\le\varepsilon\}$, a closed arc in $\T\setminus\{1\}$. For $\eta\in\arc$ and $n\ge1$, the Neumann series gives $\eta\in\rho(T^n)$ and $\|(\eta-T^n)^{-1}\|\le\frac43K_\zeta\le2K_\zeta$. Apply Lemma~\ref{lem:allpowers} and Proposition~\ref{thm:general}.
\end{proof}

\begin{remark}[the Ritt constant is not controlled]\label{rem:dependT}
The local estimate \eqref{eq:master} is uniform, but Proposition~\ref{thm:general} and Corollary~\ref{cor:kato} do not bound $\CR(T)$ in terms of upper bounds on $M(T)$ and $K_\zeta$ and the point $\zeta$ alone. Indeed, let $\zeta\in\T\setminus\mu_3$ and $T_\rho=\rho\omega\Id$ on $\C$, where $\omega=e^{2\pi i/3}$ and $0<\rho<1$. Then $M(T_\rho)=1$ and $\sigma(T_\rho)\cap\T=\emptyset$. For $\zeta,\xi\in\T$ and $0\le r\le1$,
\[
 |\zeta-\xi|\le|\zeta-r\xi|+(1-r)\le2|\zeta-r\xi|,
\]
because $|\zeta-r\xi|\ge1-r$. Taking $\xi=\omega^n$ and $r=\rho^n$ gives
\[
 K_\zeta(T_\rho):=\sup_n|\zeta-\rho^n\omega^n|^{-1}
 \le\frac2{\dist(\zeta,\mu_3)}
 \qquad(0<\rho<1).
\]
On the other hand,
$\CR(T_\rho)\ge\lim_{r\downarrow1}|r\omega-1|/(r-\rho)=\sqrt3/(1-\rho)\to\infty$ as $\rho\uparrow1$. This explains the operator-dependent constant $B_\delta(T)$ in Lemma~\ref{lem:reduction}.
\end{remark}

Norm-gap conditions are converted into arc-resolvent conditions by the following divided-difference argument from \cite{BGT}. Let $A^{1,1}(\D)$ be the algebra of $w(z)=\sum_{k\ge0}c_kz^k$ with $\|w\|_{A^{1,1}}:=\sum_k(k+1)|c_k|<\infty$; for a closed arc $\arc\subset\T$ put $m_{\arc}(w):=\min_{\eta\in\arc}|w(\eta)|$ and $\|w\|_\infty:=\max_{\T}|w|$.

\begin{lemma}[{\cite[Proposition~5.7]{BGT}}]\label{lem:divdiff}
Let $S\in\B(X)$ be power bounded, $w\in A^{1,1}(\D)$, and let $\arc\subset\T$ be a closed arc with $\|w(S)\|<m_{\arc}(w)$. Then $\arc\subset\rho(S)$ and
\[
   \sup_{\eta\in\arc}\|(\eta-S)^{-1}\|\le\frac{M(S)\|w\|_{A^{1,1}}}{m_{\arc}(w)-\|w(S)\|}.
\]
\end{lemma}

\begin{proof}
We give a short proof for completeness. For $\eta\in\T$ let $h_\eta(z)=(w(\eta)-w(z))/(\eta-z)=\sum_{k\ge1}c_k\sum_{r=0}^{k-1}\eta^{k-1-r}z^r$. The series converges absolutely after evaluation at $S$, with $\|h_\eta(S)\|\le M(S)\sum_{k\ge1}k|c_k|\le M(S)\|w\|_{A^{1,1}}$. The identity $w(\eta)-w(S)=(\eta-S)h_\eta(S)$ holds with commuting factors. If $\eta\in\arc$, the left-hand side is invertible with inverse of norm at most $(m_{\arc}(w)-\|w(S)\|)^{-1}$, so $\eta-S$ is invertible and $(\eta-S)^{-1}=h_\eta(S)(w(\eta)-w(S))^{-1}$.
\end{proof}

\begin{corollary}\label{cor:defect}
Let $T$ be power bounded with $\sigma(T)\cap\T\subset\{1\}$, let $(\mathcal N,\arc)$ have the sampling property, and let $w\in A^{1,1}(\D)$ satisfy $\sup_{n\in\mathcal N}\|w(T^n)\|<m_{\arc}(w)$. Then $T$ is a Ritt operator.
\end{corollary}

\begin{proof}
Apply Lemma~\ref{lem:divdiff} to $S=T^n$, $n\in\mathcal N$; note that $M(T^n)\le M(T)$. Then apply Proposition~\ref{thm:general}.
\end{proof}

With $\mathcal N=\N$ and $w(z)=1-z$, and using Lemma~\ref{lem:allpowers}, one recovers the discrete zero--two law \cite[Corollary~6.1]{BGT}: if $\sigma(T)\subset\D\cup\{1\}$ and $\sup_n\|\Id-T^n\|\le q<2$, then $T$ is Ritt; indeed $m_{\arc}(1-z)=a$ for $\arc=\{\eta\in\T:|\eta-1|\ge a\}$, an arc of positive length for every $a\in(q,2)$. In the next section we examine what happens without the spectral hypothesis.

\section{Iterates close to the identity: all powers}\label{sec:allpowers}

Until the general-symbol extension in Section~\ref{sec:generalsymbol}, we assume that $T\in\B(X)$ satisfies
\begin{equation}\label{eq:q}
   \sup_{n\ge1}\|\Id-T^n\|\le q<2 .
\end{equation}
Condition \eqref{eq:q} gives $\|T^n\|\le1+q$ for all $n$, so
$M(T)\le1+q$ and $\sigma(T)\subset\ol\D$.
If $q<1$, the spectral mapping theorem gives $|1-\lambda^n|\le q$ for every
$\lambda\in\sigma(T)$ and every $n\ge1$. Values with $|\lambda|<1$
are excluded by letting $n\to\infty$, and Lemma~\ref{lem:orbit}
excludes every point of $\T\setminus\{1\}$. Thus $\sigma(T)=\{1\}$.
Moreover, the Neumann series gives
$\|T^{-n}\|\le(1-q)^{-1}$ for all $n\ge1$.
Gelfand's theorem therefore implies $T=\Id$; see \cite{BDG}.
This conclusion also follows from the stronger sampled statement in
Lemma~\ref{lem:subone}, whose proof does not require power boundedness.
Hence it remains to consider $1\le q<2$.

\subsection{Local structure}
\begin{theorem}\label{thm:local}
Let $T$ satisfy \eqref{eq:q}. Then:
\begin{enumerate}[label=\textup{(\roman*)}]
\item $\sigma(T)\cap\T\subset E_q$;
\item there are $\delta(q)>0$ and $C(q)<\infty$, depending only on $q$, such that
\[
   |\lambda-1|\,\|\Res(\lambda,T)\|\le C(q)\qquad(|\lambda|>1,\ |\lambda-1|<\delta(q));
\]
one may take, with $t_q:=2\arcsin\frac{q+2}4\in(0,\pi)$ and $a_q:=\max\{1,\,2\pi/\log(1+\frac{2-q}4)\}$,
\begin{equation}\label{eq:localconstants}
   \delta(q)=\min\Big\{2-\frac{2t_q}\pi,\ \frac12\Big\},\qquad
   C(q)=\max\Big\{(1+2a_q)(1+q),\ \frac{6\pi(1+q)(6-q)}{2-q}\Big\};
\end{equation}
these functions are respectively nonincreasing and nondecreasing in $q$;
\item $T$ is a Ritt operator if and only if $\sigma(T)\cap\T\subset\{1\}$.
\end{enumerate}
\end{theorem}

\begin{proof}
(i) If $\lambda\in\sigma(T)\cap\T$, then $1-\lambda^n\in\sigma(\Id-T^n)$ and $|1-\lambda^n|\le\|\Id-T^n\|\le q$ for all $n$, i.e.\ $\od(\lambda)\le q$.

(ii) Fix $a:=(q+2)/2\in(q,2)$ and let $\arc:=\{\eta\in\T:|\eta-1|\ge a\}=\{e^{it}:t_q\le t\le2\pi-t_q\}$, since $|e^{it}-1|=2\sin(t/2)\ge a$ exactly when $t_q\le t\le2\pi-t_q$. For $\eta\in\arc$,
$\eta-T^n=(\eta-1)\Id+(\Id-T^n)$ with $\|\Id-T^n\|\le q<a\le|\eta-1|$, so the Neumann series gives
$\arc\subset\rho(T^n)$ and
$\|(\eta-T^n)^{-1}\|\le(a-q)^{-1}=2/(2-q)=:K$, uniformly in $n$. By Lemma~\ref{lem:allpowers}, $(\N,\arc)$ has the sampling property with constants $(2\pi-2t_q,2\pi)$. Lemma~\ref{lem:master} with $M=1+q$, $B=2\pi$ and $\frac1{2K}=\frac{2-q}4$ gives the estimate with $\delta=\min\{(2\pi-2t_q)/\pi,\frac12\}$ and $C=\max\{(1+2a_q)(1+q),3(2K+1)(1+q)2\pi\}$, where $2K+1=(6-q)/(2-q)$; this is \eqref{eq:localconstants}. As $q$ increases, $t_q$, $K$ and $a_q$ increase, so $\delta(q)$ decreases and $C(q)$ increases.

(iii) Ritt operators satisfy $\sigma(T)\cap\T\subset\{1\}$. Conversely, Lemma~\ref{lem:reduction} applies.
\end{proof}

\subsection{The structure theorem}
We now determine the structure of $T$ for $1\le q<2$ without any spectral assumption. We need a uniform version of Lemma~\ref{lem:orbit} away from $E_q$.

\begin{lemma}[uniform escape]\label{lem:uniformescape}
Let $1\le q<2$ and $\varepsilon>0$. There are $a=a(q,\varepsilon)>q$, $n_0=n_0(q,\varepsilon)\in\N$ and $\kappa=\kappa(q,\varepsilon)>0$ such that for every $\lambda=re^{i\theta}$ with $1<r\le1+\kappa$ and $\dist(e^{i\theta},E_q)\ge\varepsilon$ there is $n\le n_0$ with $|\lambda^n-1|\ge a$.
\end{lemma}

\begin{proof}
The set $F:=\{\eta\in\T:\dist(\eta,E_q)\ge\varepsilon\}$ is compact. If $F$ is empty there is nothing to prove. Otherwise, note that $\od=\sup_n|\eta^n-1|$ is lower semicontinuous on $\T$ as a supremum of continuous functions. Since $F\cap E_q=\emptyset$, $\od>q$ on $F$, so $a':=\min_F\od>q$. Put $a_1:=(q+a')/2$. For each $\eta\in F$ choose $n(\eta)$ with $|\eta^{n(\eta)}-1|>a_1$; the sets $\{\eta':|\eta'^{\,n(\eta)}-1|>a_1\}$ are open and cover $F$, so finitely many suffice, and we let $n_0$ be the largest of the corresponding $n(\eta)$. Choose $\kappa>0$ with $(1+\kappa)^{n_0}-1\le(a_1-q)/2$ and put $a:=(a_1+q)/2$. If $\lambda=re^{i\theta}$ with $1<r\le1+\kappa$ and $e^{i\theta}\in F$, pick $n\le n_0$ with $|e^{in\theta}-1|>a_1$; then $|\lambda^n-e^{in\theta}|=r^n-1\le(a_1-q)/2$, so $|\lambda^n-1|\ge a$.
\end{proof}

\begin{theorem}\label{thm:oddN}
Let $T$ satisfy \eqref{eq:q}, $1\le q<2$, and let $N=N(q)$. Then $T$ is a $\Ritt{E_q}$ operator and $T^N$ is a Ritt operator, with
\[
   K_{E_q}(T)\le\Phi(q),\qquad \CR(T^N)\le N\Phi(q),
\]
where $\Phi(q)$ depends only on $q$. In particular, if $\sup_n\|\Id-T^n\|<2$, then $T^N$ is a Ritt operator for some odd $N$; equivalently, $T$ is a $\Ritt{E}$ operator for a finite nonempty set $E$ of roots of unity of odd order.
\end{theorem}

\begin{proof}
Put $M:=1+q\ge M(T)$ and $S:=T^N$. Since $E_q\subset\mu_N$, Theorem~\ref{thm:local}(i) and the spectral mapping theorem give $\sigma(S)\cap\T=\{\lambda^N:\lambda\in\sigma(T)\cap\T\}\subset\{1\}$, and $\sup_n\|\Id-S^n\|\le q$. Let $\delta_0=\delta(q)$, $C_0=C(q)$ be the constants of Theorem~\ref{thm:local}(ii) applied to $S$.

\emph{Step 1: estimates near $E_q$.} Put $\delta_1:=\min\{\delta_0/(N2^{N-1}),\ \sin(\pi/N)\}$ if $N\ge2$ and $\delta_1:=\min\{\delta_0/(N2^{N-1}),\ 1\}$ if $N=1$. Let $\xi\in E_q$ and $|\lambda|>1$ with $|\lambda-\xi|<\delta_1$; write $\lambda=\xi w$, so $|w|>1$, $|w-1|<\delta_1\le1$, $|w|<2$ and $\lambda^N=w^N$. Since $|w^N-1|=|w-1|\,|\sum_{j<N}w^j|\le N2^{N-1}|w-1|<\delta_0$, Theorem~\ref{thm:local}(ii) for $S$ gives $\|\Res(\lambda^N,S)\|\le C_0/|\lambda^N-1|$. By Lemma~\ref{lem:factor}, $\|\Res(\lambda,T)\|\le MN2^{N-1}C_0/|\lambda^N-1|$. For $N\ge2$, the inequalities $|\lambda-\xi|<\sin(\pi/N)$ and $|\xi-\eta|\ge2\sin(\pi/N)$ for distinct $\xi,\eta\in\mu_N$ show that $\xi$ is the nearest point of $\mu_N$ to $\lambda$; for $N=1$ this is automatic. Lemma~\ref{lem:scalarroots} therefore gives $|\lambda^N-1|\ge c_N|\lambda-\xi|$. Hence
\[
   \dist(\lambda,E_q)\,\|\Res(\lambda,T)\|\le|\lambda-\xi|\,\|\Res(\lambda,T)\|\le C_1:=\frac{MN2^{N-1}C_0}{c_N}.
\]

\emph{Step 2: estimates away from $E_q$.} Let $a,n_0,\kappa$ be given by Lemma~\ref{lem:uniformescape} with $\varepsilon=\delta_1/2$, and assume also $\kappa\le\delta_1/2$. Let $\lambda=re^{i\theta}$, $|\lambda|>1$, $\dist(\lambda,E_q)\ge\delta_1$. If $r\ge1+\kappa$, the Neumann series gives $\|\Res(\lambda,T)\|\le M/(r-1)$ and $\dist(\lambda,E_q)\le r+1\le\frac{2+\kappa}{\kappa}(r-1)$. If $1<r<1+\kappa$, then $\dist(e^{i\theta},E_q)\ge\delta_1-(r-1)\ge\delta_1/2$, so there is $n\le n_0$ with $|\lambda^n-1|\ge a>q$. Since $\|\Id-T^n\|\le q$, the operator $\lambda^n-T^n=(\lambda^n-1)\Id+(\Id-T^n)$ is invertible with $\|(\lambda^n-T^n)^{-1}\|\le(a-q)^{-1}$, and Lemma~\ref{lem:factor} gives $\|\Res(\lambda,T)\|\le Mn_02^{n_0}/(a-q)$, while $\dist(\lambda,E_q)\le3$.

Together, $K_{E_q}(T)\le\Phi(q):=\max\{C_1,\ M(2+\kappa)/\kappa,\ 3Mn_02^{n_0}/(a-q)\}$, which depends only on $q$. Since $\dist(\lambda,\mu_N)\le\dist(\lambda,E_q)$, also $K_N(T)\le\Phi(q)$, and Proposition~\ref{prop:TN} shows that $T^N$ is Ritt with $\CR(T^N)\le N\Phi(q)$.

For the last assertion: if $T^N$ is Ritt with $N$ odd, then $T$ is $\Ritt{\mu_N}$ by Proposition~\ref{prop:TN}, and $\mu_N$ consists of roots of unity of odd order; conversely, if $T$ is $\Ritt{E}$ with $E$ a finite nonempty set of roots of unity of odd order and $N$ is the least common multiple of their orders, then $N$ is odd, $E\subset\mu_N$, $\dist(\lambda,\mu_N)\le\dist(\lambda,E)$, so $T$ is $\Ritt{\mu_N}$ and $T^N$ is Ritt by Proposition~\ref{prop:TN}.
\end{proof}

\begin{remarks}
(a) Theorem~\ref{thm:oddN} shows that the spectral hypothesis $\sigma(T)\cap\T\subset\{1\}$ in the discrete zero--two law \cite[Corollary~6.1]{BGT} is exactly the hypothesis that excludes the finitely many odd-order roots of unity in $E_q$; for $q<\sqrt3$ it is automatic (Theorem~\ref{thm:sqrt3}).

(b) The only non-explicit step in the proof is the compactness argument of Lemma~\ref{lem:uniformescape}; all the other constants are explicit. For $q<\sqrt3$ the compactness can be replaced by an explicit hitting argument, which gives a reciprocal bound (Theorem~\ref{thm:sqrt3}); the sharp leading constant is determined in Theorem~\ref{thm:sharp-asymptotic}.

(c) The parity restriction concerns the converse (Theorem~\ref{thm:UCchar}) only: $T=-\Id$ has $T^2=\Id$ Ritt while $\|\Id-T\|=2$. In the forward implication every multiple of $N(q)$, even ones included, is an admissible exponent, as the next proposition shows.  
\end{remarks}

\begin{proposition}[the exact universal exponent]\label{prop:minimalexponent}
Fix $1\le q<2$ and $L\ge1$. The implication
\[
 \sup_{n\ge1}\|\Id-T^n\|\le q\quad\Longrightarrow\quad T^L\text{ is Ritt}
\]
holds for every bounded operator on every complex Banach space if and only if
$N(q)$ divides $L$.
\end{proposition}
\begin{proof}
If $L$ is a multiple of $N(q)$, apply Theorem~\ref{thm:oddN} and the stability
of the Ritt property under positive integer powers (Theorem~\ref{thm:BGTpowers}). Conversely, for each
$\xi\in E_q$, apply the asserted implication to the scalar operator $T=\xi$, which satisfies \eqref{eq:q} by definition of $E_q$.
A unimodular scalar is Ritt only when it is $1$, so $\xi^L=1$. Thus every
order occurring in $E_q$ divides $L$, and so does their least common multiple.
\end{proof}

\begin{corollary}\label{cor:cubic}
Let $\sqrt3\le q<2\cos\frac{\pi}{10}\approx1.902$ and $\sup_n\|\Id-T^n\|\le q$. Then $T^3$ is a Ritt operator, $\sigma(T)\cap\T\subset\mu_3$, and $\|\Res(\lambda,T)\|\le\Phi(q)/\dist(\lambda,\mu_3)$ for $|\lambda|>1$.
\end{corollary}

\subsection{The threshold \texorpdfstring{$\sqrt3$}{sqrt3} and the optimal asymptotic constant}
For $q<\sqrt3$, Lemma~\ref{lem:orbit} gives $E_q=\{1\}$ and $N(q)=1$, so Theorem~\ref{thm:oddN} implies that $T$ itself is Ritt. Recall that $\Phi_*(q)$ is the optimal universal bound defined in \eqref{eq:optimalPhi}. We first give bounds valid throughout $1\le q<\sqrt3$, and then determine the leading asymptotic constant.

\begin{theorem}\label{thm:sqrt3}
Let $1\le q<\sqrt3$ and $\sup_{n\ge1}\|\Id-T^n\|\le q$. Then $T$ is a Ritt operator and
\begin{equation}\label{eq:rate}
 \CR(T)\le\frac{2\sqrt3 M(T)}{\sqrt3-q}
          \le\frac{6+2\sqrt3}{\sqrt3-q}.
\end{equation}
For $0\le r<1$, put
\[
 a_r=r\exp\bigl(i\arccos(-r/2)\bigr).
\]
Then
\begin{equation}\label{eq:qrho}
 q_r:=\sup_{n\ge1}|1-a_r^n|=\sqrt{1+2r^2},
 \qquad
 \CR(a_r)=\frac{2q_r}{1-r^2}=\frac{4q_r}{3-q_r^2}.
\end{equation}
Consequently,
\begin{equation}\label{eq:sharp-scalar-lower}
 \frac{4q}{3-q^2}\le\Phi_*(q)
             \le\frac{6+2\sqrt3}{\sqrt3-q}
 \qquad(1\le q<\sqrt3).
\end{equation}
The threshold $\sqrt3$ is optimal: $T=e^{2\pi i/3}\Id$ satisfies
$\sup_n\|\Id-T^n\|=\sqrt3$ and is not Ritt. For every $q\in[\sqrt3,2)$, the Ritt constants of operators satisfying $\sup_n\|\Id-T^n\|\le q$ are unbounded, even for scalar operators with spectrum in $\D$.
\end{theorem}

\begin{lemma}[explicit escape]\label{lem:explicitescape}
For every $\xi\in\T\setminus\{1\}$ there is an integer $n\ge1$ such that
\[
 |\xi^n-1|\ge\sqrt3,
 \qquad n|\xi-1|\le2\sqrt3.
\]
\end{lemma}

\begin{proof}
Write $\xi=e^{i\theta}$ with $0<|\theta|\le\pi$. If $|\theta|\ge2\pi/3$, take $n=1$. Otherwise take $n=\lceil2\pi/(3|\theta|)\rceil$. Then
\[
 \frac{2\pi}3\le n|\theta|<\frac{2\pi}3+|\theta|<\frac{4\pi}3,
\]
so $|\xi^n-1|\ge\sqrt3$. If $n=2$, then
$n|\xi-1|=4\sin(|\theta|/2)\le2\sqrt3$. If $n\ge3$, the minimality of $n$ gives
\[
 n|\xi-1|\le n|\theta|<\frac{2\pi n}{3(n-1)}\le\pi<2\sqrt3.
\]
For $n=1$, use $|\xi-1|\le2$.
\end{proof}

\begin{proof}[Proof of Theorem~\ref{thm:sqrt3}]
The spectral mapping theorem and Lemma~\ref{lem:orbit} give
$\sigma(T)\cap\T\subset\{1\}$, while $M(T)\le1+q$. For
$\xi\in\T\setminus\{1\}$, choose $n$ as in Lemma~\ref{lem:explicitescape}. The identity
\[
 \xi^n-T^n=(\xi^n-1)\Id+(\Id-T^n)
\]
and the Neumann series give
$\|\Res(\xi^n,T^n)\|\le(\sqrt3-q)^{-1}$. Hence Lemma~\ref{lem:factor} yields
\[
 |\xi-1|\,\|\Res(\xi,T)\|
 \le\frac{M(T)n|\xi-1|}{\sqrt3-q}
 \le\frac{2\sqrt3 M(T)}{\sqrt3-q}.
\]
Proposition~\ref{prop:boundary} proves \eqref{eq:rate}.

For the scalar assertions, note that $|a_r|=r$ and
$a_r+\bar a_r=-r^2$. Thus
\[
 |1+a_r|=1,
 \qquad a_r^2+r^2a_r+r^2=0,
 \qquad a_r^{n+3}=r^2a_r^n+(1-r^2)a_r^{n+2}\quad(n\ge0).
\]
The last identity follows by multiplying the quadratic identity by
$a_r-1$ and then by $a_r^n$. Its coefficients are nonnegative and sum to one. Induction therefore shows that every power of $a_r$ belongs to
$\operatorname{conv}\{1,a_r,a_r^2\}$. Moreover,
\[
 |1-a_r|^2=1+2r^2,
 \qquad |1-a_r^2|=|1-a_r|\,|1+a_r|=|1-a_r|.
\]
Convexity of $z\mapsto|1-z|$ proves the first identity in \eqref{eq:qrho}.

For every $a\in\D$, the substitution $w=1/\lambda$ gives
\begin{equation}\label{eq:scalar-ritt-exact}
 \CR(a)=\sup_{|w|<1}\left|\frac{1-w}{1-aw}\right|
        =\frac{2|1-a|}{1-|a|^2}.
\end{equation}
Indeed, the displayed M\"obius map has image the disc with centre
$(1-\bar a)/(1-|a|^2)$ and radius $|1-a|/(1-|a|^2)$; the modulus of the centre equals the radius. Applying this formula to $a_r$ proves the remaining identities in \eqref{eq:qrho}. Since $r\mapsto q_r$ maps $[0,1)$ onto $[1,\sqrt3)$, it also proves \eqref{eq:sharp-scalar-lower}. As $r\uparrow1$, the operators $a_r$ have unbounded Ritt constants and displacement bounds below $\sqrt3$. Finally, the orbit of $e^{2\pi i/3}$ has displacement $\sqrt3$, and this scalar operator has a peripheral spectral point different from $1$.
\end{proof}

\begin{theorem}\label{thm:sharp-asymptotic}
As $q\uparrow\sqrt3$,
\begin{equation}\label{eq:sharp-asymptotic}
 \Phi_*(q)=\frac{2}{\sqrt3-q}+O(1).
\end{equation}
In particular,
\[
 \lim_{q\uparrow\sqrt3}(\sqrt3-q)\Phi_*(q)=2.
\]
\end{theorem}

\begin{proof}
The lower bound follows from \eqref{eq:sharp-scalar-lower}, since
\begin{equation}\label{eq:lower-asymptotic}
 \frac{4q}{3-q^2}
 =\frac{2}{\sqrt3-q}-\frac1{\sqrt3}+O(\sqrt3-q).
\end{equation}
For the upper bound, let $\varepsilon=\sqrt3-q$,
$M_0=1+\sqrt3$, and suppose
$\sup_n\|\Id-T^n\|\le q$. Then $M(T)\le M_0$. All constants below are independent of $T$, $X$ and $\varepsilon$.

Put $\omega=e^{2\pi i/3}$ and define
\[
 p(z)=\frac23e^{i\pi/6}(1-z)
             +\frac13e^{-i\pi/6}(1-z^2).
\]
The displacement bounds for the first two powers give $\|p(T)\|\le q$. Direct calculation yields
\[
 p(\omega)=\sqrt3,
 \qquad p'(\omega)=-\frac1{\sqrt3}-i,
 \qquad |p'(\omega)|=\frac2{\sqrt3},
 \qquad \operatorname{Re}\bigl(i\omega p'(\omega)\bigr)=0.
\]
Write $\xi(\delta)=\omega e^{i\delta}$. The last identity says that the derivative of $|p(\xi(\delta))|$ at $\delta=0$ vanishes. Taylor's theorem therefore gives constants $A>0$ and $\delta_0>0$ such that
\begin{equation}\label{eq:quadratic-contact}
 |p(\xi(\delta))|\ge\sqrt3-A\delta^2,
 \qquad
 |\xi(\delta)-1|\,|p'(\xi(\delta))|\le2+A|\delta|
 \quad(|\delta|\le\delta_0).
\end{equation}
The two roots of $p(z)=p(\xi)$ are $\xi$ and
\[
 \eta=-2e^{i\pi/3}-\xi.
\]
At $\xi=\omega$ the second root has modulus $\sqrt7$.
By decreasing $\delta_0$, we may assume $|\eta|\ge2$ whenever
$|\delta|\le\delta_0$. Thus $\eta\in\rho(T)$ and
\begin{equation}\label{eq:second-root-bound}
 \|\Res(\eta,T)\|\le\frac{M(T)}{|\eta|-1}\le M_0.
\end{equation}
Whenever $|p(\xi)|>q$, the partial fraction decomposition give the exact identity
\begin{equation}\label{eq:quadratic-resolvent}
 \Res(\xi,T)
 =p'(\xi)\bigl(p(\xi)\Id-p(T)\bigr)^{-1}
       +\Res(\eta,T).
\end{equation}
To verify it, write $p(z)=az^2+bz+c$, where
$a=-e^{-i\pi/6}/3$. Then
$p(\xi)-p(z)=-a(\xi-z)(\eta-z)$ and
$p'(\xi)=a(\xi-\eta)$, which give
$p'(\xi)/(p(\xi)-p(z))=(\xi-z)^{-1}-(\eta-z)^{-1}$.

Fix $K\ge1$ for the moment. For $|\delta|\le K\varepsilon$ and sufficiently small $\varepsilon$, depending only on $K$, equations
\eqref{eq:quadratic-contact}--\eqref{eq:quadratic-resolvent} imply
\begin{align*}
 |\xi-1|\,\|\Res(\xi,T)\|
 &\le\frac{2+AK\varepsilon}
              {\varepsilon-AK^2\varepsilon^2}+2M_0\\
 &\le\frac2\varepsilon+C_K.
\end{align*}
Here $C_K$ is finite and independent of $T$ and $\varepsilon$; the second inequality follows, for example, by requiring
$AK^2\varepsilon\le1/2$ and expanding $(1-AK^2\varepsilon)^{-1}$.

We next control the rest of the fixed arc around $\omega$. After decreasing $\delta_0$ again, there is $c>0$ such that
\[
 |1-\xi(\delta)|\ge\sqrt3+c\delta
       \quad(0\le\delta\le\delta_0),
 \qquad
 |1-\xi(\delta)^2|\ge\sqrt3+c|\delta|
       \quad(-\delta_0\le\delta\le0).
\]
Indeed, the relevant functions are
$2\sin(\pi/3+\delta/2)$ and $2\sin(2\pi/3+\delta)$, whose derivatives at zero are $1/2$ and $-1$, respectively. Applying the Neumann series to the first power in the first case and to the second power in the second case, and then using Lemma~\ref{lem:factor}, gives
\[
 |\xi-1|\,\|\Res(\xi,T)\|
 \le\frac{4M_0}{\varepsilon+c|\delta|}
       \quad(|\delta|\le\delta_0).
\]
Choose $K$ once and for all so that $4M_0/(1+cK)\le1$. Then the last bound is at most $1/\varepsilon$ whenever
$K\varepsilon\le|\delta|\le\delta_0$. Together with the previous estimate, this proves
\begin{equation}\label{eq:critical-arc-bound}
 |\xi-1|\,\|\Res(\xi,T)\|\le\frac2\varepsilon+C
\end{equation}
on a fixed arc about $\omega$, with an absolute $C$. The polynomial
$\widetilde p(z)=\overline{p(\bar z)}$ gives the same bound on a fixed arc about $\bar\omega$.

To complete the proof, observe that the remaining part of the circle has a uniform bound independent of $\varepsilon$. Apply Theorem~\ref{thm:local}(ii) with displacement bound $\sqrt3$. Taking radial limits in that estimate gives
\[
 |\xi-1|\,\|\Res(\xi,T)\|\le C(\sqrt3)
 \quad\bigl(0<|\xi-1|<\delta(\sqrt3)/2\bigr).
\]
Choose disjoint open arcs about $1$, $\omega$ and $\bar\omega$, with closures contained in the regions already controlled, and let $H$ be their complement in $\T$. Every $\zeta\in H$ has
$\od(\zeta)>\sqrt3$ by Lemma~\ref{lem:orbit}. For each such $\zeta$, choose an integer $n(\zeta)$ and an open neighbourhood on which
$|1-\xi^{n(\zeta)}|>\sqrt3$. Finitely many of these neighbourhoods cover the compact set $H$. If $N$ is the largest of the corresponding exponents, continuity gives
\[
 g:=\min_{\xi\in H}\max_{1\le n\le N}|1-\xi^n|>\sqrt3.
\]
The Neumann series and Lemma~\ref{lem:factor} now give
\[
 |\xi-1|\,\|\Res(\xi,T)\|
 \le\frac{2M_0N}{g-\sqrt3}
 \qquad(\xi\in H).
\]
Combining this estimate with \eqref{eq:critical-arc-bound} and the estimate near $1$, and taking $\varepsilon$ sufficiently small, proves
\[
 \sup_{\xi\in\T\setminus\{1\}}
       |\xi-1|\,\|\Res(\xi,T)\|
 \le\frac2\varepsilon+C'
\]
with an absolute $C'$. Proposition~\ref{prop:boundary} and
\eqref{eq:lower-asymptotic} complete the proof.
\end{proof}

\subsection{Resolvent growth at the higher displacement thresholds}
For an odd integer $m\ge3$, put
\[
 Q_m=2\cos\frac{\pi}{2m},
 \qquad
 E_m^-:=\{\xi\in\T:\ \xi\text{ has odd order strictly less than }m\}.
\]
If $m=3$, set $Q_m^-=1$; if $m\ge5$, set $Q_m^-=Q_{m-2}$. Then
$E_q=E_m^-$ for $Q_m^-\le q<Q_m$, whereas
$E_{Q_m}\setminus E_m^-$ consists of the roots of exact order $m$.

\begin{theorem}\label{thm:higher-thresholds}
Let $m\ge3$ be odd, $Q=Q_m$, and $Q_m^-\le q<Q$. Every operator satisfying
$\sup_{n\ge1}\|\Id-T^n\|\le q$ obeys
\begin{equation}\label{eq:higher-threshold-bound}
 K_{E_q}(T)\le\frac{6m(1+Q)\Phi(Q)}{Q-q},
\end{equation}
where $\Phi$ is the uniform bound in Theorem~\ref{thm:oddN}. Moreover, there are constants $c_m>0$ and $\varepsilon_m>0$ such that
\begin{equation}\label{eq:higher-threshold-lower}
 \sup\left\{K_{E_q}(a):\ a\in\D,\
                  \sup_{n\ge1}|1-a^n|\le q\right\}
 \ge\frac{c_m}{Q-q}
 \qquad(Q-\varepsilon_m<q<Q).
\end{equation}
Thus the reciprocal order in \eqref{eq:higher-threshold-bound} is sharp at every threshold.
\end{theorem}

\begin{proof}
Put $E=E_m^-$, $F=E_Q$, $\varepsilon=Q-q$ and $M=1+q$. Theorem~\ref{thm:oddN}, applied at the bound $Q$, gives
$K_F(T)\le\Phi(Q)$. For each $\xi\in F\setminus E$, the orbit of $\xi$ contains a point at distance $Q$ from $1$, so there is $1\le n\le m$ with $|1-\xi^n|=Q$. The Neumann series and Lemma~\ref{lem:factor} imply
\[
 \xi\in\rho(T),\qquad
 \|\Res(\xi,T)\|\le\frac{nM}{\varepsilon}
                         \le\frac{mM}{\varepsilon}.
\]
Let $h=\varepsilon/(2mM)$. If $|\lambda-\xi|\le h$, resolvent perturbation gives
\[
 \|\Res(\lambda,T)\|\le\frac{2mM}{\varepsilon}.
\]
Since $E$ is nonempty, $\dist(\lambda,E)\le h+2\le3$ on each of these balls. Hence
\[
 \dist(\lambda,E)\,\|\Res(\lambda,T)\|
 \le\frac{6mM}{\varepsilon}
\]
there. Outside their union, if a nearest point of $F$ to $\lambda$ belongs to $E$, then
$\dist(\lambda,E)=\dist(\lambda,F)$. Otherwise,
$\dist(\lambda,F)\ge h$, and comparison with any point of $E$ gives
\[
 \dist(\lambda,E)\le\dist(\lambda,F)+2
      \le(1+2/h)\dist(\lambda,F).
\]
Because $\varepsilon<1$ and $mM\ge1$,
$1+2/h\le6mM/\varepsilon$. Thus, for all $|\lambda|>1$ outside the balls,
\[
 \dist(\lambda,E)\,\|\Res(\lambda,T)\|
 \le\frac{6mM\Phi(Q)}{\varepsilon}.
\]
The constant $\Phi(Q)$ is at least $1$, as follows by letting $|\lambda|\to\infty$ in the defining resolvent bound. Using $M\le1+Q$ proves \eqref{eq:higher-threshold-bound}.

For the lower bound, fix a root $\xi$ of exact order $m$ and let $a_r=r\xi$, where $r<1$ tends to $1$. For $1\le j\le m$, set
$f_j(r)=|1-r^j\xi^j|$. If $n\equiv j\pmod m$ and $n\ge j$, then $r^n\xi^j$ lies on the segment joining $0$ to $r^j\xi^j$; hence
$|1-r^n\xi^j|\le\max\{1,f_j(r)\}$. Since $a_r^n\to0$ as $n\to\infty$, this proves
\[
 q(r):=\sup_{n\ge1}|1-a_r^n|
      =\max\{1,f_1(r),\ldots,f_m(r)\}.
\]
At $r=1$, precisely two indices $j_1,j_2\in\{1,\ldots,m-1\}$ satisfy $f_j(1)=Q$; they obey $j_1+j_2=m$. All other values are strictly smaller than $Q$. For either of these two indices,
\[
 f_j'(1)=\frac{j(1-\operatorname{Re}\xi^j)}{|1-\xi^j|}
             =\frac{jQ}{2}.
\]
Consequently, with $j_*:=\min\{j_1,j_2\}$, for all $r<1$ sufficiently close to $1$ one has $q(r)=f_{j_*}(r)$ and
\begin{equation}\label{eq:higher-scalar-asymptotic}
 Q-q(r)=\frac{j_*Q}{2}(1-r)+O((1-r)^2).
\end{equation}
In particular, $q(r)<Q$ and $q(r)$ is strictly increasing on an interval ending at $1$, so every $q$ sufficiently close to $Q$ from below equals some $q(r)$.
Put $d=\dist(\xi,E)>0$. By approaching $\xi$ radially from outside the unit disc,
\[
 K_E(a_r)\ge\frac{d}{1-r}.
\]
Together with \eqref{eq:higher-scalar-asymptotic}, this gives
\[
 \liminf_{r\uparrow1}(Q-q(r))K_E(a_r)
       \ge\frac{dj_*Q}{2}>0,
\]
which proves \eqref{eq:higher-threshold-lower}.
\end{proof}

\subsection{General symbols and finite peripheral spectrum}\label{sec:generalsymbol}

The displacement symbol $w(z)=1-z$ forces the peripheral spectral points
to have odd order. For general symbols the same power-reduction argument
gives a finite set of roots of unity, without a parity restriction.

\begin{theorem}[structure under a general norm gap]\label{thm:generalsymbolstructure}
Let $T\in\B(X)$ be power bounded, let $w\in A^{1,1}(\D)$, and suppose that
\[
   \sup_{n\ge1}\|w(T^n)\|\le q<\|w\|_\infty.
\]
Define
\[
   E=E(w,q):=\Big\{\lambda\in\T:
                  \sup_{n\ge1}|w(\lambda^n)|\le q\Big\}.
\]
Then $E$ is a finite set of roots of unity and
$\sigma(T)\cap\T\subset E$. More precisely:
\begin{enumerate}[label=\textup{(\roman*)}]
\item If $E=\emptyset$, then $r(T)<1$.
\item If $E\ne\emptyset$, then $1\in E$. If $N$ is the least common
multiple of the orders of the elements of $E$, then $T^N$ is a Ritt
operator and $T$ is a $\Ritt{E}$ operator.
\end{enumerate}
\end{theorem}

\begin{proof}
If $\lambda\in\T$ is not a root of unity, its positive powers are dense
in $\T$, so continuity gives
$\sup_{n\ge1}|w(\lambda^n)|=\|w\|_\infty>q$. If $\lambda$ has order $m$,
then
\[
   \sup_{n\ge1}|w(\lambda^n)|=\max_{\eta\in\mu_m}|w(\eta)|.
\]
Uniform continuity of $w$ on $\T$ shows that the right-hand side tends
to $\|w\|_\infty$ as $m\to\infty$. Thus the orders occurring in $E$
are bounded, and $E$ is finite. If $E$ contains an element of order $m$,
then $|w(1)|=|w(\lambda^m)|\le q$, whence $1\in E$.

For the peripheral spectrum inclusion, let $\lambda\in\sigma(T)\cap\T$
and fix $n\ge1$. The spectral mapping theorem gives
$\eta:=\lambda^n\in\sigma(T^n)\cap\T$. The divided-difference identity
from Lemma~\ref{lem:divdiff}, with $S=T^n$, is
\[
   w(\eta)\Id-w(S)=(\eta\Id-S)h_\eta(S).
\]
The series defining $h_\eta(S)$ converges in operator norm because
$w\in A^{1,1}(\D)$ and $S$ is power bounded. The two factors on the
right commute. If their product were invertible, then
$\eta\Id-S$ would also be invertible, a contradiction. Consequently
\[
   w(\lambda^n)\in\sigma(w(T^n)),\qquad
   |w(\lambda^n)|\le\|w(T^n)\|\le q.
\]
Since $n$ was arbitrary, $\lambda\in E$. If $E=\emptyset$, compactness
of $\sigma(T)\subset\ol\D$ gives $r(T)<1$.

Suppose that $E\ne\emptyset$, and put $S=T^N$. Then $S$ is
power bounded and 
\[
   \sigma(S)\cap\T
      =\{\lambda^N:\lambda\in\sigma(T)\cap\T\}
      \subset\{1\},
   \qquad
   \sup_{n\ge1}\|w(S^n)\|\le q.
\]
Since $|w(1)|\le q<\|w\|_\infty$, there is a closed arc
$\arc\subset\T\setminus\{1\}$ of positive length such that
$m_\arc(w)>q$. By Lemma~\ref{lem:allpowers}, $(\N,\arc)$ has the
sampling property. Corollary~\ref{cor:defect} shows that $S$
is Ritt, and Proposition~\ref{prop:TN} shows that $T$ is
$\Ritt{\mu_N}$.

To remove the points of $\mu_N\setminus E$, choose $0<\delta<1$
smaller than half the distance between distinct points of $\mu_N$
when $N>1$; for $N=1$ choose any $\delta\in(0,1)$. If
$|\lambda|>1$ and $\dist(\lambda,E)<\delta$, a nearest point of
$\mu_N$ belongs to $E$, and hence
\[
   \dist(\lambda,E)\|\Res(\lambda,T)\|
      =\dist(\lambda,\mu_N)\|\Res(\lambda,T)\|
      \le K_N(T).
\]
The compact set
\[
   \{\lambda:1\le|\lambda|\le2,
                   \ \dist(\lambda,E)\ge\delta\}
\]
is disjoint from $\sigma(T)$, so the resolvent is bounded there.
For $|\lambda|\ge2$, using the power-boundedness of $T$ the Neumann series gives
\[
   \dist(\lambda,E)\|\Res(\lambda,T)\|
      \le\frac{(|\lambda|+1)M(T)}{|\lambda|-1}
      \le3M(T).
\]
Thus $K_E(T)<\infty$.
\end{proof}

The uniform estimate can be expressed through finitely many scalar
escape inequalities. We use angular distance on $\T$, taking values
in $[0,\pi]$, and write
\[
   L_w:=\sum_{k\ge1}k|c_k|
   \quad\text{when}\quad w(z)=\sum_{k\ge0}c_kz^k.
\]

\begin{theorem}[a finite scalar bound]\label{thm:symbol-constructive}
Under the hypotheses of Theorem~\ref{thm:generalsymbolstructure}, suppose
that $E=E(w,q)\ne\emptyset$, and put $M=M(T)$ and
$D=\max_{\xi\in E}\operatorname{ord}(\xi)$. Choose a closed arc
$\Gamma\subset\T$ of angular length $\ell>0$ and a number $a>q$ such that
\[
   \min_{\eta\in\Gamma}|w(\eta)|\ge a.
\]
Let $U$ be the union of the open angular neighbourhoods of radius
$\ell/(2D)$ about the points of $E$, and set $F_0=\T\setminus U$.
If $F_0\ne\emptyset$, there exists $N\ge1$ such that
\[
   g:=\min_{\eta\in F_0}\max_{1\le n\le N}|w(\eta^n)|-q>0.
\]
For any such $N$ and $g$,
\begin{equation}\label{eq:symbol-constructive}
   K_E(T)\le M+(1+M)ML_w
      \max\left\{\frac{2\pi}{a-q},\frac{2N}{g}\right\}.
\end{equation}
When $F_0=\emptyset$, the second term in the maximum is omitted.
\end{theorem}

\begin{proof}
Since $1\in E$, we have $|w(1)|\le q$, so $1\notin\Gamma$ and
$0<\ell<2\pi$. In particular the angular neighbourhoods used above
have radius less than $\pi$. Every point $\eta\in F_0$ lies outside
$E$, and therefore $|w(\eta^n)|>q$ for some $n\ge1$. The corresponding
open sets cover $F_0$. Compactness gives a finite subcover and hence
an integer $N$ for which the continuous function
$\max_{1\le n\le N}|w(\eta^n)|$ is everywhere greater than $q$ on
$F_0$. Its minimum exceeds $q$, proving the existence of $g$.

For any $n\ge1$, apply the divided-difference argument directly to
$v_n(z)=w(z^n)$. Its weighted coefficient sum is $nL_w$, and
$v_n(T)=w(T^n)$. Consequently, whenever $\eta\in\T$ and
$|w(\eta^n)|>q$,
\begin{equation}\label{eq:symbol-direct-divdiff}
   \|\Res(\eta,T)\|
      \le\frac{MnL_w}{|w(\eta^n)|-q}.
\end{equation}
For $\eta\in F_0$, choose a witnessing $n\le N$. Since
$\dist(\eta,E)\le2$, this gives
\[
   \dist(\eta,E)\|\Res(\eta,T)\|
      \le\frac{2MNL_w}{g}.
\]

Now let $\eta\in U\setminus E$. Choose $\xi\in E$ and write
$\eta=\xi e^{it}$ with $0<|t|<\ell/(2D)$. Put
$d=\operatorname{ord}(\xi)$ and $u=d|t|$, so $0<u<\ell/2$.
If $t>0$, write $\Gamma=\{e^{is}:\alpha\le s\le\beta\}$ with
$0<\alpha<\beta<2\pi$ and $\beta-\alpha=\ell$. The integer
$k=\lceil\alpha/u\rceil$ satisfies
\[
   \alpha\le ku<\alpha+u<\beta,
   \qquad ku<2\pi.
\]
If $t<0$, use the reflected arc $\ol\Gamma$ in the same construction.
In both cases, with $n=dk$ we obtain
\[
   \eta^n=e^{int}\in\Gamma,
   \qquad n\dist(\eta,E)\le n|\eta-\xi|
      \le n|t|=ku<2\pi.
\]
Using \eqref{eq:symbol-direct-divdiff} yields
\[
   \dist(\eta,E)\|\Res(\eta,T)\|
      \le\frac{2\pi ML_w}{a-q}.
\]
Thus the supremum of the boundary resolvent expression on
$\T\setminus E$ is at most
\[
   B:=ML_w\max\left\{\frac{2\pi}{a-q},\frac{2N}{g}\right\},
\]
with the stated convention when $F_0=\emptyset$.

To pass to the exterior, let $\lambda=r\eta$, where $r>1$ and
$\eta\in\T$. If $\eta\notin E$, the resolvent identity and
$\|\Res(\lambda,T)\|\le M/(r-1)$ give
\[
   \|\Res(\lambda,T)\|
      \le\bigl(1+(r-1)\|\Res(\lambda,T)\|\bigr)
             \|\Res(\eta,T)\|
      \le(1+M)\|\Res(\eta,T)\|.
\]
Since $\dist(\lambda,E)\le(r-1)+\dist(\eta,E)$, it follows that
\[
   \dist(\lambda,E)\|\Res(\lambda,T)\|\le M+(1+M)B.
\]
If $\eta\in E$, then $\dist(\lambda,E)=r-1$ and the Neumann estimate
gives the smaller bound $M$. This proves \eqref{eq:symbol-constructive}.
\end{proof}

\begin{remark}\label{rem:symbol-certificate}
The scalar data in Theorem~\ref{thm:symbol-constructive} involve only
finitely many root orders and powers. Indeed, any root of order
$m>2\pi/\ell$ has a power in $\Gamma$, and hence cannot belong to
$E(w,q)$. Thus $E(w,q)$ can be determined by testing the finitely many
orders $m\le\lfloor2\pi/\ell\rfloor$. Having determined $E$, the
search over $N$ in the theorem terminates. For a polynomial symbol,
the remaining minima are minima of explicitly given continuous
functions of one angular variable.
\end{remark}

For a root of unity $\xi$, put
\[
   b_w(\xi):=\max_{1\le n\le\operatorname{ord}(\xi)}|w(\xi^n)|.
\]
We call $Q\in(0,\|w\|_\infty)$ a \emph{peripheral threshold} if
$b_w(\xi)=Q$ for some root of unity $\xi$. Such thresholds are locally
finite below $\|w\|_\infty$, because $E(w,Q')$ is finite for every
$Q'<\|w\|_\infty$.

\begin{theorem}[growth at a peripheral threshold]\label{thm:symbol-threshold}
Let $w\in A^{1,1}(\D)$, let $Q\in(0,\|w\|_\infty)$ be a peripheral
threshold, and fix $M_0\ge1$. Set
\[
   E_+=E(w,Q),\qquad
   E_-:=\{\xi\in E_+:b_w(\xi)<Q\},\qquad
   A=M_0 L_w \max_{\xi\in E_+}\operatorname{ord}(\xi).
\]
Let $C_+=C(w,Q,M_0)$ be the right-hand side of
\eqref{eq:symbol-constructive} obtained at level $Q$, with $M$ replaced
by $M_0$ and any admissible scalar data. There exists $q_0<Q$ such that
$E(w,q)=E_-$ for $q_0\le q<Q$. For every power bounded $T$ satisfying
\[
   M(T)\le M_0,
   \qquad \sup_{n\ge1}\|w(T^n)\|\le q,
   \qquad q_0\le q<Q,
\]
the following estimates hold, after increasing $q_0$ if necessary:
\begin{enumerate}[label=\textup{(\roman*)}]
\item If $E_-\ne\emptyset$, then
\begin{equation}\label{eq:symbol-threshold-bound}
   K_{E_-}(T)\le\frac{6AC_+}{Q-q}.
\end{equation}
\item If $E_-=\emptyset$, then
\begin{equation}\label{eq:symbol-empty-threshold}
   \sup_{|\lambda|\ge1}\|\Res(\lambda,T)\|
       \le\frac{2AC_+}{Q-q},
   \qquad r(T)\le1-\frac{Q-q}{2AC_+}.
\end{equation}
\end{enumerate}
\end{theorem}

\begin{proof}
The finiteness of $E_+$ implies that the numbers $b_w(\xi)<Q$ for
$\xi\in E_+$ have a common upper bound strictly below $Q$. Choosing
$q_0\in[0,Q)$ above that bound gives $E(w,q)=E_-$ for
$q_0\le q<Q$; when $E_-$ is empty there is no lower bound to impose.
Since $Q$ is a threshold strictly below $\|w\|_\infty$, the symbol
$w$ is nonconstant, so $L_w>0$ and $A>0$. Increase $q_0$ so that
$Q-q_0\le2A$.

Put $\varepsilon=Q-q$ and $\delta=\varepsilon/(2A)\le1$.
Theorem~\ref{thm:symbol-constructive}, applied at level $Q$, gives
$K_{E_+}(T)\le C_+$; the displayed formula for $C_+$ also shows
$C_+\ge M_0\ge1$. For every $\xi\in E_+\setminus E_-$, some
$n\le\operatorname{ord}(\xi)$ satisfies $|w(\xi^n)|=Q$.
Equation~\eqref{eq:symbol-direct-divdiff} gives
\[
   \|\Res(\xi,T)\|\le\frac{M(T)L_wn}{\varepsilon}
      \le\frac A\varepsilon.
\]
The resolvent Neumann series therefore implies
\begin{equation}\label{eq:symbol-root-removal}
   \|\Res(\lambda,T)\|\le\frac{2A}{\varepsilon}
   \quad\text{if}\quad
   |\lambda-\xi|\le\delta,
   \quad \xi\in E_+\setminus E_-.
\end{equation}

Suppose first that $E_-\ne\emptyset$. Inside any of these discs,
$|\lambda|\le2$ and hence $\dist(\lambda,E_-)\le3$, so
\eqref{eq:symbol-root-removal} gives
\[
   \dist(\lambda,E_-)\|\Res(\lambda,T)\|
      \le\frac{6A}{\varepsilon}.
\]
Outside all the discs, let $|\lambda|>1$ and choose a nearest point
of $E_+$ to $\lambda$. If this point belongs to $E_-$, the required
resolvent expression is at most $C_+$. Otherwise
$d:=\dist(\lambda,E_+)\ge\delta$ and
$\dist(\lambda,E_-)\le d+2$, so the expression is at most
\[
   C_+\left(1+\frac2d\right)
      \le C_+\left(1+\frac{4A}{\varepsilon}\right)
      \le\frac{6AC_+}{\varepsilon}.
\]
This proves \eqref{eq:symbol-threshold-bound}.

If $E_-=\emptyset$, every point of $E_+$ has been removed. Outside
the discs we have $\dist(\lambda,E_+)\ge\delta$, giving
$\|\Res(\lambda,T)\|\le2AC_+/\varepsilon$ for $|\lambda|>1$.
Inside the discs use \eqref{eq:symbol-root-removal}. The structure
theorem gives $r(T)<1$, so continuity extends the bound to $|\lambda|=1$.
For each $\eta\in\T$, the disc centred at $\eta$ of radius
$1/\|\Res(\eta,T)\|$ is contained in the resolvent set. Consequently
every spectral point $z$ satisfies
$1-|z|\ge\varepsilon/(2AC_+)$, by choosing $\eta=z/|z|$ when $z\ne0$.
The same inequality for $z=0$ follows from
$\varepsilon/(2AC_+)\le1$. This proves
\eqref{eq:symbol-empty-threshold}.
\end{proof}

\begin{remark}[the power-boundedness is essential]\label{rem:generalsymbolnecessity}
Let $V\ne0$ satisfy $V^2=0$, put $T=\Id+V$, and take
$w(z)=(1-z)^2$. Then $T^n=\Id+nV$ and $w(T^n)=0$ for every $n\ge1$,
so the norm-gap condition holds with $q=0<\|w\|_\infty=4$ and
$E(w,0)=\{1\}$. Nevertheless $T$ is not power bounded, and therefore
is not Ritt. Thus the power-boundedness hypothesis in
Theorem~\ref{thm:generalsymbolstructure} cannot be omitted.

The peripheral set for a general symbol may contain roots of even
order. For instance, $T=-\Id$ and $w(z)=1-z^2$ satisfy $w(T^n)=0$
for every $n$; here $E(w,0)=\{1,-1\}$ and $T^2=\Id$ is Ritt.
\end{remark}

\section{Sparse sequences of powers}\label{sec:sparse}

Throughout this section $\bn=(n_k)_{k\ge0}$ is a strictly increasing sequence of positive integers with $n_0=1$; $c(\bn)$, $c_\infty(\bn)$, the arcs $\arc_c$ and the Jamison constant $J(\bn)$ are as in the introduction. The sequence $\bn$ is a \emph{Jamison sequence} if $J(\bn)>0$; equivalently, every bounded operator $T$ on a separable Banach space with $\sup_k\|T^{n_k}\|<\infty$ has countable unimodular point spectrum \cite{BadeaGrivaux07,RansfordRoginskaya,EisnerGrivaux}. Note that $m_{\arc_c}(1-z)=2\sin\frac{\pi}{c+1}$, the value of $|1-z|$ at the endpoints of $\arc_c$.

\subsection{Jamison constants}
The survey \cite{BGSurvey} and the paper \cite{BDG} contain extensive information on Jamison sequences and Jamison constants. We shall need the following result.

\begin{proposition}\label{prop:jamison}
\begin{enumerate}[label=\textup{(\roman*)}]
\item If $c(\bn)\le c<\infty$, then $J(\bn)\ge2\sin\frac{\pi}{c+1}$.
\item For $n_k=b^k$ with an integer $b\ge2$, $J(\bn)=2\sin\frac{\pi}{b+1}$; for $\bn=\N$, $J(\bn)=\sqrt3$.
\end{enumerate}
\end{proposition}

\begin{proof}
See Proposition~1.4 in \cite{BDG}.
\end{proof}

\subsection{Sampling the discrete Kato criterion} The following result is an extension of \cite[Theorem~5.5]{BGT} and \cite[Theorem~5.8]{BGT}.

\begin{theorem}\label{thm:arcsampling}
Let $c_\infty(\bn)<\infty$, let $c'\ge c_\infty(\bn)$, and let $T\in\B(X)$ be power bounded with $\sigma(T)\cap\T\subset\{1\}$. Suppose that one of the following holds:
\begin{enumerate}[label=\textup{(\roman*)}]
\item $\displaystyle\sup_{k\ge0}\ \sup_{\eta\in\arc_{c'}}\|(\eta-T^{n_k})^{-1}\|<\infty$;
\item $\displaystyle\sup_{k\ge0}\|w(T^{n_k})\|<m_{\arc_{c'}}(w)$ for some $w\in A^{1,1}(\D)$.
\end{enumerate}
Then $T$ is a Ritt operator. Here $\arc_1=\{-1\}$. The arc $\arc_{c'}$ may be replaced by any closed arc $\{e^{it}:\alpha\le t\le\gamma\}\subset\T\setminus\{1\}$ satisfying \eqref{eq:arccond} with $c'$.
\end{theorem}

\begin{proof}
In case (ii), Lemma~\ref{lem:divdiff} applied to $S=T^{n_k}$ gives (i). Let $K$ be the uniform bound in (i). For $\eta'\in\T$ with $\dist(\eta',\arc_{c'})\le\frac1{2K}$, the Neumann series gives $\eta'\in\rho(T^{n_k})$ and $\|(\eta'-T^{n_k})^{-1}\|\le2K$ for all $k$. Choose $\varepsilon\in(0,\frac1{2K}]$ so small that the closed $\varepsilon$-neighbourhood $\arc'$ of $\arc_{c'}$ in $\T$ does not contain $1$; then $\arc'$ is a closed arc $\{e^{it}:\alpha'\le t\le\gamma'\}$ with $\alpha'<\frac{2\pi}{c'+1}$ and $\gamma'>\frac{2\pi c'}{c'+1}$, so both ratios in \eqref{eq:arccond} computed for $\arc'$ are strictly larger than $c'$. Choose $c''$ with
\[
   c'<c''\le\min\Big\{\frac{\gamma'}{\alpha'},\ \frac{2\pi-\alpha'}{2\pi-\gamma'}\Big\};
\]
then $\arc'$ satisfies \eqref{eq:arccond} with $c''$, and $c''>c'\ge c_\infty(\bn)$. Lemma~\ref{lem:subseq} shows that $(\{n_k\},\arc')$ has the sampling property, and Proposition~\ref{thm:general} applies. The same enlargement increases both ratios of a general arc, and the argument includes $c'=1$, where $\arc_1=\{-1\}$ and $\arc'$ is a small arc around $-1$.
\end{proof}

For the full sequence we have $c_\infty(\N)=1$, and every nondegenerate closed arc $\{e^{it}:\alpha\le t\le\gamma\}\subset\T\setminus\{1\}$ satisfies \eqref{eq:arccond} for some $c'>1$; thus Theorem~\ref{thm:arcsampling} contains the arc-resolvent criterion \cite[Theorem~5.5]{BGT} and the defect criterion \cite[Theorem~5.8]{BGT}.

\begin{corollary}\label{cor:sparse}
Let $T\in\B(X)$ be power bounded.
\begin{enumerate}[label=\textup{(\alph*)}]
\item If $n_{k+1}/n_k\to1$, $\sigma(T)\cap\T\subset\{1\}$, and $\sup_k\|(\zeta-T^{n_k})^{-1}\|<\infty$ for one point $\zeta\in\T\setminus\{1\}$, then $T$ is a Ritt operator.
\item Suppose $n_0=1$. If $c(\bn)\le c<\infty$ and $\sup_k\|\Id-T^{n_k}\|<2\sin\frac{\pi}{c+1}$, then $T$ is a Ritt operator.
\item If $c_\infty(\bn)<\infty$, $\sigma(T)\cap\T\subset\{1\}$ and $\sup_k\|\Id-T^{n_k}\|<2\sin\frac{\pi}{c_\infty(\bn)+1}$, then $T$ is a Ritt operator.
\end{enumerate}
\end{corollary}

\begin{proof}
(a) Let $K=\sup_k\|(\zeta-T^{n_k})^{-1}\|$, $\varepsilon=\frac12\min\{1/(2K),|\zeta-1|\}$, and let $\arc=\{e^{it}:\alpha\le t\le\gamma\}$ be the arc $\{\eta:|\eta-\zeta|\le\varepsilon\}$; by the Neumann series $\|(\eta-T^{n_k})^{-1}\|\le2K$ for $\eta\in\arc$. Both ratios in \eqref{eq:arccond} are $>1$, so \eqref{eq:arccond} holds with $c'=1=c_\infty(\bn)$, and Theorem~\ref{thm:arcsampling}(i) applies.

(b) Let $q=\sup_k\|\Id-T^{n_k}\|<2\sin\frac{\pi}{c+1}\le J(\bn)$ (Proposition~\ref{prop:jamison}). If $\lambda\in\sigma(T)\cap\T$, then $|\lambda^{n_k}-1|\le\|\Id-T^{n_k}\|\le q$ for all $k$, so $\lambda=1$ by the definition of $J(\bn)$. Since $c_\infty(\bn)\le c$, we have $q<2\sin\frac{\pi}{c+1}\le2\sin\frac{\pi}{c_\infty(\bn)+1}=m_{\arc_{c_\infty(\bn)}}(1-z)$, and Theorem~\ref{thm:arcsampling}(ii) applies with $w(z)=1-z$ and $c'=c_\infty(\bn)$. (c) is proved in a similar way.
\end{proof}

\begin{remark}[normalisation and sampling]\label{rem:n0}
(a) The normalisation $n_0=1$ is essential in (b), through the Jamison constant: for $n_k=2\cdot3^k$ one has $c(\bn)=3$, but $T=-\Id$ satisfies $T^{n_k}=\Id$ for all $k$, so that $\sup_k\|\Id-T^{n_k}\|=0$ while $T$ is not Ritt; here $J(\bn)=0$. Theorem~\ref{thm:arcsampling} and parts (a), (c), which carry the spectral hypothesis, do not use $n_0=1$.

(b) In (b) the power-boundedness assumption is superfluous when the gaps of $\bn$ are bounded: if $G:=\sup_k(n_{k+1}-n_k)<\infty$ and $A:=\sup_k\|T^{n_k}\|<\infty$, then $M(T)\le\max\{1,A\max_{0\le r<G}\|T^r\|\}$, since every $n\ge1$ can be written as $n=n_k+r$ with $0\le r<G$; and a sampled displacement bound $q$ gives $A\le1+q$.

(c) A one-point resolvent bound along $(T^{n_k})$ detects the Ritt property whenever $n_{k+1}/n_k\to1$ (for instance for $n_k=(k+1)^2$, or along the primes after adjoining $1$), while for geometric sequences an arc of prescribed logarithmic length is needed. The next result shows that for $n_k=b^k$ the arc $\arc_b$ is optimal among arcs symmetric about $-1$, even under the spectral hypothesis $\sigma(T)\subset\D\cup\{1\}$, and determines the sharp norm thresholds.
\end{remark}

\subsection{Geometric sequences}
For an integer $b\ge2$ put $Q_b:=\max\{1,2\sin\frac{\pi}{b+1}\}$; thus $Q_2=\sqrt3$, $Q_3=\sqrt2$, $Q_4=2\sin\frac\pi5\approx1.176$ and $Q_b=1$ for $b\ge5$.

\begin{theorem}[geometric sequences]\label{thm:sharp}
Let $b\ge2$ be an integer and $n_k=b^k$.
\begin{enumerate}[label=\textup{(\roman*)}]
\item Let $\arc$ be a closed arc contained in the interior of $\arc_b$, i.e.\ $\arc\subset\{e^{it}:\frac{2\pi}{b+1}<t<\frac{2\pi b}{b+1}\}$. There is a normal contraction $T$ on $\ell^2$ with $\sigma(T)\subset\D\cup\{1\}$, which is not a Ritt operator, such that
\[
   \sup_{k\ge0}\ \sup_{\eta\in\arc}\|(\eta-T^{b^k})^{-1}\|<\infty
   \qquad\text{and}\qquad
   \sup_{k\ge0}\|\Id-T^{b^k}\|\le Q_b .
\]
\item Among power bounded operators with $\sigma(T)\cap\T\subset\{1\}$, the implication
$\sup_k\|\Id-T^{b^k}\|<Q\ \Rightarrow\ T$ Ritt holds for $Q=Q_b$, and it fails for every $Q>Q_b$; it also fails if $<Q_b$ is replaced by $\le Q_b$, even for normal contractions on $\ell^2$.
\item Without the spectral assumption the sharp threshold is $2\sin\frac{\pi}{b+1}$: the implication $\sup_k\|\Id-T^{b^k}\|<2\sin\frac\pi{b+1}\Rightarrow T$ Ritt holds for every power bounded $T$, and the scalar operator $T=e^{2\pi i/(b+1)}$ has $\sup_k|1-T^{b^k}|=2\sin\frac{\pi}{b+1}$ and is not Ritt.
\end{enumerate}
\end{theorem}

\begin{proof}
(i) Put $\alpha=2\pi/(b+1)$, choose $c_m\in(0,1]$ with $c_m\to0$, and let
\[
   \theta_m=\alpha b^{-m},\qquad \rho_m=e^{-c_m\theta_m},\qquad \mu_m=\rho_me^{i\theta_m}\qquad(m\ge1).
\]
Let $T$ be the diagonal operator on $\ell^2$ with eigenvalues $\mu_m$. Then $T$ is a normal contraction, $\mu_m\to1$, and $\sigma(T)=\{\mu_m:m\ge1\}\cup\{1\}\subset\D\cup\{1\}$. Since $1-|\mu_m|\le c_m\theta_m$ while $|1-\mu_m|\ge\rho_m\sin\theta_m\ge\theta_m/\pi$ for large $m$, we have $|1-\mu_m|/(1-|\mu_m|)\to\infty$, so $\sigma(T)$ is contained in no Stolz domain $\{|1-\lambda|\le C(1-|\lambda|)\}\cup\{1\}$, and $T$ is not Ritt \cite{LeMerdy}.

Fix $k\ge0$. If $k\ge m$, then $b^k\theta_m=\alpha b^{k-m}$ and $b^{k-m}\equiv(-1)^{k-m}\pmod{b+1}$, so $\mu_m^{b^k}=\rho_m^{b^k}e^{\pm i\alpha}$ lies on one of the radii $R_\pm=\{\rho e^{\pm i\alpha}:0\le\rho\le1\}$. If $k<m$, then $b^k\theta_m=\alpha b^{-(m-k)}\in(0,\alpha/b]$, so $\mu_m^{b^k}$ lies in the closed sector $S=\{\rho e^{is}:0\le\rho\le1,\ 0\le s\le\alpha/b\}$. Hence $\sigma(T^{b^k})\subset R_+\cup R_-\cup S\cup\{1\}=:F$ for every $k$. The compact set $F$ meets $\T$ exactly in $\{e^{\pm i\alpha}\}\cup\{e^{is}:0\le s\le\alpha/b\}$, which is disjoint from $\arc$; since $\arc\subset\T$, the compact sets $\arc$ and $F$ are disjoint and $d:=\dist(\arc,F)>0$. As $T^{b^k}$ is normal, $\|(\eta-T^{b^k})^{-1}\|=1/\dist(\eta,\sigma(T^{b^k}))\le1/d$ for all $\eta\in\arc$ and all $k$.

For $0\le\rho\le1$ and $z\in\T$, writing $1-\rho z=(1-\rho)\cdot1+\rho(1-z)$ gives
$|1-\rho z|\le(1-\rho)+\rho|1-z|\le\max\{1,|1-z|\}$.
All the unimodular directions just encountered satisfy $|1-z|\le2\sin(\alpha/2)=2\sin\frac\pi{b+1}$. Therefore every sampled eigenvalue satisfies $|1-\mu_m^{b^k}|\le Q_b$, and $\|\Id-T^{b^k}\|=\sup_m|1-\mu_m^{b^k}|\le Q_b$ by normality. In fact the supremum over $k$ equals $Q_b$: taking $k=m$ gives $\mu_m^{b^m}=e^{-\alpha c_m}e^{i\alpha}\to e^{i\alpha}$, while fixing $m$ and letting $k\to\infty$ gives $\mu_m^{b^k}\to0$. These yield the lower bounds $2\sin(\alpha/2)$ and $1$, respectively.

(ii) When $Q_b>1$, i.e.\ $b\le4$, sufficiency for $Q=Q_b$ is Corollary~\ref{cor:sparse}(c) with $c_\infty(\bn)=b$; when $Q_b=1$ it is Lemma~\ref{lem:subone} below. The operator of (i) has $\sup_k\|\Id-T^{b^k}\|\le Q_b<Q$ for every $Q>Q_b$ and is not Ritt, which proves both failure assertions.

(iii) Sufficiency is Corollary~\ref{cor:sparse}(b) with $c=b$; the scalar example follows from $b^k\equiv(-1)^k\pmod{b+1}$.
\end{proof}

\begin{remark}[an off-centre dyadic arc]\label{rem:offcentre}
The position of the arc matters in the optimality statement for $\arc_2$.
The shorter arc
\[
 A=\{e^{it}:\pi/2\le t\le\pi\}
\]
also gives a dyadic resolvent criterion: if $T$ is power bounded,
$\sigma(T)\cap\T\subset\{1\}$ and
\[
 \sup_k\sup_{\eta\in A}\|\Res(\eta,T^{2^k})\|<\infty,
\]
then $T$ is Ritt.
Indeed, a Neumann series argument extends this uniform bound to
\[
 A_\delta=\{e^{2\pi it}:1/4-\delta\le t\le1/2+\delta\}
\]
for some $0<\delta<1/8$. We verify the sampling property for this enlarged arc.
Given $0<|\theta|<\pi/2$, put $u=|\theta|/(2\pi)$ and choose $m\ge0$
with $x=2^m u\in[1/4,1/2)$. For $\theta>0$ the time $2^m$ hits $A$.
For $\theta<0$, if $x\ge1/2-\delta$, then the phase $1-x$ belongs to
$[1/2,1/2+\delta]$. Otherwise the next phase is
$y=1-2x\in(2\delta,1/2]$. If $y\ge1/4$ it already belongs to the arc;
if $y<1/4$, choose the least $l\ge1$ with $2^ly\ge1/4$. Then
$2^ly<1/2$ and $2^l<1/(2y)<1/(4\delta)$.
Thus a chosen dyadic time $n$ always satisfies
\[
 e^{in\theta}\in A_\delta,\qquad n|\theta|\le\frac\pi{2\delta}.
\]
For the last case use $n|\theta|=4\pi 2^l x$; the earlier cases have
$n|\theta|<2\pi$. Proposition~\ref{thm:general} now applies.
Since the lengths of $A$ and $\arc_2$ are $\pi/2$ and $2\pi/3$,
respectively, Theorem~\ref{thm:sharp}(i) cannot be interpreted as
minimality among all arcs of arbitrary position.
\end{remark}
The following lemma is implicit in \cite{BDG}; we give a simple proof below following \cite{BDG}. 

\begin{lemma}[a sampled gap below one]\label{lem:subone}
Let $T$ be a bounded operator on a complex Banach space, let
$n_k\to\infty$, and suppose that $\sigma(T)\cap\T\subset\{1\}$.
If $\sup_k\|\Id-T^{n_k}\|<1$, then $T=\Id$.
\end{lemma}
\begin{proof}
Put $q=\sup_k\|\Id-T^{n_k}\|<1$.
For every $\lambda\in\sigma(T)$, spectral mapping gives
$|1-\lambda^{n_k}|\le q$ for all $k$.
If $|\lambda|<1$, the left-hand side tends to $1$; if
$|\lambda|>1$, it tends to infinity. Thus $\sigma(T)\subset\T$,
and the peripheral spectral assumption gives $\sigma(T)=\{1\}$.

The principal logarithm $L=\log T$ is therefore defined by the
holomorphic functional calculus. For each positive integer $n$, the
identity $\log(z^n)=n\log z$ holds on a sufficiently small
neighbourhood of $1$. Since $\sigma(T)=\sigma(T^n)=\{1\}$, the
composition rule for this calculus consequently gives
$\log(T^n)=nL$. In particular,
\[
 n_k L=\log(T^{n_k})
 =-\sum_{j\ge1}\frac{(\Id-T^{n_k})^j}{j},
\]
where the series converges in norm and represents the principal
logarithm because $\|\Id-T^{n_k}\|\le q<1$. Hence
\[
 n_k\|L\|\le\sum_{j\ge1}\frac{q^j}{j}
 =-\log(1-q).
\]
Letting $k\to\infty$ yields $L=0$, and the functional calculus
identity $T=e^L$ gives $T=\Id$.
\end{proof}

Thus the dyadic threshold is $\sqrt3$ even with the spectral assumption,
whereas the full-sequence threshold is $2$ under that assumption; for $b\ge5$ the threshold with the spectral assumption is $1$. Below $1$ we obtain $T=I$. 
\subsection{Unbounded quotients}
For unbounded quotients no fixed arc criterion survives. Part (i) below is the arc version of \cite[Proposition~6.8]{BGT}, whose construction we follow; part (ii) shows that every norm-gap threshold greater than $1$ fails. Lemma~\ref{lem:subone} explains the role of $1$ under the peripheral spectral assumption.

\begin{theorem}\label{thm:unbounded}
Let $c(\bn)=\infty$.
\begin{enumerate}[label=\textup{(\roman*)}]
\item For every closed arc $\arc\subset\T\setminus\{1\}$ there is a normal contraction $T$ on $\ell^2$ with $\sigma(T)\subset\D\cup\{1\}$ and $\sup_k\sup_{\eta\in\arc}\|(\eta-T^{n_k})^{-1}\|<\infty$ which is not Ritt.
\item For every $q>1$ there is a normal contraction $T$ on $\ell^2$ with $\sigma(T)\subset\D\cup\{1\}$ and $\sup_k\|\Id-T^{n_k}\|\le q$ which is not Ritt.
\end{enumerate}
\end{theorem}

\begin{proof}
(i) Let $d=\dist(1,\arc)>0$. Since the quotients are unbounded, we may choose indices $K(1)<K(2)<\cdots$ with $K(j)\ge1$ and $n_{K(j)}/n_{K(j)-1}\ge16j\log2/d$; then $K(j)\ge j$ and $n_{K(j)-1}\ge n_{j-1}\ge j$. Put $\varepsilon_j=\log2/n_{K(j)}$, $\theta_j=j\varepsilon_j$, $\mu_j=(1-\varepsilon_j)e^{i\theta_j}$, and let $T$ be diagonal with eigenvalues $\mu_j$. We have
\[
   \theta_j=\frac{j\log2}{n_{K(j)}}\le\frac{j\log2\,d}{16j\log2\,n_{K(j)-1}}=\frac{d}{16n_{K(j)-1}}\le\frac d{16j}\to0
\]
and $\varepsilon_j\to0$, so $\mu_j\to1$ and $\sigma(T)\subset\D\cup\{1\}$; moreover $|1-\mu_j|/(1-|\mu_j|)\ge(1-\varepsilon_j)\sin(\theta_j)/\varepsilon_j\gtrsim j\to\infty$, so $T$ is not Ritt. Fix $j$ and $k$. If $k\ge K(j)$, then $|\mu_j^{n_k}|\le(1-\varepsilon_j)^{n_{K(j)}}\le e^{-\log2}=\frac12$, so $|\eta-\mu_j^{n_k}|\ge\frac12$ for $\eta\in\T$. If $k<K(j)$, then $n_k\varepsilon_j\le n_{K(j)-1}\varepsilon_j\le d/(16j)$ and $n_k\theta_j\le d/16$, so $|\mu_j^{n_k}-1|\le|(1-\varepsilon_j)^{n_k}-1|+|e^{in_k\theta_j}-1|\le n_k\varepsilon_j+n_k\theta_j\le d/8$ and $|\eta-\mu_j^{n_k}|\ge d-d/8$ for $\eta\in\arc$. Together with $\dist(\eta,1)\ge d$, this gives $\dist(\eta,\sigma(T^{n_k}))\ge\min\{\frac12,\frac{7d}8\}$ for all $\eta\in\arc$ and $k$, and normality gives the uniform resolvent bound.

(ii) We may assume $1<q<2$. Put $\tau=(q-1)/2\in(0,\frac12)$, $L=\log(1/\tau)$, $\rho_m=e^{-L/n_{m+1}}$, $\theta_m=\tau/n_m$ and $\lambda_m=\rho_me^{i\theta_m}$, $m\ge0$, and let $T$ be diagonal with eigenvalues $\lambda_m$; then $T$ is a normal contraction with $\lambda_m\to1$, so $\sigma(T)\subset\D\cup\{1\}$. Fix $k$ and $m$. If $k\le m$, then $n_k\le n_m$, $\rho_m^{n_k}\ge\rho_m^{n_m}=e^{-Ln_m/n_{m+1}}\ge e^{-L}=\tau$, and
\[
   |1-\lambda_m^{n_k}|\le(1-\rho_m^{n_k})+|1-e^{in_k\theta_m}|\le(1-\tau)+n_m\theta_m=(1-\tau)+\tau=1 .
\]
If $k\ge m+1$, then $|\lambda_m|^{n_k}\le\rho_m^{n_{m+1}}=\tau$ and $|1-\lambda_m^{n_k}|\le1+\tau<q$. Hence $\sup_k\|\Id-T^{n_k}\|\le q$ by normality. Finally, using $1-e^{-x}\le x$, $|1-\lambda_m|/(1-|\lambda_m|)\ge\rho_m\sin(\theta_m)\,n_{m+1}/L\ge(\tau/(\pi L))\,n_{m+1}/n_m$ for large $m$, which is unbounded since $c(\bn)=\infty$; so $T$ is not Ritt.
\end{proof}

The endpoint $q=1$ depends on the sampling set and is considered in
Section~\ref{sec:endpoint}.

\subsection{Sampling the Ritt condition}
The Ritt property is characterised by $\sup_{n\ge1}n\|T^n-T^{n-1}\|<\infty$ for power bounded $T$. Recall that $\bn$ is a \emph{Ritt sampling sequence} if, for every Banach space $X$ and every power bounded $T\in\B(X)$,
\[
   \sup_{k\ge0}n_k\|T^{n_k}-T^{n_k-1}\|<\infty\quad\Longrightarrow\quad T\ \text{is a Ritt operator}.
\]
The following result can be viewed as an operator-theoretical characterisation of sequences with bounded quotients. 
\begin{theorem}\label{thm:rittsampling}
$\bn$ is a Ritt sampling sequence if and only if $c(\bn)<\infty$. If $c(\bn)<\infty$, then for every power bounded $T$,
\[
   \sup_{n\ge1}n\|T^n-T^{n-1}\|\le M(T)\,c(\bn)\,\sup_{k\ge0}n_k\|T^{n_k}-T^{n_k-1}\| .
\]
\end{theorem}

\begin{proof}
Suppose $c(\bn)<\infty$ and let $A=\sup_kn_k\|T^{n_k-1}(\Id-T)\|$. For $n\ge1$ choose $k$ with $n_k\le n<n_{k+1}$; then $T^{n-1}(\Id-T)=T^{n-n_k}T^{n_k-1}(\Id-T)$ and $n\|T^{n-1}(\Id-T)\|\le M(T)\frac{n}{n_k}\,n_k\|T^{n_k-1}(\Id-T)\|\le M(T)c(\bn)A$.

Conversely, suppose $c(\bn)=\infty$ and choose indices $k_m$ with $R_m:=n_{k_m+1}/n_{k_m}\to\infty$; put $p_m=n_{k_m}$, $q_m=n_{k_m+1}$, $N_m=\sqrt{p_mq_m}=p_m\sqrt{R_m}$, $L_m=R_m^{1/4}$ and $\lambda_m=\exp(-1/N_m+iL_m/N_m)$. Since $L_m/N_m=R_m^{-1/4}/p_m\to0$, $\lambda_m\to1$, and $|\lambda_m|<1$. Let $T$ be the diagonal contraction on $\ell^2$ with eigenvalues $\lambda_m$. For $n\ge1$ and $x=n/N_m$,
\begin{equation}\label{eq:diagbound}
   n|\lambda_m|^{n-1}|1-\lambda_m|\le n\,e^{-(n-1)/N_m}\Big(\frac1{N_m}+\frac{L_m}{N_m}\Big)\le2e\,L_m\,xe^{-x},
\end{equation}
using $|1-\lambda_m|\le|1-e^{-1/N_m}|+|1-e^{iL_m/N_m}|$, $L_m\ge1$ and $e^{1/N_m}\le e$. If $n=n_k\le p_m$, then $x\le p_m/N_m=R_m^{-1/2}$ and $n|\lambda_m|^{n-1}|1-\lambda_m|\le2eL_mR_m^{-1/2}=2eR_m^{-1/4}\le2e$. If $n=n_k\ge q_m$, then $x\ge R_m^{1/2}\ge1$, $xe^{-x}$ is decreasing on $[1,\infty)$, and $n|\lambda_m|^{n-1}|1-\lambda_m|\le2eR_m^{3/4}e^{-R_m^{1/2}}$, which is bounded. As no $n_k$ lies strictly between $p_m$ and $q_m$, we get $\sup_k\sup_mn_k|\lambda_m|^{n_k-1}|1-\lambda_m|<\infty$, i.e.\ $\sup_kn_k\|T^{n_k}-T^{n_k-1}\|<\infty$. On the other hand, for $n=\lfloor N_m\rfloor$ we have $x\in(1/2,1]$ for large $m$, $|1-\lambda_m|\ge|\operatorname{Im}\lambda_m|\ge e^{-1/N_m}\frac2\pi\frac{L_m}{N_m}$, and $n|\lambda_m|^{n-1}|1-\lambda_m|\ge\frac{N_m}{2}e^{-1}\cdot\frac{2L_m}{e\pi N_m}=\frac{L_m}{e^2\pi}\to\infty$. Hence $\sup_nn\|T^n-T^{n-1}\|=\infty$ and $T$ is not Ritt.
\end{proof}

The same dichotomy holds for fractional differences. For a power bounded $T$ put $B=\Id-T$; since $e^{-tB}=e^{-t}\sum_{j\ge0}t^jT^j/j!$ satisfies $\|e^{-tB}\|\le M(T)$, the operator $B$ is sectorial, its fractional powers $B^s$, $s>0$, are defined by the usual functional calculus (and are bounded operators, $B$ being bounded), and they satisfy the moment inequality
\begin{equation}\label{eq:moment}
   \|B^{s_1}y\|\le C_{s_1,s_2}\|y\|^{1-s_1/s_2}\|B^{s_2}y\|^{s_1/s_2}\qquad(0<s_1<s_2,\ y\in X),
\end{equation}
with $C_{s_1,s_2}$ depending only on $s_1,s_2$ and $M(T)$; see \cite[Chapter~6]{Haase}. Ritt operators satisfy $\sup_nn^s\|T^{n-1}B^s\|<\infty$ for every $s>0$ \cite{LeMerdy}.

\begin{proposition}[fractional differences]\label{prop:fractional}
Let $s>0$.
\begin{enumerate}[label=\textup{(\roman*)}]
\item If $c(\bn)<\infty$, then every power bounded $T$ satisfies
\[
 \sup_{n\ge1}n^s\|T^{n-1}B^s\|\le M(T)c(\bn)^s\sup_kn_k^s\|T^{n_k-1}B^s\|,
\]
and if the right-hand side is finite, then $T$ is a Ritt operator.
\item If $c(\bn)=\infty$, the diagonal contraction $T$ constructed in the proof of
Theorem~\ref{thm:rittsampling} satisfies $\sup_kn_k^s\|T^{n_k-1}B^s\|<\infty$, and it is
not a Ritt operator.
\end{enumerate}
\end{proposition}

\begin{proof}
(i) The interpolation inequality is proved exactly as in Theorem~\ref{thm:rittsampling}, writing $T^{n-1}B^s=T^{n-n_k}T^{n_k-1}B^s$ for $n_k\le n<n_{k+1}$. Suppose now that $A:=\sup_nn^s\|T^{n-1}B^s\|<\infty$ and choose an integer $l$ with $ls>1$. For $n\ge2$ write $n-1=l(k-1)+r$ with $0\le r<l$, so that $k\ge1$ and $k\ge(n-1)/l\ge n/(2l)$. Commutativity gives
\[
 T^{n-1}B^{ls}=T^r\big(T^{k-1}B^s\big)^l,\qquad\text{hence}\qquad
 \|T^{n-1}B^{ls}\|\le M(T)\Big(\frac{A}{k^s}\Big)^l\le M(T)A^l(2l)^{ls}n^{-ls}.
\]
The moment inequality \eqref{eq:moment} with $s_1=1$, $s_2=ls$, applied to $y=T^{n-1}x$, yields
\[
 \|T^{n-1}Bx\|\le C_{1,ls}\,(M(T)\|x\|)^{1-1/(ls)}\big(M(T)A^l(2l)^{ls}n^{-ls}\|x\|\big)^{1/(ls)}\lesssim n^{-1}\|x\|,
\]
so $\sup_nn\|T^{n-1}(\Id-T)\|<\infty$ and $T$ is Ritt.

(ii) For the diagonal operator $T$ with eigenvalues $\lambda_m$, the operator $B^s$ is diagonal with eigenvalues $(1-\lambda_m)^s$ (principal branch, $\operatorname{Re}(1-\lambda_m)>0$), so $\|T^{n-1}B^s\|=\sup_m|\lambda_m|^{n-1}|1-\lambda_m|^s$. As in \eqref{eq:diagbound}, with $x=n/N_m$,
\[
 n^s|\lambda_m|^{n-1}|1-\lambda_m|^s\le n^se^{-(n-1)/N_m}\Big(\frac{2L_m}{N_m}\Big)^s\le2^se\,L_m^s\,x^se^{-x}.
\]
For $n=n_k\le p_m$ this is at most $2^seR_m^{s/4}R_m^{-s/2}\le2^se$; for $n=n_k\ge q_m$ we have $x\ge R_m^{1/2}$, and once $R_m^{1/2}\ge s$ the function $x^se^{-x}$ is decreasing on $[R_m^{1/2},\infty)$, so the bound is at most $2^seR_m^{3s/4}e^{-R_m^{1/2}}$, which is bounded in $m$. For each of the finitely many remaining coordinates, the supremum over all $n\ge1$ of $n^s|\lambda_m|^{n-1}|1-\lambda_m|^s$ is finite because $|\lambda_m|<1$. Since $T$ is not Ritt (Theorem~\ref{thm:rittsampling}), the claim follows.
\end{proof}

\section{Exact displacement thresholds}\label{sec:exact-sampling}
\label{samp:subsection}
In this section we allow $\bn=(n_k)_{k\ge0}$ to be any
strictly increasing sequence of positive integers; the normalisation
$n_0=1$ is not required. We write $\mathcal N=\{n_k:k\ge0\}$.
All operators in this section are assumed to be power bounded. 

We characterise the sequences for which a sampled displacement bound,
together with the peripheral spectral condition, forces the Ritt property.
Recall that $\mathcal N=\{n_k:k\ge0\}$ and define, for $B,t>0$,
\begin{equation}\label{samp:profile}
 h_{\mathcal N}(B,t)
 :=\max\bigl\{|1-e^{int}|:n\in\mathcal N,\ nt\le B\bigr\},
 \qquad
 \alpha(\mathcal N)
 :=\lim_{B\to\infty}\liminf_{t\downarrow0}h_{\mathcal N}(B,t),
\end{equation}
where the maximum of the empty set is $0$. The outer limit exists by
monotonicity in $B$, and $0\le\alpha(\mathcal N)\le2$. In particular, if
$a<\alpha(\mathcal N)$, there are $B,\delta>0$ such that every
$t\in(0,\delta)$ admits $n\in\mathcal N$ with
\begin{equation}\label{samp:escape}
 nt\le B,\qquad |1-e^{int}|>a.
\end{equation}
Conversely, \eqref{samp:escape} implies $\alpha(\mathcal N)\ge a$.
Thus the definition records angular escape at small angles, with a uniform
bound on the rescaled sampling time $nt$.

The quantity $\alpha(\mathcal N)$ is unchanged by adding or removing
finitely many sampling times. Indeed, for a nonempty finite set $F$ of
positive integers,
\[
 0\le h_{\mathcal N\cup F}(B,t)-h_{\mathcal N}(B,t)
 \le\max_{n\in F}|1-e^{int}|
 \le t\max F,
\]
and one then takes the two limits in \eqref{samp:profile}. It is also
monotone under inclusion of sampling sets. Unlike the Jamison constant,
it depends only on small angles and on sampling times of order $1/t$.
For example, $\alpha(\N)=2$, as shown below, whereas
$J(\N)=\sqrt3$ by Proposition~\ref{prop:jamison}.

\begin{theorem}[exact displacement criterion]\label{samp:exact}
Fix $1<q<2$. The following assertions are equivalent.
\begin{enumerate}[label=\textup{(\roman*)}]
\item For every complex Banach space $X$ and every power bounded
operator $T\in\B(X)$,
\begin{equation}\label{samp:hypotheses}
 \sup_{n\in\mathcal N}\|\Id-T^n\|\le q,
 \qquad \sigma(T)\cap\T\subset\{1\}
\end{equation}
imply that $T$ is a Ritt operator.
\item The same implication holds for every normal contraction on $\ell^2$.
\item $q<\alpha(\mathcal N)$.
\end{enumerate}
In particular, when $1<\alpha(\mathcal N)<2$, the implication fails
at $q=\alpha(\mathcal N)$, already for a diagonal contraction on $\ell^2$.
\end{theorem}

\begin{proof}
Assume (iii), and choose $q_1$ with
$q<q_1<\alpha(\mathcal N)$. By \eqref{samp:escape}, there are
$B>0$ and $\delta\in(0,\pi)$ such that for each $0<t<\delta$ one can
choose $n\in\mathcal N$ satisfying
$nt\le B$ and $|1-e^{int}|>q_1$.
Let $T$ satisfy \eqref{samp:hypotheses}, and set $M=M(T)$.
For $\xi=e^{it}$ or $e^{-it}$, the decomposition
\[
 \xi^n-T^n=(\xi^n-1)\Id+(\Id-T^n)
\]
and the Neumann series give
\[
 \|\Res(\xi^n,T^n)\|\le\frac1{q_1-q}.
\]
Lemma~\ref{lem:factor} consequently yields
\begin{equation}\label{samp:local-resolvent}
 |\xi-1|\,\|\Res(\xi,T)\|
 \le\frac{Mn|\xi-1|}{q_1-q}
 \le\frac{MB}{q_1-q}
 \qquad(0<|\arg\xi|<\delta).
\end{equation}
The peripheral spectral assumption ensures that the resolvent is bounded
on the remaining compact arc
$\{\xi\in\T:|\arg\xi|\ge\delta\}$. Proposition~\ref{prop:boundary}
now shows that $T$ is Ritt. This proves (iii)$\Rightarrow$(i), and
(i)$\Rightarrow$(ii) is immediate.

To prove (ii)$\Rightarrow$(iii), suppose that
$\alpha(\mathcal N)\le q$. Choose an increasing sequence $B_j\to\infty$
so large that, with $\varepsilon_j=B_j^{-1/2}$,
\begin{equation}\label{samp:damping-choice}
 \varepsilon_j\le\frac1{36},\qquad
 1+e^{-\sqrt{B_j}}\le q \qquad(j\ge1).
\end{equation}
For each $j$,
$\liminf_{t\downarrow0}h_{\mathcal N}(B_j,t)\le\alpha(\mathcal N)\le q$.
We may therefore choose $1>\theta_1>\theta_2>\cdots\downarrow0$ such that
\begin{equation}\label{samp:bad-angles}
 h_{\mathcal N}(B_j,\theta_j)
 \le q+B_j^{-1}=q+\varepsilon_j^2.
\end{equation}
Define a diagonal contraction on $\ell^2$ by
\begin{equation}\label{samp:diagonal}
 T e_j=z_j e_j,\qquad
 z_j=\exp(-\varepsilon_j\theta_j+i\theta_j).
\end{equation}

We first verify the displacement bound for every coordinate and every
sampling time. Fix $j$ and $n\in\mathcal N$, and abbreviate
\[
 B=B_j,\quad \varepsilon=\varepsilon_j,\quad
 t=n\theta_j,\quad r=e^{-\varepsilon t}.
\]
If $t>B$, then
\[
 |1-z_j^n|\le1+r\le1+e^{-\varepsilon B}
 =1+e^{-\sqrt B}\le q.
\]
Suppose next that $t\le B$, so that $|1-e^{it}|\le q+\varepsilon^2$ by
\eqref{samp:bad-angles}. If $|1-e^{it}|\le q$, radial convexity gives
\[
 |1-re^{it}|\le(1-r)+r|1-e^{it}|\le q,
\]
since $q>1$. It remains to consider $|1-e^{it}|>q$. Since $|1-e^{it}|\le t$, we have
$t>q$ and hence $0\le r\le r_*:=e^{-q\varepsilon}$.
The identity
\[
 |1-re^{it}|^2=r|1-e^{it}|^2+(1-r)^2
\]
and convexity of the right-hand side as a function of $r$ show that
\begin{equation}\label{samp:convex-bound}
 |1-re^{it}|^2
 \le\max\bigl\{1,\ r_*(q+\varepsilon^2)^2+(1-r_*)^2\bigr\}.
\end{equation}
Put $a=1-r_*$. Since $q\varepsilon\le1$, the elementary inequalities
$x/2\le1-e^{-x}\le x$ for $0\le x\le1$ give
$q\varepsilon/2\le a\le q\varepsilon$. Therefore
\begin{align*}
 r_*(q+\varepsilon^2)^2+(1-r_*)^2-q^2
 &=-aq^2+r_*(2q\varepsilon^2+\varepsilon^4)+a^2\\
 &\le-\frac{q^3\varepsilon}{2}
       +(q^2+2q)\varepsilon^2+\varepsilon^4\\
 &\le-\frac{\varepsilon}{2}+9\varepsilon^2
 \le-\frac{\varepsilon}{4}<0.
\end{align*}
Here we used $1<q<2$ and $\varepsilon\le1/36$.
Together with \eqref{samp:convex-bound}, this proves
$|1-z_j^n|\le q$ in the remaining case. Taking suprema over $j$ and $n$
gives $\sup_{n\in\mathcal N}\|\Id-T^n\|\le q$.

Finally, $|z_j|<1$ and $z_j\to1$, so
$\sigma(T)=\{z_j:j\ge1\}\cup\{1\}$ and
$\sigma(T)\cap\T=\{1\}$. Nevertheless,
\begin{align*}
 |e^{i\theta_j}-1|\,\|\Res(e^{i\theta_j},T)\|
 &\ge\frac{|e^{i\theta_j}-1|}{|e^{i\theta_j}-z_j|}
 =\frac{|e^{i\theta_j}-1|}{1-e^{-\varepsilon_j\theta_j}}\\
 &\ge\frac{2\theta_j/\pi}{\varepsilon_j\theta_j}
 =\frac{2}{\pi\varepsilon_j}\longrightarrow\infty.
\end{align*}
Proposition~\ref{prop:boundary} shows that $T$ is not Ritt, contradicting
(ii). This completes the proof.
\end{proof}

The case of the full threshold $2$ admits a purely arithmetic formulation.
The bound on the odd multiplier in the next statement is uniform in $x$;
it may depend on the prescribed accuracy $\varepsilon$.

\begin{corollary}[the full threshold]\label{samp:full-threshold}
The following assertions are equivalent.
\begin{enumerate}[label=\textup{(\roman*)}]
\item Every power bounded operator $T$ on every complex Banach space
satisfying
\[
 \sup_{n\in\mathcal N}\|\Id-T^n\|<2,
 \qquad \sigma(T)\cap\T\subset\{1\},
\]
is a Ritt operator.
\item $\alpha(\mathcal N)=2$.
\item For every $\varepsilon>0$ there are $L\in\N$ and $x_0>0$ such
that for every real $x\ge x_0$ there exist $n\in\mathcal N$ and a
positive odd integer $m\le L$ satisfying
\begin{equation}\label{samp:odd-approximation}
 |n-mx|<\varepsilon x.
\end{equation}
\end{enumerate}
Assertion \textup{(i)} is unchanged if it is restricted to normal
contractions on $\ell^2$.
\end{corollary}

\begin{proof}
If (ii) holds and $Q:=\sup_{n\in\mathcal N}\|\Id-T^n\|<2$, choose
$\max\{1,Q\}<q<2$. Theorem~\ref{samp:exact} applies with this $q$.
Conversely, if $\alpha(\mathcal N)<2$, choose
$\max\{1,\alpha(\mathcal N)\}<q<2$. The same theorem supplies a normal
contraction violating (i). This also proves the last assertion.

For (ii)$\Rightarrow$(iii), it suffices to consider $0<\varepsilon<1$.
Choose $B,\delta>0$ so that for every $0<t<\delta$ there is
$n\in\mathcal N$ with
\[
 nt\le B,\qquad |1-e^{int}|>2\cos(\pi\varepsilon/2).
\]
The second inequality means that $nt$ lies within $\pi\varepsilon$ of
an odd multiple $m\pi$ of $\pi$. Since $nt>0$ and $\varepsilon<1$,
this odd integer $m$ is positive; moreover $m<B/\pi+\varepsilon$.
Choose an integer $L\ge B/\pi+1$ and take $t=\pi/x$, with $x$ large
enough that $t<\delta$. Dividing $|nt-m\pi|<\pi\varepsilon$ by $t$
gives \eqref{samp:odd-approximation}.

Conversely, suppose (iii), fix $0<\varepsilon<1$, and put $t=\pi/x$.
For all sufficiently small $t>0$, \eqref{samp:odd-approximation} yields
\[
 |nt-m\pi|<\pi\varepsilon,\qquad
 nt<\pi(L+\varepsilon).
\]
Consequently
\[
 h_{\mathcal N}\bigl(\pi(L+\varepsilon),t\bigr)
 >2\cos(\pi\varepsilon/2).
\]
It follows that $\alpha(\mathcal N)\ge2\cos(\pi\varepsilon/2)$.
Letting $\varepsilon\downarrow0$ proves (ii).
\end{proof}

\begin{proposition}[quotient bounds]\label{samp:quotients}
If $c_\infty(\bn)<\infty$, then
\begin{equation}\label{samp:quotient-lower}
 \alpha(\mathcal N)\ge
 2\sin\frac{\pi}{c_\infty(\bn)+1}.
\end{equation}
If $c_\infty(\bn)=\infty$, then $\alpha(\mathcal N)=0$.
In particular,
\[
 \alpha(\mathcal N)>0
 \quad\Longleftrightarrow\quad c_\infty(\bn)<\infty.
\]
\end{proposition}

\begin{proof}
Suppose first that $c:=c_\infty(\bn)<\infty$. Fix $C>c$ and choose
$k_0$ so that $n_{k+1}\le Cn_k$ for all $k\ge k_0$.
Set $a=2\pi/(C+1)$. For $0<t<a/n_{k_0}$, let $k>k_0$ be the first
index with $n_kt\ge a$. Then
\[
 a\le n_kt\le Cn_{k-1}t<Ca=2\pi-a.
\]
Thus $n_kt<Ca$ and $|1-e^{in_kt}|\ge2\sin(a/2)$, whence
$\alpha(\mathcal N)\ge2\sin(\pi/(C+1))$. Letting $C\downarrow c$
gives \eqref{samp:quotient-lower}.

Suppose next that $c_\infty(\bn)=\infty$. Choose consecutive sampling
times $a_j=n_{k_j}$ and $b_j=n_{k_j+1}$, with $k_j\to\infty$, such
that $b_j/a_j\to\infty$, and put $t_j=(a_jb_j)^{-1/2}$.
Then $t_j\to0$, while
\[
 a_jt_j=\sqrt{a_j/b_j}\longrightarrow0,
 \qquad b_jt_j=\sqrt{b_j/a_j}\longrightarrow\infty.
\]
Fix $B>0$. For all sufficiently large $j$, $b_jt_j>B$.
Because $a_j$ and $b_j$ are consecutive in $\mathcal N$, every
$n\in\mathcal N$ with $nt_j\le B$ then satisfies $n\le a_j$. Hence
\[
 h_{\mathcal N}(B,t_j)\le a_jt_j\longrightarrow0.
\]
The inner lower limit in \eqref{samp:profile} is therefore $0$ for
every $B$, proving $\alpha(\mathcal N)=0$.
\end{proof}

The lower bound \eqref{samp:quotient-lower} explains the threshold in
Corollary~\ref{cor:sparse}(c). The exact value of $\alpha(\mathcal N)$
can, however, be strictly larger.

\begin{example}[quotients tending to one]\label{samp:ratio-one}
If $n_{k+1}/n_k\to1$, then $\alpha(\mathcal N)=2$.
Indeed, \eqref{samp:quotient-lower} gives
$\alpha(\mathcal N)\ge2\sin(\pi/2)=2$, and the reverse inequality is
part of the definition. Such sequences therefore satisfy the implication
in Corollary~\ref{samp:full-threshold}(i).
This includes the full sequence of positive integers and, for every
integer $d\ge1$, the sequence $n_k=(k+1)^d$.
In particular, sampling along the squares preserves the full threshold $2$.
\end{example}

\begin{proposition}[asymptotically geometric sampling]
\label{prop:samp-asymptotic-geometric}
Suppose that $n_{k+1}/n_k\to b$, where $b\ge2$ is an integer.
Then
\begin{equation}\label{samp:geometric-alpha}
 \alpha(\mathcal N)=a_b:=2\sin\frac{\pi}{b+1}.
\end{equation}
\end{proposition}
\begin{proof}
Proposition~\ref{samp:quotients} gives
$\alpha(\mathcal N)\ge a_b$. For the reverse inequality, put
$\beta=2\pi/(b+1)$ and $t_j=\beta/n_j$, so that $t_j\downarrow0$.
If $k\le j$, then $0<n_kt_j\le\beta<\pi$ and therefore
\[
 |1-e^{in_kt_j}|\le2\sin(\beta/2)=a_b.
\]
Fix $B>0$, and choose a number $c$ with $1<c<b$.
For all sufficiently large $j$, we have $n_{k+1}\ge cn_k$ whenever
$k\ge j$. Consequently, if $k=j+\ell>j$ and $n_kt_j\le B$, then
\[
 \beta c^\ell\le\beta\frac{n_{j+\ell}}{n_j}
 =n_{j+\ell}t_j\le B.
\]
Thus only finitely many positive offsets $\ell$ can contribute to
$h_{\mathcal N}(B,t_j)$, with their number bounded independently of
$j$. For each fixed such $\ell$,
\[
 n_{j+\ell}t_j
 =\beta\prod_{r=0}^{\ell-1}\frac{n_{j+r+1}}{n_{j+r}}
 \longrightarrow\beta b^\ell.
\]
The congruence $b^\ell\equiv(-1)^\ell\pmod{b+1}$ implies
\[
 |1-e^{i\beta b^\ell}|=2\sin\frac{\pi}{b+1}=a_b.
\]
Convergence over this finite list of offsets is uniform. Together
with the estimate for $k\le j$, this proves
\[
 \limsup_{j\to\infty}h_{\mathcal N}(B,t_j)\le a_b.
\]
In particular,
$\liminf_{t\downarrow0}h_{\mathcal N}(B,t)\le a_b$.
Since $B>0$ was arbitrary, the definition of $\alpha$ gives the
required upper bound.
\end{proof}

\begin{example}[geometric sequences and their asymptotic perturbations]
\label{samp:geometric}
The preceding proposition applies to $n_k=b^k$, to $n_k=b^k+k$,
and, after omission of an initial segment, to
$n_k=\lfloor cb^k\rfloor$ for every $c>0$.
The initial segment is omitted only to ensure that all terms are
positive and strictly increasing. 
Adding or deleting finitely many terms does not change $\alpha$.
For $b=2,3,4$, the critical values $a_b$ are, respectively,
$\sqrt3$, $\sqrt2$, and $2\sin(\pi/5)$.
Theorem~\ref{samp:exact} shows that, for each of these sequences and
each $1<q<2$, the sampled displacement bound $\le q$ detects the
Ritt property under the peripheral spectral assumption precisely
when $q<a_b$. The implication fails at $q=a_b$.

For $b\ge5$, one has $a_b\le1$, so no bound with $1<q<2$ has
this implication. The value of $\alpha$ alone does not determine
what happens at $q=1$. For the exact geometric sequence, the
normal contraction in Theorem~\ref{thm:sharp}(i) has sampled
displacement supremum $1$ and is not Ritt. 
\end{example}

\begin{example}[unbounded quotients]\label{samp:unbounded}
If $\sup_k n_{k+1}/n_k=\infty$, then
$c_\infty(\bn)=\infty$, since each finite initial list of quotients
is bounded. Proposition~\ref{samp:quotients} gives
$\alpha(\mathcal N)=0$. Thus for every $1<q<2$ the construction
\eqref{samp:damping-choice}--\eqref{samp:diagonal} yields a normal
contraction $T$ with $\sigma(T)\cap\T=\{1\}$ and
$\sup_{n\in\mathcal N}\|\Id-T^n\|\le q$ which is not Ritt.
This recovers the displacement assertion of Theorem~\ref{thm:unbounded}.
Examples include $n_k=(k+1)!$, whose successive quotients are $k+2$,
and the sequence obtained by adjoining $1$ to $\{2^{2^k}:k\ge0\}$,
whose successive quotients tend to infinity.
\end{example}

\begin{example}[the full threshold without quotients tending to one]
\label{samp:blocks}
Let $\bn$ be the increasing enumeration of
\begin{equation}\label{samp:block-sequence}
 \mathcal N=\bigcup_{j\ge0}
 \bigl([4^j,2\cdot4^j]\cap\N\bigr).
\end{equation}
Then $c_\infty(\bn)=2$, but $\alpha(\mathcal N)=2$.
Within each block successive quotients are $(n+1)/n$ and tend to $1$
as the blocks move to infinity. Every jump between blocks has quotient
$4^{j+1}/(2\cdot4^j)=2$, proving the assertion about $c_\infty$.

To compute $\alpha$, fix $0<t<1$ and choose $j\ge0$ so that
\[
 2\pi\le a:=4^jt<8\pi.
\]
The interval $[a,2a]$ has length at least $2\pi$, so it contains an
odd multiple $u$ of $\pi$. Choose an integer $n$ nearest to $u/t$.
Since both endpoints of $[4^j,2\cdot4^j]$ are integers, this choice
may be made within that interval. Thus $n\in\mathcal N$ and
\[
 |nt-u|\le t/2,\qquad nt\le2a<16\pi.
\]
It follows that
\[
 h_{\mathcal N}(16\pi,t)\ge|1-e^{int}|
 \ge2\cos(t/4)\longrightarrow2
 \qquad(t\downarrow0).
\]
Hence $\alpha(\mathcal N)=2$, and
Corollary~\ref{samp:full-threshold} applies.
In particular, $n_{k+1}/n_k\to1$ is sufficient but not necessary for
the full threshold. Moreover, the value $c_\infty(\bn)=2$ occurs both
here and for the dyadic sequence, although their values of $\alpha$
are $2$ and $\sqrt3$, respectively. The asymptotic quotient bound alone
therefore does not determine the exact displacement threshold.
\end{example}

\begin{remark}[the endpoint one and comparison with difference sampling]
\label{samp:endpoints}
For $q<1$, Lemma~\ref{lem:subone} gives $T=\Id$ for every sampling
sequence under the hypotheses considered here. At $q=1$, the sufficiency
proof of Theorem~\ref{samp:exact} still applies whenever
$\alpha(\mathcal N)>1$. Its necessity proof uses $q>1$ in
\eqref{samp:damping-choice}, so the endpoint requires a separate analysis. Section~\ref{sec:endpoint}
gives sufficient conditions on Banach spaces, contrasting examples, and
an exact criterion for normal contractions. Example~\ref{samp:geometric} does
provide endpoint counterexamples for every $b\ge5$.

Finally, Proposition~\ref{samp:quotients} shows that positivity of
$\alpha(\mathcal N)$ is equivalent to sequences with bounded quotients, which is the same as the condition
in Theorem~\ref{thm:rittsampling} for detecting the Ritt property through
sampled discrete differences. For displacement bounds with $1<q<2$,
the stronger, exact requirement is $q<\alpha(\mathcal N)$.
\end{remark}

\section{Endpoint displacement bounds}\label{sec:endpoint}
The endpoint bound $1$ can impose rigidity even when the quotients of
the sampling sequence are unbounded. We first give a Banach-space
criterion, then characterise the endpoint for normal contractions.

\begin{theorem}[Sectoriality from sampled displacement bounds]
\label{thm:endpoint-sectoriality}
Let $\mathcal N$ be an infinite set of positive integers containing $1$,
and define
\[
 \Theta_{\mathcal N}
 :=
 \bigl\{\theta\in[-\pi,\pi]:
       \cos(n\theta)\ge0 \text{ for every }n\in\mathcal N\bigr\},
 \qquad
 \beta_{\mathcal N}
 :=
 \max_{\theta\in\Theta_{\mathcal N}}|\theta|.
\]
Let $X$ be a complex Banach space and let $T\in\B(X)$.
Suppose that
\begin{equation}\label{eq:endpoint-sampled-bound}
 \sup_{n\in\mathcal N}\|\Id-T^n\|\le1.
\end{equation}
Then the following assertions hold.
\begin{enumerate}
\item
The operator $T$ is sectorial of angle $\beta_{\mathcal N}$.
More precisely,
\[
 \sigma(T)\subset
 \{0\}\cup
 \bigl\{re^{i\theta}:0<r\le1,\ \theta\in\Theta_{\mathcal N}\bigr\},
\]
and, writing
\[
 \Sigma_\omega
 :=
 \{z\in\C\setminus\{0\}:|\arg z|<\omega\},
 \qquad 0<\omega<\pi,
\]
we have
\[
 \sup_{\substack{z\ne0\\z\notin\Sigma_\omega}}
 |z|\,\|(z\Id-T)^{-1}\|<\infty
 \qquad(\beta_{\mathcal N}<\omega<\pi).
\]

\item
Suppose, in addition, that
\begin{equation}\label{eq:endpoint-phase-condition}
 \lambda\in\T,\qquad
 \sup_{n\in\mathcal N}|1-\lambda^n|\le\sqrt2
 \quad\Longrightarrow\quad \lambda=1.
\end{equation}
Then $T$ is sectorial of angle zero,
$\sigma(T)\subset[0,1]$, and
\begin{equation}\label{eq:endpoint-all-powers}
 \sup_{m\ge1}\|\Id-T^m\|\le1.
\end{equation}
In particular, $T$ is a Ritt operator.

\item
Condition~\eqref{eq:endpoint-phase-condition} is compatible with
unbounded quotients. For example, put
\[
 A_j=2^{2^j}\quad(j\ge1),
 \qquad
 \mathcal N
 =
 \{1\}\cup
 \bigcup_{j\ge1}
 \{A_j,A_j+1,\ldots,A_j+j\}.
\]
Then $\mathcal N$ satisfies
\eqref{eq:endpoint-phase-condition}, whereas its increasing
enumeration $(n_k)_{k\ge0}$ satisfies
\[
 \sup_{k\ge0}\frac{n_{k+1}}{n_k}=\infty.
\]
Consequently, there is a sequence $\mathcal N$ with unbounded quotients such that every operator $T$ satisfying \eqref{eq:endpoint-sampled-bound} is Ritt.
\end{enumerate}
\end{theorem}

\begin{proof}
For $z\in\rho(T)$, write $R(z,T)=(z\Id-T)^{-1}$.

We first prove the spectral inclusion in~(1).
Let $z=re^{i\theta}\in\sigma(T)\setminus\{0\}$.
The sampled estimate implies $\|T^n\|\le2$ for $n\in\mathcal N$.
Consequently $r(T)\le2^{1/n}$ for arbitrarily large $n$, and hence
$r(T)\le1$. Thus $r\le1$, and spectral mapping together with
\eqref{eq:endpoint-sampled-bound} gives
\[
 |1-z^n|^2
 =
 1+r^{2n}-2r^n\cos(n\theta)
 \le1
 \qquad(n\in\mathcal N).
\]
Thus
\begin{equation}\label{eq:endpoint-spectral-angle}
 \cos(n\theta)\ge\frac{r^n}{2}>0
 \qquad(n\in\mathcal N),
\end{equation}
so $\theta\in\Theta_{\mathcal N}$.
Notice also that $1\in\mathcal N$ implies
$\Theta_{\mathcal N}\subset[-\pi/2,\pi/2]$, and hence
$0\le\beta_{\mathcal N}\le\pi/2$.

To obtain the sectorial resolvent bound, we use an elementary moment
inequality. Since $\|\Id-T\|\le1$,
\[
 \|e^{-tT}\|
 =
 e^{-t}\|e^{t(\Id-T)}\|
 \le1
 \qquad(t\ge0).
\]
Taylor's formula yields
\[
 e^{-hT}x
 =
 x-hTx+\int_0^h(h-s)e^{-sT}T^2x\,ds
 \qquad(h>0),
\]
and therefore
\[
 \|Tx\|
 \le\frac{2}{h}\|x\|+\frac h2\|T^2x\|.
\]
Optimising over $h>0$ gives the following moment inequality
\begin{equation}\label{eq:endpoint-kallman-rota}
 \|Tx\|^2\le4\|x\|\,\|T^2x\|.
\end{equation}
For a fixed $y\in X$, apply this inequality to $T^{j-1}y$.
Writing $a_j=\|T^jy\|$, we obtain
\[
 a_j^2\le4a_{j-1}a_{j+1}.
\]
Hence the sequence $b_j=2^{j^2}a_j$ is log-convex. Interpolation
between its zeroth and $n$th terms gives
\begin{equation}\label{eq:endpoint-moment}
 \|T^jy\|
 \le
 2^{j(n-j)}
 \|y\|^{1-j/n}\|T^ny\|^{j/n},
 \qquad 1\le j<n.
\end{equation}
If $\|T^ny\|=0$, inequality
\eqref{eq:endpoint-kallman-rota} propagates this vanishing backwards,
so \eqref{eq:endpoint-moment} remains valid.

Fix $n\in\mathcal N$. Again,
\[
 \|e^{-tT^n}\|\le1\qquad(t\ge0).
\]
The Laplace representation of the resolvent therefore gives
\[
 \|R(w,T^n)\|\le\frac1{-\Re w}
 \qquad(\Re w<0).
\]
Suppose that $z\ne0$ and
\[
 \Re(z^n)\le-\delta|z|^n
\]
for some $\delta>0$. Then
\begin{equation}\label{eq:endpoint-power-resolvent}
 \|R(z^n,T^n)\|\le\delta^{-1}|z|^{-n},
 \qquad
 \|T^nR(z^n,T^n)\|\le1+\delta^{-1}.
\end{equation}
Applying \eqref{eq:endpoint-moment} to
$y=R(z^n,T^n)x$ gives
\[
 \|T^jR(z^n,T^n)\|
 \le
 2^{j(n-j)}
 \delta^{-1}(1+\delta)^{j/n}|z|^{j-n},
 \qquad 1\le j<n.
\]
The same estimate for $j=0$ is simply the first inequality in
\eqref{eq:endpoint-power-resolvent}.

Since $z^n\in\rho(T^n)$, spectral mapping gives $z\in\rho(T)$.
The polynomial factorisation of $z^n\Id-T^n$ yields
\[
 R(z,T)
 =
 \sum_{j=0}^{n-1}
 z^{n-1-j}T^jR(z^n,T^n).
\]
Consequently,
\begin{equation}\label{eq:endpoint-resolvent-transfer}
 |z|\,\|R(z,T)\|
 \le
 K_{n,\delta}
 :=
 \delta^{-1}
 \sum_{j=0}^{n-1}
 2^{j(n-j)}(1+\delta)^{j/n}.
\end{equation}
This estimate is uniform in $|z|$.

Let $\beta_{\mathcal N}<\omega<\pi$. For every
$\theta$ in the compact set
\[
 K_\omega
 :=
 \{\theta\in[-\pi,\pi]:|\theta|\ge\omega\},
\]
the definition of $\Theta_{\mathcal N}$ provides an exponent
$n\in\mathcal N$ such that $\cos(n\theta)<0$.
By continuity and compactness, there are finitely many exponents
$p_1,\ldots,p_s\in\mathcal N$ and a number $\delta>0$ such that
\[
 \min_{1\le i\le s}\cos(p_i\theta)\le-\delta
 \qquad(\theta\in K_\omega).
\]
For every non-zero $z\notin\Sigma_\omega$, we may therefore apply
\eqref{eq:endpoint-resolvent-transfer} with one of these exponents.
It follows that
\[
 \sup_{\substack{z\ne0\\z\notin\Sigma_\omega}}
 |z|\,\|R(z,T)\|
 \le
 \max_{1\le i\le s}K_{p_i,\delta}
 <\infty.
\]
This proves~(1).

Suppose now that \eqref{eq:endpoint-phase-condition} holds.
The identity
\[
 |1-e^{in\theta}|^2=2-2\cos(n\theta)
\]
shows that this condition is equivalent to
$\Theta_{\mathcal N}=\{0\}$. Assertion~(1) therefore implies that
$T$ is sectorial of angle zero and $\sigma(T)\subset[0,1]$.

Since $T$ is sectorial of angle zero, every positive integer power $T^n$
is also sectorial of angle zero. The composition rule for sectorial
fractional powers, see \cite[Chapter~3]{Haase}, therefore gives, for all
positive integers $m,n$,
\begin{equation}\label{eq:endpoint-power-composition}
 (T^n)^{m/n}=T^m.
\end{equation}
Here all fractional powers are defined by the sectorial functional
calculus, and the identity remains valid when $0\in\sigma(T)$. 
We note that the same idea of using fractional powers occurred in the proof of \cite[Theorem 2.11]{BDG}.

For completeness, we verify the displacement estimate for these
fractional powers at the endpoint of the norm condition.
Fix $m\ge1$, choose $n\in\mathcal N$ with $n>m$, and put
\[
 \alpha=\frac mn,\qquad A=T^n,\qquad S=\Id-A.
\]
Thus $0<\alpha<1$ and $\|S\|\le1$. The Balakrishnan formula gives
\[
 A^\alpha
 =
 \frac{\sin(\pi\alpha)}{\pi}
 \int_0^\infty
 t^{\alpha-1}A(t\Id+A)^{-1}\,dt.
\]
For $t>0$, the Neumann series yields
\[
 A(t\Id+A)^{-1}
 =
 \frac1{1+t}\Id
 -
 \sum_{j\ge1}\frac{t}{(1+t)^{j+1}}S^j.
\]
Termwise integration is justified by
\[
 \sum_{j\ge1}\frac{t}{(1+t)^{j+1}}
 =
 \frac1{1+t},
 \qquad
 \int_0^\infty\frac{t^{\alpha-1}}{1+t}\,dt<\infty.
\]
We obtain the norm-convergent series
\[
 A^\alpha
 =
 \Id-\sum_{j\ge1}a_j^{(\alpha)}S^j,
\]
where
\[
 a_j^{(\alpha)}
 =
 \frac{\sin(\pi\alpha)}{\pi}
 \int_0^\infty\frac{t^\alpha}{(1+t)^{j+1}}\,dt
 =
 (-1)^{j+1}\binom{\alpha}{j}>0.
\]
Moreover,
\[
 \sum_{j\ge1}a_j^{(\alpha)}
 =
 \frac{\sin(\pi\alpha)}{\pi}
 \int_0^\infty\frac{t^{\alpha-1}}{1+t}\,dt
 =1.
\]
It follows that
\[
 \|\Id-A^\alpha\|
 \le
 \sum_{j\ge1}a_j^{(\alpha)}\|S^j\|
 \le1.
\]
Using \eqref{eq:endpoint-power-composition}, we conclude that
$\|\Id-T^m\|\le1$. Since $m$ was arbitrary,
\eqref{eq:endpoint-all-powers} follows.
Theorem~\ref{thm:sqrt3} now implies that $T$ is Ritt, proving~(2).

Finally, consider the set $\mathcal N$ in~(3).
The successive blocks are disjoint, and the quotient between the
first element of one block and the last element of the preceding
block is
\[
 \frac{A_{j+1}}{A_j+j}
 =
 \frac{A_j^2}{A_j+j}
 \longrightarrow\infty.
\]
Thus its increasing enumeration has unbounded quotients.

To verify \eqref{eq:endpoint-phase-condition}, suppose that
$\lambda\in\T$ satisfies
\[
 \Re(\lambda^n)\ge0\qquad(n\in\mathcal N).
\]
Since $1\in\mathcal N$, we have $\lambda\ne-1$.
Suppose that $\lambda\ne1$. Then $\lambda^2\ne1$.
For each $j\ge1$, put
\[
 c_\ell=\Re(\lambda^{A_j+\ell}),
 \qquad 0\le\ell\le j.
\]
The definition of $\mathcal N$ gives $0\le c_\ell\le1$, and hence
\[
 \sum_{\ell=0}^j c_\ell^2
 \le
 \sum_{\ell=0}^j c_\ell.
\]
The geometric-series formula gives
\[
 \sum_{\ell=0}^j c_\ell
 \le\frac2{|1-\lambda|}.
\]
On the other hand,
\[
 \sum_{\ell=0}^j c_\ell^2
 =
 \frac{j+1}{2}
 +
 \frac12\Re\left(
   \lambda^{2A_j}
   \sum_{\ell=0}^j\lambda^{2\ell}
 \right)
 \ge
 \frac{j+1}{2}-\frac1{|1-\lambda^2|}.
\]
These inequalities contradict one another as $j\to\infty$.
Therefore $\lambda=1$, proving
\eqref{eq:endpoint-phase-condition}.

Assertion~(2) applies to this set $\mathcal N$, so every bounded
operator satisfying \eqref{eq:endpoint-sampled-bound} is Ritt.
This proves~(3).
\end{proof}

\begin{remark}[the Hilbert-space conclusion]
\label{rem:endpoint-positive}
Under \eqref{eq:endpoint-phase-condition}, a stronger conclusion is
available on Hilbert space. Indeed, \cite[Corollary~6.1]{BDG} states
that every Hilbert-space operator $T$ satisfying
\begin{equation}\label{eq:BDG-gap}
 \sup_{n\in\mathcal N}\|\Id-T^n\|\le1
\end{equation}
is a positive contraction. It is therefore Ritt; for example, the
spectral theorem gives
\[
 m\|T^{m-1}(\Id-T)\|
 \le\max_{0\le s\le1}m s^{m-1}(1-s)\le1
 \qquad(m\ge1).
\]
Theorem~\ref{thm:endpoint-sectoriality} provides a Banach-space
conclusion under the same arithmetic condition, without any power
boundedness assumption.
\end{remark}

\begin{proposition}[divisibility chains at the endpoint]
\label{prop:endpoint-divisibility}
Let $(n_k)$ be a strictly increasing sequence of positive integers
such that $n_k$ divides $n_{k+1}$ for every $k$.
Suppose that $n_j/n_{j-1}\ge6$ for infinitely many indices $j$.
Then there is a normal contraction $T$ on $\ell^2$ such that
\[
 \sigma(T)\cap\T=\{1\},\qquad
 \sup_k\|\Id-T^{n_k}\|=1,
\]
but $T$ is not Ritt. In particular, this conclusion holds for every
divisibility chain with unbounded successive quotients.
\end{proposition}
\begin{proof}
Choose strictly increasing indices $j_l$ such that
$n_{j_l}/n_{j_l-1}\ge6$, and put
\[
 t_l=\frac{2\pi}{n_{j_l}},\qquad
 z_l=e^{-t_l^2+it_l},\qquad Te_l=z_le_l.
\]
Then $t_l\downarrow0$, $|z_l|<1$, and $z_l\to1$, so $T$ is a
normal contraction with
$\sigma(T)=\{z_l:l\ge1\}\cup\{1\}$.
For $k<j_l$ we have
\[
 0<n_kt_l\le2\pi\frac{n_{j_l-1}}{n_{j_l}}\le\frac\pi3,
\]
and hence $|1-e^{in_kt_l}|\le1$.
For $k\ge j_l$, divisibility gives $e^{in_kt_l}=1$.
Thus, in both cases, writing $r=e^{-n_kt_l^2}$ and using radial
convexity, we obtain
\[
 |1-z_l^{n_k}|
 \le(1-r)+r|1-e^{in_kt_l}|\le1.
\]
Normality gives $\sup_k\|\Id-T^{n_k}\|\le1$.
For any fixed $l$, $z_l^{n_k}\to0$ as $k\to\infty$, proving
the reverse inequality.

Finally, for all sufficiently large $l$, $0<t_l<\pi$ and
\[
 |e^{it_l}-1|\,\|\Res(e^{it_l},T)\|
 \ge\frac{|e^{it_l}-1|}{1-e^{-t_l^2}}
 \ge\frac{2}{\pi t_l}\longrightarrow\infty.
\]
Here $e^{it_l}\in\rho(T)$ because the only unimodular point of
$\sigma(T)$ is $1$. Proposition~\ref{prop:boundary} therefore
shows that $T$ is not Ritt.
\end{proof}

\begin{example}[factorial sampling]\label{ex:endpoint-factorial}
For $\mathcal N=\{k!:k\ge1\}$, an explicit choice in the preceding
proof is
\[
 t_j=\frac{2\pi}{j!},\qquad z_j=e^{-t_j^2+it_j}
 \quad(j\ge6).
\]
Indeed, if $k<j$, then $k!t_j\le2\pi/j\le\pi/3$; if $k\ge j$,
then $k!t_j$ is an integer multiple of $2\pi$.
The diagonal operator with entries $z_j$ consequently satisfies
the endpoint bound with equality and is not Ritt.
The same conclusion holds for
$\{1\}\cup\{2^{2^j}:j\ge1\}$, by
Proposition~\ref{prop:endpoint-divisibility}.
In contrast, adjoining the blocks of consecutive integers in
Theorem~\ref{thm:endpoint-sectoriality}(3) makes the endpoint
bound imply the Ritt property for every bounded operator.
\end{example}

\subsection{An exact endpoint criterion for normal contractions}
\label{subsec:normal-endpoint}
For an infinite set $\mathcal N$ of positive integers, define
the extended nonnegative function
\begin{equation}\label{eq:endpoint-kappa}
 \kappa_{\mathcal N}(t)
 :=\sup_{n\in\mathcal N}
 \frac{[-\log(2\cos nt)]_+}{n}
 \qquad(t\in\mathbb R).
\end{equation}
Here $[x]_+=\max\{x,0\}$, and the expression is defined to be
$+\infty$ if $\cos(nt)\le0$ for at least one $n\in\mathcal N$.
It may also equal $+\infty$ when every cosine is positive but
the supremum diverges. The function is even and $2\pi$-periodic,
and $\kappa_{\mathcal N}(0)=0$.

\begin{theorem}[the normal endpoint criterion]
\label{thm:normal-endpoint}
The following assertions are equivalent.
\begin{enumerate}[label=\textup{(\roman*)}]
\item Every normal contraction $T$ on a complex Hilbert space
such that
\[
 \sigma(T)\cap\T\subset\{1\},\qquad
 \sup_{n\in\mathcal N}\|\Id-T^n\|\le1
\]
is a Ritt operator.
\item The same assertion holds for diagonal contractions on $\ell^2$.
\item The scalar function $\kappa_{\mathcal N}$ satisfies
\begin{equation}\label{eq:endpoint-kappa-condition}
 \liminf_{t\downarrow0}\frac{\kappa_{\mathcal N}(t)}{t}>0.
\end{equation}
\end{enumerate}
The value $+\infty$ is allowed in \eqref{eq:endpoint-kappa-condition}.
If this condition fails, the diagonal operator in \textup{(ii)} can
be chosen with sampled displacement supremum exactly $1$.
\end{theorem}
\begin{proof}
We first record the scalar meaning of $\kappa_{\mathcal N}$.
For $a\ge0$, $t\in\mathbb R$, and $n\ge1$,
\[
 |1-e^{-an+int}|^2
 =1+e^{-2an}-2e^{-an}\cos(nt).
\]
Since $e^{-an}>0$, the inequality $|1-e^{-an+int}|\le1$ is
equivalent to
\[
 e^{-an}\le2\cos(nt).
\]
It is impossible when $\cos(nt)\le0$; otherwise it is equivalent,
under the condition $a\ge0$, to
$a\ge[-\log(2\cos nt)]_+/n$.
Taking all $n\in\mathcal N$ gives
\begin{equation}\label{eq:endpoint-scalar-kappa}
 \sup_{n\in\mathcal N}|1-e^{-an+int}|\le1
 \quad\Longleftrightarrow\quad
 a\ge\kappa_{\mathcal N}(t).
\end{equation}
The point $0$ separately satisfies all the scalar displacement
inequalities.

Assume (iii). There exist $c>0$ and $0<\delta<\pi$ such that
\[
 \kappa_{\mathcal N}(t)\ge c|t|
 \qquad(0<|t|<\delta),
\]
where evenness handles negative $t$.
Let $T$ satisfy the assumptions in (i).
The spectral theorem and \eqref{eq:endpoint-scalar-kappa} imply
that every nonzero $\lambda=e^{-a+it}\in\sigma(T)$, with
$a\ge0$ and $-\pi\le t\le\pi$, satisfies
$a\ge\kappa_{\mathcal N}(t)$.
If $0<|t|<\delta$, then $a\ge c|t|>0$, and
\begin{align*}
 \frac{|1-\lambda|}{1-|\lambda|}
 &\le 1+\frac{e^{-a}|1-e^{it}|}{1-e^{-a}}\\
 &\le1+\frac{|t|}{e^a-1}
 \le1+\frac{|t|}{a}
 \le1+\frac1c.
\end{align*}
For $t=0$ and $0\le\lambda<1$, the same ratio equals $1$.

It remains to consider spectral points whose principal arguments
satisfy $|t|\ge\delta$. Their closure, with $0$ included if
necessary, is a compact subset of $\sigma(T)$ disjoint from
$\T$: the peripheral spectral assumption excludes every
unimodular point in this angular region. Its moduli therefore
have a maximum $r_0<1$, unless the region is empty.
On this set,
\[
 \frac{|1-\lambda|}{1-|\lambda|}\le\frac2{1-r_0}.
\]
We have proved that, for some finite $K$, the following Stolz condition
\begin{equation}\label{eq:endpoint-stolz-spectrum}
 |1-\lambda|\le K(1-|\lambda|)
 \qquad(\lambda\in\sigma(T)).
\end{equation}
This includes $\lambda=1$, where both sides vanish.
For $|z|>1$ and $\lambda\in\sigma(T)$, the triangle inequality
and \eqref{eq:endpoint-stolz-spectrum} give
\[
 |z-1|\le|z-\lambda|+K(1-|\lambda|)
 \le(1+K)|z-\lambda|,
\]
since $1-|\lambda|\le|z-\lambda|$.
Normality now yields
\[
 |z-1|\,\|\Res(z,T)\|
 =\sup_{\lambda\in\sigma(T)}\frac{|z-1|}{|z-\lambda|}
 \le1+K.
\]
Thus $T$ is Ritt, proving (iii)$\Rightarrow$(i).
The implication (i)$\Rightarrow$(ii) is immediate.

Suppose that (iii) fails. Nonnegativity implies that the lower
limit in \eqref{eq:endpoint-kappa-condition} is zero.
Choose $0<t_j\downarrow0$ with finite $\kappa_{\mathcal N}(t_j)$
such that
\[
 \frac{\kappa_{\mathcal N}(t_j)}{t_j}\longrightarrow0,
 \qquad
 a_j:=\kappa_{\mathcal N}(t_j)+t_j^2.
\]
Then $a_j>0$, $a_j/t_j\to0$, and $a_j\to0$.
Define $Te_j=z_je_j$ on $\ell^2$, where $z_j=e^{-a_j+it_j}$.
This is a normal contraction with
$\sigma(T)=\{z_j:j\ge1\}\cup\{1\}$ and
$\sigma(T)\cap\T=\{1\}$.
By \eqref{eq:endpoint-scalar-kappa}, every coordinate satisfies
all sampled displacement bounds, so
$\sup_{n\in\mathcal N}\|\Id-T^n\|\le1$.
For a fixed $j$, $z_j^n\to0$ as $n\to\infty$ in $\mathcal N$;
therefore this supremum is exactly $1$.
For large $j$,
\[
 |e^{it_j}-1|\,\|\Res(e^{it_j},T)\|
 \ge\frac{|e^{it_j}-1|}{1-e^{-a_j}}
 \ge\frac{2t_j}{\pi a_j}\longrightarrow\infty.
\]
Proposition~\ref{prop:boundary} shows that $T$ is not Ritt.
This contradicts (ii), proving (ii)$\Rightarrow$(iii).
\end{proof}

\begin{remark}[phase separation and the normal criterion]
\label{rem:phase-kappa}
Under \eqref{eq:endpoint-phase-condition},
$\Theta_{\mathcal N}=\{0\}$.
For every $0<|t|\le\pi$, some $n\in\mathcal N$ then satisfies
$\cos(nt)<0$, and hence $\kappa_{\mathcal N}(t)=+\infty$.
Thus the arithmetic hypothesis of
Theorem~\ref{thm:endpoint-sectoriality} implies the normal
criterion. Theorem~\ref{thm:normal-endpoint} requires only a
positive lower bound for $\kappa_{\mathcal N}(t)/|t|$ near zero.
\end{remark}

\section{Converses on uniformly convex spaces}\label{sec:UC}

By Theorem~\ref{thm:BGTpowers} and Proposition~\ref{prop:boundary}, a Ritt operator satisfies $\sup_n\|(\zeta-T^n)^{-1}\|\le C'(T)/|\zeta-1|$ for $\zeta\in\T\setminus\{1\}$, where $C'(T)=\sup_m\CR(T^m)$ \cite[Corollary~7.7]{BGT}. Applied with $\zeta=-1$, this controls the distance from $T^n$ to $-\Id$; on a uniformly convex space this in turn controls the distance from $T^n$ to $\Id$ through the modulus of convexity
\[
   \delta_X(\varepsilon):=\inf\Big\{1-\Big\|\frac{x+y}2\Big\|:\ \|x\|\le1,\ \|y\|\le1,\ \|x-y\|\ge\varepsilon\Big\},\qquad0<\varepsilon\le2,
\]
which is positive for all $\varepsilon$ exactly when $X$ is uniformly convex.

\begin{theorem}\label{thm:UC}
Let $X$ be uniformly convex, let $T\in\B(X)$ be a contraction, and suppose that $\sup_{n\ge1}\|(\Id+T^n)^{-1}\|\le K$. Then
\[
   \sup_{n\ge1}\|\Id-T^n\|\le2\big(1-\delta_X(1/K)\big)<2 .
\]
\end{theorem}

\begin{proof}
For $\|x\|=1$ put $y=-T^nx$; then $\|y\|\le1$ and $\|x-y\|=\|x+T^nx\|\ge1/K$, so $\|(x-T^nx)/2\|=\|(x+y)/2\|\le1-\delta_X(1/K)$.
\end{proof}

Together with the uniform resolvent bounds for powers of Ritt operators,
this gives the zero--two conclusion of \cite[Theorem~7.9]{BGT} with an
explicit dependence on the modulus of convexity. On Hilbert space,
$\delta_H(\varepsilon)=1-\sqrt{1-\varepsilon^2/4}$ and
Theorem~\ref{thm:rootsofzeta} give an explicit gap. A direct use of the
Fourier estimate yields the bound in Corollary~\ref{cor:hilbertdisp}.

We can now characterise the contractions on uniformly convex spaces whose iterates have displacement uniformly less than $2$.

\begin{theorem}\label{thm:UCchar}
Let $X$ be uniformly convex and let $T\in\B(X)$ be a contraction. The following assertions are equivalent:
\begin{enumerate}[label=\textup{(\roman*)}]
\item $\sup_{n\ge1}\|\Id-T^n\|<2$;
\item $T^N$ is a Ritt operator for some odd positive integer $N$;
\item $T$ is a $\Ritt{E}$ operator for some finite nonempty set $E\subset\T$ of roots of unity of odd order.
\end{enumerate}
The implication \textup{(i)}$\Rightarrow$\textup{(ii)},\textup{(iii)} holds for every bounded operator on every Banach space, with $N=N(q)$ and $E=E_q$ when $q=\sup_n\|\Id-T^n\|\ge1$; if $q<1$, then $T=\Id$, and one takes $N=1$, $E=\{1\}$. In \textup{(ii)}$\Rightarrow$\textup{(i)},
\[
   \sup_{n\ge1}\|\Id-T^n\|\le2\big(1-\delta_X(1/(NK_+))\big),\qquad K_+:=\sup_{n\ge1}\|(\Id+T^{nN})^{-1}\|<\infty .
\]
\end{theorem}

\begin{proof}
(i)$\Rightarrow$(ii),(iii) and (iii)$\Rightarrow$(ii) are contained in Theorem~\ref{thm:oddN} and its last assertion.

(ii)$\Rightarrow$(i). Put $S=T^N$. By Theorem~\ref{thm:BGTpowers} and Proposition~\ref{prop:boundary} applied to the Ritt operators $S^n$, $K_+=\sup_n\|(\Id+S^n)^{-1}\|\le\frac12\sup_n\CR(S^n)<\infty$. Since $N$ is odd,
\[
 \Id+T^{nN}=(\Id+T^n)\sum_{j=0}^{N-1}(-T^n)^j,
\]
with commuting factors, so $\Id+T^n$ is invertible with $(\Id+T^n)^{-1}=\sum_{j<N}(-T^n)^j(\Id+T^{nN})^{-1}$ and
$\|(\Id+T^n)^{-1}\|\le NK_+$, because $T$ is a contraction.
Theorem~\ref{thm:UC} gives the explicit estimate.
\end{proof}

\begin{remarks}
(a) On Hilbert space all constants are explicit. If $T$ is a contraction and $T^N$ is Ritt with $N$ odd, Theorem~\ref{thm:rootsofzeta} applied to $S=T^N$ gives $K_+\le1+\CR(T^N)/2$, hence $\sup_n\|(\Id+T^n)^{-1}\|\le N(2+\CR(T^N))/2$ and, by the parallelogram law (see Corollary~\ref{cor:hilbertdisp}),
\[
   \sup_n\|\Id-T^n\|\le(4-\theta^2)^{1/2},\qquad\theta=\frac{2}{N(2+\CR(T^N))}.
\]
Another bound follows from Proposition~\ref{prop:rootsN}: the same displacement estimate holds with $\theta=2/(2+NK_N(T))$. The constants $K_N(T)$ and $\CR(T^N)$ are related by Proposition~\ref{prop:TN}.

(b) Neither hypothesis of Theorem~\ref{thm:UCchar} can be dropped. An oblique projection $P$ on a Hilbert space is Ritt ($P^n=P$), but $\|\Id-P\|=\|P\|$ is arbitrary; and on $\C^2$ with the norm $\|(x,y)\|=\max\{|x|,|x+y|\}$ the contractive projection $P(x,y)=(x,0)$ has $\|\Id-P\|=2$ \cite[Example~7.10]{BGT}.

(c) The equivalence (i)$\iff$(ii) fails for even $N$ ($T=-\Id$), and the Ritt property of $T$ itself is not implied by (i) unless $\sigma(T)\cap\T\subset\{1\}$ (Theorem~\ref{thm:local}).
\end{remarks}

\section{Quantitative estimates on Hilbert and \texorpdfstring{$L^p$}{Lp} spaces}\label{sec:hilbert}

\subsection{Roots of \texorpdfstring{$\zeta$}{zeta}}
We shall use the classical identity
\begin{equation}\label{eq:csc}
   \sum_{k=0}^{n-1}\csc^2\Big(t+\frac{k\pi}{n}\Big)=n^2\csc^2(nt),\qquad nt\notin\pi\Z.
\end{equation}
To prove this identity, start from $\sin(nt)=2^{n-1}\prod_{k=0}^{n-1}\sin(t+k\pi/n)$ and then take the logarithmic derivative. One obtains $\sum_{k=0}^{n-1}\cot(t+k\pi/n)=n\cot(nt)$. Now differentiate. 

The Mittag--Leffler expansion gives a second proof. 

\begin{remark}[a second proof of \eqref{eq:csc}]\label{rem:cscML}
Identity \eqref{eq:csc} also follows directly from the Mittag--Leffler
expansion
\[
   \csc^2u=\sum_{m\in\Z}\frac1{(u-m\pi)^2}\qquad(u\notin\pi\Z),
\]
which converges absolutely and locally uniformly off $\pi\Z$. Indeed,
\[
   \sum_{k=0}^{n-1}\csc^2\Big(t+\frac{k\pi}{n}\Big)
   =\sum_{k=0}^{n-1}\sum_{m\in\Z}\frac1{\big(t+\frac{(k-mn)\pi}{n}\big)^2}
   =\sum_{j\in\Z}\frac1{\big(t+\frac{j\pi}{n}\big)^2}
   \]
   \[
   =n^2\sum_{j\in\Z}\frac1{(nt+j\pi)^2}
   =n^2\csc^2(nt),
\]
the middle step being the bijection $(k,m)\mapsto k-mn$ from
$\{0,\dots,n-1\}\times\Z$ onto $\Z$ (division with remainder), and absolute
convergence justifying the rearrangement. This proof also explains the
identity: the left-hand side has a double pole with principal part $(t-s)^{-2}$
at every $s\in\frac{\pi}{n}\Z$ and no others, and so does the right-hand side;
their difference is therefore entire. It is $\pi/n$-periodic and tends
uniformly to $0$ as $|\operatorname{Im}t|\to\infty$ in a vertical strip
of width $\pi/n$, so it is bounded on the plane. Liouville's theorem
then makes it constant, and its limit at imaginary infinity shows that
the constant is $0$.
\end{remark}

\begin{lemma}\label{lem:rootsum}
Let $n\ge1$, $\zeta\in\T$, and let $\mu_0,\dots,\mu_{n-1}$ be the $n$-th roots of $\zeta$. For every $\xi\in\T$ with $\xi^n\ne\zeta$,
\[
   \sum_{j=0}^{n-1}\frac{1}{|\mu_j-\xi|^2}=\frac{n^2}{|\zeta-\xi^n|^2}.
\]
\end{lemma}

\begin{proof}
Replacing $\mu_j$ by $\bar\xi\mu_j$ and $\zeta$ by $\bar\xi^n\zeta$ we may assume $\xi=1$. Write $\zeta=e^{i\beta}$, $0<\beta<2\pi$, and $\mu_j=e^{i(\beta+2\pi j)/n}$. Then $|\mu_j-1|^2=4\sin^2\big(\frac{\beta}{2n}+\frac{j\pi}{n}\big)$, and \eqref{eq:csc} with $t=\beta/(2n)$ gives $\sum_j|\mu_j-1|^{-2}=\frac14n^2\csc^2(\beta/2)=n^2/|\zeta-1|^2$.
\end{proof}

\begin{theorem}\label{thm:rootsofzeta}
Let $H$ be a Hilbert space, let $T\in\B(H)$ be a Ritt operator, $C=\CR(T)$, $M=M(T)$, and let $\zeta\in\T\setminus\{1\}$. Then for all $n\ge1$ and $x\in H$,
\begin{equation}\label{eq:zetabound}
   \|(\zeta-T^n)x\|\ge\frac{|\zeta-1|}{|\zeta-1|+CM}\,\|x\|,
   \qquad\text{hence}\qquad
   \sup_{n\ge1}\|(\zeta-T^n)^{-1}\|\le1+\frac{CM}{|\zeta-1|}.
\end{equation}
\end{theorem}

\begin{proof}
Fix $n\ge1$ and $x\in H$ with $\|x\|=1$, and put $d:=\|(\zeta-T^n)x\|$; we may assume $d<1$. Let $\mu_0,\dots,\mu_{n-1}$ be the $n$-th roots of $\zeta$ and define
\[
   y_j:=\sum_{k=0}^{n-1}\mu_j^{-k}T^kx,\qquad j=0,\dots,n-1 .
\]
A telescoping computation gives
\[
   (T-\mu_j)y_j=\sum_{k=1}^{n}\mu_j^{-(k-1)}T^kx-\sum_{k=0}^{n-1}\mu_j^{1-k}T^kx
   =\mu_j^{-(n-1)}T^nx-\mu_jx=\mu_j\bar\zeta\,(T^n-\zeta)x,
\]
since $\mu_j^{-n}=\bar\zeta$. As $\zeta\ne1$ we have $\mu_j\ne1$, so Proposition~\ref{prop:boundary} yields
\[
   \|y_j\|\le\frac{C}{|\mu_j-1|}\,d .
\]
The matrix $(\mu_j^{-k})_{j,k}$ is $e^{-i\beta k/n}$ times the discrete Fourier matrix $(e^{-2\pi ijk/n})_{j,k}$, where $\zeta=e^{i\beta}$; its columns are orthogonal with squared norm $n$, so
\[
   \sum_{j=0}^{n-1}\|y_j\|^2=n\sum_{k=0}^{n-1}\|T^kx\|^2\ge n^2\min_{0\le k<n}\|T^kx\|^2\ge\frac{n^2}{M^2}\|T^nx\|^2\ge\frac{n^2(1-d)^2}{M^2},
\]
using $\|T^nx\|\le M\|T^kx\|$ for $k\le n$ and $\|T^nx\|\ge\|\zeta x\|-\|(\zeta-T^n)x\|=1-d$. On the other hand, Lemma~\ref{lem:rootsum} with $\xi=1$ gives
\[
   \sum_{j=0}^{n-1}\|y_j\|^2\le C^2d^2\sum_{j=0}^{n-1}\frac1{|\mu_j-1|^2}=\frac{C^2d^2n^2}{|\zeta-1|^2}.
\]
Hence $(1-d)/M\le Cd/|\zeta-1|$, which is \eqref{eq:zetabound}. Since $\sigma(T^n)\cap\T\subset\{1\}$, $\zeta\in\rho(T^n)$ and the inverse bound follows.
\end{proof}

The same argument applies to operators with peripheral spectrum included in the roots of unity.

\begin{proposition}\label{prop:rootsN}
Let $N\ge1$ and let $T$ be a $\Ritt{\mu_N}$ operator on a Hilbert space, with $K=K_N(T)$ and $M=M(T)$. Let $\zeta\in\T\setminus\mu_N$, $n\ge1$, $g=\gcd(n,N)$ and $d=N/g$. Then $\zeta^d\ne1$ and
\[
 \|(\zeta-T^n)^{-1}\|
 \le1+\frac{KMN}{\sqrt{g}\,|\zeta^d-1|}
 \le1+\frac{KMN}{\min_{e\mid N}|\zeta^e-1|}.
\]
In particular, if $N$ is odd, then for every $n\ge1$
\[
 \|(\Id+T^n)^{-1}\|
 \le1+\frac{KMN}{2\sqrt{\gcd(n,N)}}
 \le1+\frac{KMN}{2}.
\]
\end{proposition}
\begin{proof}
Since $d\mid N$ and $\zeta\notin\mu_N$, $\zeta^d\ne1$. Let $\mu_j$ be the $n$-th roots of $\zeta$; then $\mu_j\notin\mu_N$, for otherwise $\zeta=\mu_j^n\in\mu_N$. Use the vectors $y_j$ of the proof of Theorem~\ref{thm:rootsofzeta}, with $d_0:=\|(\zeta-T^n)x\|<1$ in place of $d$ there. The lower bound $\sum_j\|y_j\|^2\ge n^2(1-d_0)^2/M^2$ is unchanged. For the upper bound, $\|y_j\|\le Kd_0/\dist(\mu_j,\mu_N)$ and $\dist(\mu_j,\mu_N)^{-2}\le\sum_{\xi\in\mu_N}|\mu_j-\xi|^{-2}$, so by Lemma~\ref{lem:rootsum} (applied with the point $\xi\in\mu_N$, for which $\xi^n\ne\zeta$)
\[
 \sum_j\|y_j\|^2
 \le K^2d_0^2\sum_{\xi\in\mu_N}\sum_j|\mu_j-\xi|^{-2}
 =K^2d_0^2n^2\sum_{\xi\in\mu_N}|\zeta-\xi^n|^{-2}.
\]
The map $\xi\mapsto\xi^n$ sends $\mu_N$ onto $\mu_d$, with $g$ preimages per point, and Lemma~\ref{lem:rootsum} applied to the $d$-th roots of unity and the point $\zeta$ gives $\sum_{\omega\in\mu_d}|\zeta-\omega|^{-2}=d^2/|\zeta^d-1|^2$. Hence
\[
 \sum_{\xi\in\mu_N}|\zeta-\xi^n|^{-2}=\frac{gd^2}{|\zeta^d-1|^2}=\frac{N^2}{g\,|\zeta^d-1|^2}.
\]
Comparison of the two bounds gives $(1-d_0)/M\le Kd_0N/(\sqrt g|\zeta^d-1|)$, which proves the first estimate; the second follows since $g\ge1$ and $d\mid N$. Invertibility of $\zeta-T^n$ follows from $\sigma(T^n)\cap\T\subset\mu_N$ and $\zeta\notin\mu_N$. For $N$ odd and $\zeta=-1$ we have $-1\notin\mu_N$ and $|(-1)^e-1|=2$ for every divisor $e$ of $N$.
\end{proof}

\subsection{Displacement bounds}
The lower bound \eqref{eq:zetabound} for $\zeta=-1$ controls the distance from $T^n$ to $\Id$ through the parallelogram law.

\begin{corollary}\label{cor:hilbertdisp}
Let $T$ be a Ritt contraction on a nonzero Hilbert space and let
$C=\CR(T)\ge1$. Then
\[
   \inf_{n\ge1}\inf_{\|x\|=1}\|x+T^nx\|\ge\frac{2}{C+2}
\]
and
\begin{equation}\label{eq:hilbertdispsharper}
 \sup_{n\ge1}\|\Id-T^n\|\le D(C):=
 \begin{cases}
 \displaystyle\frac{2(C+1)}{C+2},&1\le C\le\sqrt2,\\[6pt]
 \displaystyle2\sqrt{1-C^{-2}},&C\ge\sqrt2.
 \end{cases}
\end{equation}
In particular, $D(C)<2$ for every finite $C$.
\end{corollary}

\begin{proof}
Fix $n\ge1$ and a unit vector $x$, and set
\[
 r=\|T^nx\|\in[0,1],\qquad d=\|x+T^nx\|.
\]
Apply the Fourier calculation in the proof of
Theorem~\ref{thm:rootsofzeta} with $\zeta=-1$. Since $T$ is a
contraction, $\|T^kx\|\ge r$ for $0\le k<n$, and hence that
calculation gives directly
\[
 n^2r^2\le\sum_{j=0}^{n-1}\|y_j\|^2
       \le\frac{C^2d^2n^2}{4}.
\]
Thus $d\ge2r/C$. The reverse triangle inequality gives
$d\ge1-r$, and therefore
\begin{equation}\label{eq:hilbert-rd}
 d\ge\max\{1-r,2r/C\}\ge\frac{2}{C+2}.
\end{equation}
This proves the first assertion without any restriction on $d$.
By the parallelogram law,
\begin{equation}\label{eq:hilbert-parallelogram-optimisation}
 \|x-T^nx\|^2
 =2+2r^2-d^2
 \le2+2r^2-\max\{(1-r)^2,4r^2/C^2\}.
\end{equation}
The two expressions inside the maximum are equal at
$r_0=C/(C+2)$. On $0\le r\le r_0$, the right-hand side of
\eqref{eq:hilbert-parallelogram-optimisation} is $(1+r)^2$, whose
maximum is
\[
 (1+r_0)^2=\frac{4(C+1)^2}{(C+2)^2}.
\]
On $r_0\le r\le1$, it is $2+(2-4/C^2)r^2$. If
$C\le\sqrt2$, this is nonincreasing and its maximum is the same
value at $r_0$. If $C\ge\sqrt2$, it is nondecreasing and its
maximum is $4-4/C^2$ at $r=1$; this also dominates the value at
$r_0$. Taking square roots and then the supremum over $x$ and $n$
proves \eqref{eq:hilbertdispsharper}. Finally, $C\ge1$ follows by
letting $\lambda\to\infty$ in the definition of $\CR(T)$.
\end{proof}

\subsection{\texorpdfstring{$L^p$}{Lp} spaces}
Let $(\Omega,\nu)$ be a $\sigma$-finite measure space and $1<p<\infty$. Write $\zR(p):=\sum_{k\ge1}k^{-p}$ for the Riemann zeta function, and for $\zeta\in\T\setminus\{1\}$ put
\[
   K_p(\zeta):=2^{1/p}\left[\left(\frac{\pi}{2|\zeta-1|}\right)^p+\frac{\zR(p)}{4^p}\right]^{1/p}.
\]

\begin{theorem}\label{thm:Lp}
Let $T$ be a Ritt contraction on $L^p(\Omega,\nu)$, $1<p<\infty$, with $C=\CR(T)$, and let $\zeta\in\T\setminus\{1\}$. Then
\[
   \sup_{n\ge1}\|(\zeta-T^n)^{-1}\|_{p\to p}\le1+CK_{p^*}(\zeta),\qquad p^*:=\min\{p,p'\},\ \tfrac1p+\tfrac1{p'}=1 .
\]
Consequently $\sup_n\|\Id-T^n\|_{p\to p}\le2\big(1-\delta_{L^p}(1/(1+CK_{p^*}(-1)))\big)<2$.
\end{theorem}

\begin{proof}
Assume first $1<p\le2$. Fix $n$, $x\in L^p$ with $\|x\|_p=1$, put $d=\|(\zeta-T^n)x\|_p$, and define $\mu_j$ and $y_j=\sum_{k<n}\mu_j^{-k}T^kx$ as in the proof of Theorem~\ref{thm:rootsofzeta}; then $\|y_j\|_p\le Cd/|\mu_j-1|$ by Proposition~\ref{prop:boundary}. For a.e.\ $t\in\Omega$ the vector $(y_j(t))_j$ is the discrete Fourier transform, up to unimodular factors, of $(T^kx(t))_k$, so
\[
   \sum_{j=0}^{n-1}|y_j(t)|^2=n\sum_{k=0}^{n-1}|T^kx(t)|^2\qquad\text{for a.e.\ }t .
\]
Since $p\le2$, pointwise $(\sum_j|y_j|^2)^{1/2}\le(\sum_j|y_j|^p)^{1/p}$, whence
\[
   \Big\|\Big(\sum_j|y_j|^2\Big)^{1/2}\Big\|_p\le\Big(\sum_j\|y_j\|_p^p\Big)^{1/p}\le Cd\Big(\sum_j|\mu_j-1|^{-p}\Big)^{1/p}.
\]
To bound the last sum write $\zeta=e^{i\beta}$, $0<\beta<2\pi$, so that $|\mu_j-1|=2\sin t_j$ with $t_j=(\beta+2\pi j)/(2n)\in(0,\pi)$, and use $\sin t\ge\frac2\pi\min\{t,\pi-t\}$. Partition the indices according to $t_j\le\pi/2$ or $t_j>\pi/2$. The numbers $\min\{t_j,\pi-t_j\}$ then lie in two truncated arithmetic progressions, possibly empty, of step $\pi/n$ and initial points $\beta/(2n)$ and $(2\pi-\beta)/(2n)$. Both initial points are at least $\delta:=|\zeta-1|/(2n)$ because $|\zeta-1|=2\sin(\beta/2)\le\min\{\beta,2\pi-\beta\}$. In each nonempty progression, the first term is at least $\delta$ and the $(i+1)$-st is at least $i\pi/n$ for $i\ge1$. Consequently,
\[
   \sum_j|\mu_j-1|^{-p}\le\Big(\frac\pi4\Big)^p\cdot2\Big(\delta^{-p}+\sum_{i\ge1}\Big(\frac{i\pi}{n}\Big)^{-p}\Big)
   =2n^p\Big(\Big(\frac{\pi}{2|\zeta-1|}\Big)^p+\frac{\zR(p)}{4^p}\Big),
\]
Taking $p$-th roots gives $(\sum_j|\mu_j-1|^{-p})^{1/p}\le nK_p(\zeta)$. On the other hand, H\"older's inequality $(\sum_{k<n}|a_k|^p)^{1/p}\le n^{1/p-1/2}(\sum_{k<n}|a_k|^2)^{1/2}$ applied pointwise gives
\[
   \Big\|\Big(\sum_k|T^kx|^2\Big)^{1/2}\Big\|_p\ge n^{1/2-1/p}\Big(\sum_k\|T^kx\|_p^p\Big)^{1/p}\ge n^{1/2}\min_{k<n}\|T^kx\|_p\ge n^{1/2}\|T^nx\|_p\ge n^{1/2}(1-d),
\]
using that $T$ is a contraction. Combining, $n(1-d)\le\|(\sum_j|y_j|^2)^{1/2}\|_p\le Cd\,nK_p(\zeta)$, i.e.\ $d\ge(1+CK_p(\zeta))^{-1}$, which gives the inverse bound since $\zeta\in\rho(T^n)$.

If $2<p<\infty$, the adjoint $T^*$ is a Ritt contraction on $L^{p'}$ with $\CR(T^*)=\CR(T)$ and $1<p'<2$, and $\|(\zeta-T^n)^{-1}\|_{p\to p}=\|(\bar\zeta-(T^*)^n)^{-1}\|_{p'\to p'}$; apply the first part to $T^*$ and $\bar\zeta$, noting that $K_{p'}(\bar\zeta)=K_{p'}(\zeta)$. The last assertion follows from Theorem~\ref{thm:UC}, $L^p$ being uniformly convex \cite{Clarkson,Hanner}.
\end{proof}

\subsection{Uniform Ritt constants for powers}
Theorem~\ref{thm:rootsofzeta} is a boundary estimate for the powers of $T$; Proposition~\ref{prop:boundary} transfers it to the exterior of the disc without loss and gives an explicit form of Theorem~\ref{thm:BGTpowers} on Hilbert space.

\begin{proposition}\label{prop:explicitpowers}
If $T$ is a Ritt operator on a Hilbert space, with $C=\CR(T)$ and $M=M(T)$, then
\[
 \sup_{m\ge1}\CR(T^m)\le CM+2 .
\]
In particular, every Ritt contraction on a Hilbert space satisfies $\sup_{m\ge1}\CR(T^m)\le\CR(T)+2$.
\end{proposition}
\begin{proof}
Fix $m\ge1$. The operator $S=T^m$ is power bounded with $\sigma(S)\cap\T\subset\{1\}$, and Theorem~\ref{thm:rootsofzeta} gives, for $\xi\in\T\setminus\{1\}$,
$|\xi-1|\|\Res(\xi,S)\|\le|\xi-1|+CM\le2+CM$. Proposition~\ref{prop:boundary} yields $\CR(S)\le2+CM$.
\end{proof}

\begin{remark}[scalar restrictions on uniform power estimates]\label{rem:powerconstant-obstructions}
For a scalar $a\in\D$, the boundary characterisation of the Ritt
constant, or the image of the disc under a M\"obius transformation,
gives
\begin{equation}\label{eq:scalarritt-powersection}
 \CR(a\Id)=\frac{2|1-a|}{1-|a|^2}.
\end{equation}
Indeed, after setting $z=\lambda^{-1}$, the relevant function is
$(1-z)/(1-az)$; its image of $\D$ is a disc whose centre and radius
both have modulus $|1-a|/(1-|a|^2)$.

Any universal additive estimate
\[
 \CR(T^m)\le\CR(T)+a_0\qquad(m\ge1)
\]
for Hilbert-space Ritt contractions must have $a_0\ge1$.
For $T=t\Id$ with $0<t<1$, formula
\eqref{eq:scalarritt-powersection} gives
\[
 \CR(T)=\frac2{1+t},\qquad
 \sup_{m\ge1}\CR(T^m)=2.
\]
Letting $t\uparrow1$ proves the assertion. We also notice that another possible bound,
$\sup_m\CR(T^m)\le\max\{\CR(T),2\}$, fails already for a
scalar contraction. Namely, for $t=(3+i)/4$,
\[
 \CR(t\Id)=\frac{4\sqrt2}{3}<2,
 \qquad
 \CR(t^2\Id)=\frac{80}{39}>2.
\]
\end{remark}

\section{Norm gaps for operator-valued functions}\label{sec:matrix}

For Ritt operators and symbols $w\in A^{1,1}(\D)$ with
$|w(1)|<\|w\|_\infty$, Borichev, Gomilko and Tomilov obtain
exponential estimates $\sup_n\|w(T^n)^N\|\le C\eta^N$ with
$\eta<\|w\|_\infty$. They also show that the ordinary gap
$\sup_n\|w(T^n)\|<\|w\|_\infty$ can fail without contractivity
\cite[Theorem~7.15 and Example~7.16]{BGT}. For Hilbert-space
contractions, ordinary norm gaps extend to disc-algebra functions
with coefficients in the operators between arbitrary Hilbert spaces.
We first give a computable estimate for finite matrix polynomials by
unitary dilation and Fej\'er--Riesz factorisation. We then prove the
operator-valued result and its converse.

\subsection{Explicit polynomial gaps}
For an $r\times d$ matrix polynomial $W(z)=\sum_{j=0}^{l}W_jz^j$ and $S\in\B(H)$, put
$W(S)=\sum_{j=0}^{l}W_j\otimes S^j:H^d\to H^r$ and
$\|W\|_\infty=\max_{|z|=1}\|W(z)\|$. Evaluation $W\mapsto W(S)$ is an algebra homomorphism from matrix polynomials (of compatible sizes) to operators, because $p\mapsto p(S)$ is one for scalar polynomials.

\begin{theorem}[matrix polynomial norm gaps]\label{thm:polydefect}
Let $T$ be a Ritt contraction on a Hilbert space and put $C=\CR(T)$.
Let $W$ be an $r\times d$ matrix polynomial with
\[
 m:=\|W\|_\infty>\|W(1)\|.
\]
Then $\sup_{n\ge1}\|W(T^n)\|<m$.
More precisely, let $A$ be a $d\times d$ outer matrix polynomial with
\begin{equation}\label{eq:FRgap}
 m^2\Id_d-W(\zeta)^*W(\zeta)=A(\zeta)^*A(\zeta)
 \qquad(\zeta\in\T),
\end{equation}
and write
\[
 \det A(z)=c\prod_{j=1}^{h}(\alpha_j-z),\qquad
 H_A=\max_{\T}\|\operatorname{adj}A(z)\|,
\]
where zeros are repeated with multiplicity and an empty product is $1$. Then $c\ne0$, $|\alpha_j|\ge1$ and $\alpha_j\ne1$ for all $j$, and with
\[
 b_C(\alpha)=
 \begin{cases}
 (|\alpha|-1)^{-1},&|\alpha|>1,\\
 1+C/|\alpha-1|,&|\alpha|=1,\ \alpha\ne1,
 \end{cases}
 \qquad
 L_A=\frac{H_A}{|c|}\prod_{j=1}^{h}b_C(\alpha_j),
\]
one has
\begin{equation}\label{eq:matrixgap}
 \sup_{n\ge1}\|W(T^n)\|
 \le\sqrt{m^2-L_A^{-2}}<m.
\end{equation}
The constant depends on $W$ and $C$ only, and not on $n$ or on the dimension of $H$.
\end{theorem}

\begin{proof}
\emph{The outer factor.} The function $\zeta\mapsto m^2\Id_d-W(\zeta)^*W(\zeta)$ is a matrix trigonometric polynomial which is positive semidefinite on $\T$. By the matrix Fej\'er--Riesz theorem \cite[Theorem~2.1]{DR} there is a $d\times d$ matrix polynomial $A$ satisfying \eqref{eq:FRgap} which is outer, in the sense that $AH^2(\C^d)$ is dense in $H^2(\mathcal M)$ for a subspace $\mathcal M$ of $\C^d$. Since $A(z)\C^d\subset\mathcal M$ for every $z\in\ol\D$ (evaluate $Af$ at $z$ for constant $f$, and use continuity up to the boundary) and $A(1)^*A(1)=m^2\Id_d-W(1)^*W(1)$ is positive definite, $A(1)$ is invertible and $\mathcal M=\C^d$. For each $z\in\D$, evaluation at $z$ is a continuous map from $H^2(\C^d)$ onto $\C^d$. Since $AH^2(\C^d)$ is dense in $H^2(\C^d)$, the range $A(z)\C^d$ is dense in $\C^d$ and hence equals $\C^d$. Thus $A(z)$ is invertible for every $z\in\D$, so $\det A$ has no zeros in $\D$; it does not vanish at $1$ either. This proves $c\ne0$, $|\alpha_j|\ge1$ and $\alpha_j\ne1$.

\emph{The dilation inequality.} Fix $n$ and put $S=T^n$, a contraction. Let $U$ be a unitary dilation of $S$ on a Hilbert space $\mathcal K\supset H$, so that $S^j=P_HU^j|_H$ for all $j\ge0$ \cite[Chapter~I]{SzNagyFoias}. Then $W(S)$ is obtained by restricting $W(U)$ to $H^d$ and projecting its range onto $H^r$, and similarly for $A(S)$: $W(S)x=(\Id_r\otimes P_H)W(U)x$ for $x\in H^d$. By the spectral theorem for $U$, evaluation of matrix trigonometric polynomials at $U$ is a $*$-homomorphism, so \eqref{eq:FRgap} gives $W(U)^*W(U)+A(U)^*A(U)=m^2\Id$. Hence, for $x\in H^d$,
\[
 \|W(S)x\|^2+\|A(S)x\|^2
 \le\|W(U)x\|^2+\|A(U)x\|^2
 =m^2\|x\|^2.
\]

\emph{Invertibility of $A(S)$.} For every zero $\alpha_j$ of $\det A$, the operator $\alpha_j-S$ is invertible: the Neumann series applies when $|\alpha_j|>1$, and Theorem~\ref{thm:rootsofzeta} (with $M=1$) applies when $|\alpha_j|=1$, $\alpha_j\ne1$. In both cases
$\|(\alpha_j-S)^{-1}\|\le b_C(\alpha_j)$, uniformly in $n$. Hence $(\det A)(S)=c\prod_j(\alpha_j-S)$ is invertible with
\[
 \|((\det A)(S))^{-1}\|\le |c|^{-1}\prod_jb_C(\alpha_j).
\]
The adjugate identity $\operatorname{adj}(A)A=A\operatorname{adj}(A)=(\det A)\Id_d$, evaluated at $S$, shows that $A(S)$ is invertible with
\[
 A(S)^{-1}=\operatorname{adj}(A)(S)
 \bigl(\Id_d\otimes((\det A)(S))^{-1}\bigr),
\]
the scalar rational factor commuting with every entry of $\operatorname{adj}(A)(S)$. The von Neumann inequality for matrix polynomials (which follows from the dilation, since $\operatorname{adj}(A)(S)$ is the compression of $\operatorname{adj}(A)(U)$) gives
$\|\operatorname{adj}(A)(S)\|\le H_A$. Hence
$\|A(S)^{-1}\|\le L_A$ and $\|A(S)x\|\ge L_A^{-1}\|x\|$.
Substitution in the dilation inequality proves
\eqref{eq:matrixgap}; in particular its radicand is nonnegative.
\end{proof}

\begin{corollary}[scalar symbols]\label{cor:scalarsymbols}
If $w$ is a scalar polynomial with $|w(1)|<\|w\|_\infty$, then
$\sup_n\|w(T^n)\|<\|w\|_\infty$ for every Ritt contraction $T$ on a Hilbert space.
If $a(z)=c\prod_j(\alpha_j-z)$ is the scalar outer factor of
$\|w\|_\infty^2-|w|^2$, then
\[
   \sup_n\|w(T^n)\|^2\le\|w\|_\infty^2-|c|^2\prod_j b_C(\alpha_j)^{-2}.
\]
\end{corollary}
As an example, consider $w(z)=1-z$. One has $4-|1-z|^2=|1+z|^2$ on $\T$, so $a(z)=1+z$, $h=1$, $\alpha_1=-1$, $|c|=1$, $b_C(-1)=1+C/2$, and the bound is $\sup_n\|\Id-T^n\|^2\le4-4/(C+2)^2$: this also follows from the first estimate in
Corollary~\ref{cor:hilbertdisp} and the parallelogram law.

\subsection{Operator-valued disc-algebra functions}

Let $E$ and $F$ be complex Hilbert spaces. Write
$A(\D;\B(E,F))$ for the Banach space of functions
$W:\ol\D\to\B(E,F)$ that are continuous in operator norm on
$\ol\D$ and holomorphic on $\D$, with norm
\[
 \|W\|_\infty=\max_{\zeta\in\T}\|W(\zeta)\|.
\]
Such functions are uniform limits of operator-valued polynomials:
first replace $W(z)$ by $W(rz)$, $r<1$, and then truncate its
norm-convergent Taylor series. For a contraction $S\in\B(H)$ and
a polynomial $W(z)=\sum_jW_jz^j$, define
\[
 W(S)=\sum_jW_j\otimes S^j:E\otimes H\longrightarrow F\otimes H,
\]
where all tensor products of Hilbert spaces are completed Hilbert
tensor products. A unitary dilation of $S$ gives
$\|W(S)\|\le\|W\|_\infty$, so evaluation extends continuously
to $A(\D;\B(E,F))$. This definition agrees with the matrix
functional calculus when $E$ and $F$ are finite dimensional.

For clarity, this use of unitary dilation does not require separable
coefficient spaces. If $U$ is unitary on $\mathcal K$ with spectral
measure $P_U$, the tensor projections
$\Id_E\otimes P_U(B)$ form a spectral measure on
$E\otimes\mathcal K$. The functional calculus for continuous
operator-valued functions can be defined by uniform approximation
by trigonometric polynomials. For $x\in E\otimes\mathcal K$, put
\[
 \nu_x(B)=\langle(\Id_E\otimes P_U(B))x,x\rangle.
\]
The spectral theorem then gives
\begin{equation}\label{eq:operator-valued-spectral-bound}
 \|W(U)x\|^2\le\int_\T\|W(\zeta)\|^2\,d\nu_x(\zeta).
\end{equation}
One can verify this inequality first for step functions on disjoint
Borel sets, whose spectral projections have orthogonal ranges, and
then pass to the uniform limit. This also proves the asserted
contractivity after compression to $E\otimes H$ and
$F\otimes H$.

\begin{theorem}[operator-valued disc-algebra norm gaps]\label{thm:disc-algebra-gap}
Let $S$ be a contraction on a Hilbert space such that
$\sigma(S)\cap\T\subset\{1\}$. Let $E,F$ be Hilbert spaces, and
let $W\in A(\D;\B(E,F))$ satisfy
\[
 m:=\|W\|_\infty>\|W(1)\|.
\]
Then $\|W(S)\|<m$.
\end{theorem}

\begin{proof}
Let
\[
 \mathcal F=\{\zeta\in\T:\|W(\zeta)\|=m\}.
\]
Continuity in operator norm implies that $\mathcal F$ is a nonempty
compact subset of $\T\setminus\{1\}$, disjoint from $\sigma(S)$.
We first construct a scalar polynomial $p$ such that
\begin{equation}\label{eq:Runge-separation}
 \|p(S)-\Id\|<\tfrac14,
 \qquad \max_{\mathcal F}|p|<\tfrac14.
\end{equation}

Choose
$0<\delta<\min\{1/4,\dist(1,\mathcal F)/4\}$.
The compact set
$\sigma(S)\setminus\{z:|z-1|<\delta/2\}$ lies in $\D$, so there
is $r\in(1-\delta,1)$ such that
\[
 K=\{z:|z|\le r\}\cup\{z:|z-1|\le\delta\}
\]
contains $\sigma(S)$ in its interior. The sets $K$ and $\mathcal F$
are disjoint, and $\C\setminus(K\cup\mathcal F)$ is connected.
Here is a direct verification. For $r<s\le1$, the part of
$|z|=s$ outside $\{z:|z-1|\le\delta\}$ is a connected arc
containing $se^{2i\delta}$, since $\sin(2\delta)>\delta$.
Every point of the complement with modulus less than $1$ can
therefore be joined along its circle to the ray of angle $2\delta$,
and then along that ray to infinity. The ray avoids $K$ once
$|z|>r$, and its intersection with $\T$ avoids $\mathcal F$, since
\[
 |e^{2i\delta}-1|\le2\delta<\dist(1,\mathcal F).
\]
A point of the complement on or outside $\T$ can be joined radially
to infinity: its distance from $1$ is nondecreasing along the
outward radial segment, and $\mathcal F$ lies on $\T$.

The function equal to $1$ near $K$ and to $0$ near $\mathcal F$
is holomorphic on a neighbourhood of $K\cup\mathcal F$.
Runge's theorem supplies polynomials $p_j$ converging uniformly to
this function on $K\cup\mathcal F$. To obtain operator convergence,
let $\gamma=\partial K$, with positive orientation. The two discs
overlap, so $\gamma$ is a piecewise smooth Jordan curve enclosing
$\sigma(S)$ and lying in $\rho(S)$. The holomorphic functional
calculus gives
\[
 p_j(S)-\Id
 =\frac{1}{2\pi i}\int_\gamma
       (p_j(z)-1)(z-S)^{-1}\,dz.
\]
Thus $\|p_j(S)-\Id\|\to0$. Since
$\max_{\mathcal F}|p_j|\to0$, some $p_j$ satisfies
\eqref{eq:Runge-separation}.

Suppose now that $\|W(S)\|=m$. Choose unit vectors
$x_k\in E\otimes H$ such that $\|W(S)x_k\|\to m$, and let $U$
be a unitary dilation of $S$ on $\mathcal K\supset H$. The scalar
measures
\[
 \nu_k(B)=\langle(\Id_E\otimes P_U(B))x_k,x_k\rangle
\]
are probability measures on $\T$. Compression and
\eqref{eq:operator-valued-spectral-bound} give
\[
 \|W(S)x_k\|^2
 \le\|W(U)x_k\|^2
 \le\int_\T\|W(\zeta)\|^2\,d\nu_k(\zeta)
 \le m^2.
\]
For every relatively open neighbourhood $V$ of $\mathcal F$ in
$\T$, the continuous function
$m^2-\|W(\zeta)\|^2$ has a positive minimum on $\T\setminus V$.
It follows that $\nu_k(\T\setminus V)\to0$.

Apply the dilation inequality to the scalar polynomial $p$ in
\eqref{eq:Runge-separation}, amplified by $\Id_E$. The preceding
concentration property and continuity of $p$ imply
\[
 \limsup_k\|(\Id_E\otimes p(S))x_k\|^2
 \le\limsup_k\int_\T|p(\zeta)|^2\,d\nu_k(\zeta)
 \le\max_{\mathcal F}|p|^2<\tfrac1{16}.
\]
On the other hand,
$\|\Id_E\otimes(p(S)-\Id)\|<1/4$ gives
$\|(\Id_E\otimes p(S))x_k\|>3/4$ for every $k$.
This contradiction proves the theorem.
\end{proof}

\begin{corollary}[uniform disc-algebra gaps for powers of Ritt operators]\label{cor:disc-algebra-powers}
Let $T$ be a Ritt contraction on a Hilbert space. For arbitrary
Hilbert spaces $E,F$ and every $W\in A(\D;\B(E,F))$ with
$\|W(1)\|<\|W\|_\infty$,
\[
 \sup_{n\ge1}\|W(T^n)\|<\|W\|_\infty.
\]
For each fixed $W$ and finite $C_0$, the gap can be chosen uniformly
over all Ritt contractions $T$ on Hilbert spaces satisfying
$\CR(T)\le C_0$.
\end{corollary}

\begin{proof}
On $\mathcal H=\bigoplus_{n\ge1}H$, set
$S=\bigoplus_{n\ge1}T^n$. This is a contraction. By
Proposition~\ref{prop:explicitpowers},
\[
 \sup_{|\lambda|>1}|\lambda-1|\|(\lambda-S)^{-1}\|
 =\sup_{n\ge1}\CR(T^n)\le\CR(T)+2.
\]
Hence $S$ is Ritt and $\sigma(S)\cap\T\subset\{1\}$.
Theorem~\ref{thm:disc-algebra-gap} applies to $S$. Hilbert tensor
products distribute over Hilbert direct sums, and the functional
calculus respects these identifications, first for polynomials and
then by uniform approximation. Consequently
\[
 \|W(S)\|=\sup_{n\ge1}\|W(T^n)\|<\|W\|_\infty.
\]

For the final assertion, suppose that a sequence of Ritt contractions
$T_j\in\B(H_j)$ satisfies $\CR(T_j)\le C_0$ and
$\sup_n\|W(T_j^n)\|\to\|W\|_\infty$.
The contraction
$S=\bigoplus_{j,n}T_j^n$ is Ritt, with $\CR(S)\le C_0+2$,
whereas $\|W(S)\|=\|W\|_\infty$. This contradicts
Theorem~\ref{thm:disc-algebra-gap}.
\end{proof}

\begin{theorem}[characterisation by a fixed operator-valued symbol]\label{thm:operator-symbol-characterisation}
Let $T$ be a contraction on a Hilbert space with
$\sigma(T)\cap\T\subset\{1\}$. Fix Hilbert spaces $E,F$ and a
function $W\in A(\D;\B(E,F))$ satisfying
$\|W(1)\|<\|W\|_\infty$. Then
\[
 T\text{ is Ritt}
 \quad\Longleftrightarrow\quad
 \sup_{n\ge1}\|W(T^n)\|<\|W\|_\infty.
\]
\end{theorem}

\begin{proof}
The forward implication is Corollary~\ref{cor:disc-algebra-powers}.
For the converse, put
$q=\sup_{n\ge1}\|W(T^n)\|<m:=\|W\|_\infty$.
Choose $\zeta_0\in\T$ with $\|W(\zeta_0)\|=m$;
then $\zeta_0\ne1$. Choose $\varepsilon>0$ and unit vectors
$e\in E$, $f\in F$ such that
\[
 |\langle W(\zeta_0)e,f\rangle|>q+3\varepsilon.
\]
This uses only the definition of the operator norm; no norm
attainment by $W(\zeta_0)$ is needed. Define the scalar function
$w(z)=\langle W(z)e,f\rangle\in A(\D)$. The tensor compression defining
$w(T^n)$ shows that
\[
 \|w(T^n)\|\le\|W(T^n)\|\le q\qquad(n\ge1).
\]
Choose a scalar polynomial $p$ with
$\|p-w\|_\infty<\varepsilon$. Since every $T^n$ is a
contraction, the disc-algebra functional calculus gives
\[
 \sup_{n\ge1}\|p(T^n)\|\le q+\varepsilon,
 \qquad |p(\zeta_0)|>q+2\varepsilon.
\]
By continuity, there is a nondegenerate closed arc
$\arc\subset\T\setminus\{1\}$ containing $\zeta_0$ such that
$m_{\arc}(p)>q+\varepsilon$. The full sequence $\N$ has the
sampling property for this arc, by Lemma~\ref{lem:allpowers}.
Corollary~\ref{cor:defect} applied to $p$ therefore proves that
$T$ is Ritt.
\end{proof}

The uniformity argument for general operator-valued disc-algebra
functions is nonconstructive. For finite matrix polynomials,
Theorem~\ref{thm:polydefect} provides the computable bound
\eqref{eq:matrixgap} through the outer Fej\'er--Riesz factor.

\begin{example}[uniform convexity does not imply polynomial norm gaps]\label{ex:UCpolynomial}
The Hilbert-space hypothesis cannot be replaced by uniform convexity,
even for scalar polynomials and nilpotent contractions. On the uniformly
convex space $X=\ell^4_3$, let
\[
 T(x_1,x_2,x_3)=(0,x_1,x_2),\qquad w(z)=1+z-z^2.
\]
Then $\|T\|=1$ and $T^3=0$, so $T$ is Ritt. Moreover,
\[
 |w(1)|=1,\qquad
 |w(e^{it})|^2=3-2\cos(2t),\qquad
 \|w\|_\infty=\sqrt5.
\]
For $x=(-1,1,1)$ one has $w(T)x=(-1,0,3)$, and consequently
\[
 \|w(T)\|\ge\frac{\|w(T)x\|_4}{\|x\|_4}
 =\left(\frac{82}{3}\right)^{1/4}>\sqrt5.
\]
Thus even the non-strict inequality $\|w(T)\|\le\|w\|_\infty$ fails.
Extensions to $L^p$ contractions for $p\ne2$ would require additional
functional-calculus hypotheses or a different comparison norm.
\end{example}

\subsection*{Acknowledgement.}
I first heard a presentation of \cite{BGT}, and realised that I could answer some of the questions raised there, during Yuri Tomilov's talk at the conference ``Recent Trends in Operator Theory and Function Theory'', Lille, 1--5 June 2026. I am grateful to Yuri for his stimulating lecture and for the discussions that followed it.

I am also grateful to Sophie Grivaux for many conversations about Jamison sequences in the course of our collaboration on \cite{BDG,BadeaGrivaux07,BadeaGrivaux17,BGSurvey} and during Sophie's visit to Reading.

This work was supported by the CDP C2EMPI and its institutional partners through the project R-CDP-24-004-C2EMPI, by the EU COST network OC-2025-1-28418, and by the Heilbronn Institute for Mathematical Research.

\end{document}